\pdfoutput=1
\RequirePackage{ifpdf}
\ifpdf 
\documentclass[pdftex]{sigma}
\else
\documentclass{sigma}
\fi

\numberwithin{equation}{section}
\usepackage{pgfplots,multirow,diagbox,algorithm,mathtools}
\usepackage{algpseudocode}

\newtheorem{Theorem}{Theorem}[section]
\newtheorem*{Theorem*}{Theorem}

\newtheorem{Claim}[Theorem]{Claim}
 { \theoremstyle{definition}

\newtheorem*{Notation}{Notation}

\newtheorem{Remark}[Theorem]{Remark}

}
\usepackage{subfig}
\usepackage[labelsep=period,labelfont=bf,font=small]{caption}
\newcommand{\chebpts}{K}
\pgfplotsset{compat=1.18}

\begin{document}

\allowdisplaybreaks

\renewcommand{\thefootnote}{}

\newcommand{\arXivNumber}{2304.04951}

\renewcommand{\PaperNumber}{005}

\FirstPageHeading

\ShortArticleName{Computing the Tracy--Widom Distribution for Arbitrary $\beta>0$}

\ArticleName{Computing the Tracy--Widom Distribution\\ for Arbitrary $\boldsymbol{\beta>0}$\footnote{This paper is a~contribution to the Special Issue on Evolution Equations, Exactly Solvable Models and Random Matrices in honor of Alexander Its' 70th birthday. The~full collection is available at \href{https://www.emis.de/journals/SIGMA/Its.html}{https://www.emis.de/journals/SIGMA/Its.html}}}

\Author{Thomas TROGDON and Yiting ZHANG}

\AuthorNameForHeading{T.~Trogdon and Y.~Zhang}

\Address{Department of Applied Mathematics, University of Washington, Seattle, Washington, USA}
\Email{\href{mailto:trogdon@uw.edu}{trogdon@uw.edu}, \href{mailto:yitinz91@uw.edu}{yitinz91@uw.edu}}

\ArticleDates{Received April 19, 2023, in final form January 03, 2024; Published online January 13, 2024}

\Abstract{We compute the Tracy--Widom distribution describing the asymptotic distribution of the largest eigenvalue of a large random matrix by solving a boundary-value problem posed by Bloemendal in his Ph.D.~Thesis (2011). The distribution is computed in two ways. The first method is a second-order finite-difference method and the second is a highly accurate Fourier spectral method. Since $\beta$ is simply a parameter in the boundary-value problem, any $\beta> 0$ can be used, in principle. The limiting distribution of the $n$th largest eigenvalue can also be computed. Our methods are available in the \textsc{Julia} package \texttt{TracyWidomBeta.jl}.}


\Classification{65M06; 60B20; 60H25}

\renewcommand{\thefootnote}{\arabic{footnote}}
\setcounter{footnote}{0}

\section{Introduction}

Tracy and Widom~\cite{TRACY1993115,cmp/1104254495,cmp/1104286442} introduced the Tracy--Widom distribution that gives the limiting distribution of the rescaled largest eigenvalue of a random matrix taken from an appropriate symmetry class. More precisely, the largest eigenvalue $\lambda_{\max}$ satisfies the following fundamental limit
\[
\lim_{n\to \infty}\mathbb {P}\big(n^{1/6}(\lambda_{\max}(A_{n})-2\sqrt{n})\leq x\big)= F_{\beta}(x) \qquad \text{for all} \ x\in \mathbb {R},
\]
where $F_{\beta}$ is the Tracy--Widom distribution and $\beta=1$, $2$, $4$ if $A_{n}\sim$ Gaussian orthogonal ensemble, Gaussian unitary ensemble, Gaussian symplectic ensemble, respectively.

For an $n\times n$ Gaussian ensemble $A_n$ with ordered eigenvalues $\{\lambda_{j}\}_{j=1}^{n}$, the joint probability density function (jpdf) for its eigenvalues is given by
\begin{equation}
 \rho\left(\lambda_{1},\lambda_{2},\dots,\lambda_{n}\right)=Z^{-1}_{n,\beta}\prod_{i}{\rm e}^{-\beta \lambda^{2}_{i}/4}\prod_{i<j}\left|\lambda_{j}-\lambda_{i}\right|^{\beta},\qquad \lambda_{1}\geq \lambda_{2}\geq \cdots\geq \lambda_{n},
 \label{eqn:density}
\end{equation}
where $Z_{n,\beta}$ is the partition function. For $\beta =1,2,4$,~\eqref{eqn:density} is solvable: all correlation functions in finite dimensions can be explicitly calculated using Hermite polynomials. This provides definitive local limit theorems and establishes clear limits for the random points~\cite{Meh2004}. Importantly, \eqref{eqn:density}~also describes a one-dimensional Coulomb gas at inverse temperature $\beta$ for any $\beta > 0$. If $\beta\neq 1,2,4,6$, there are no known explicit formulae which appear amenable to asymptotic analysis (see, for example,~\cite{Grava2016,Li2018OnTO,Rumanov2016} for $\beta = 6$). One, in general, must resort to numerical computations, see~\cite{Bornemann2009OnTN} for $\beta=1,2,4$.

Dumitriu and Edelman~\cite{article} established that a family of symmetric tridiagonal (Jacobi) matrix models have~\eqref{eqn:density} as their eigenvalue density for any $\beta> 0$. More specifically, the matrix given~by%
\begin{equation}
 H_{n}^{\beta}=\frac {1}{\sqrt{\beta}}\begin{bmatrix}
 g & \chi_{(n-1)\beta} & & & \\
 \chi_{(n-1)\beta} & g & \chi_{(n-2)\beta} & & \\
 & \chi_{(n-2)\beta} & g & \ddots & \\
 & & \ddots & \ddots & \chi_{\beta}\\
 & & & \chi_{\beta} & g
 \end{bmatrix},
 \label{eqn:matrix}
\end{equation}
has~\eqref{eqn:density} as the jpdf of its eigenvalues. Here the entries are independent random variables, up to symmetry, $g\sim N(0, 2)$ and $\chi_{k}\sim$ Chi($k$). We call $H^{\beta}_{n}$ the $\beta$-Hermite ensemble.

Sutton and Edelman~\cite{article2, article3} then presented an argument\footnote{A version of this argument was first presented by Edelman at the SIAM Conference on Applied Linear Algebra held at the College of William $\&$ Mary in 2003.} describing how the spectrum of the rescaled operator
\[
 \tilde{H}_{n}^{\beta}=n^{1/6}\big(2\sqrt{n}I-H_{n}^{\beta}\big),
\]
where $I$ is the $n\times n$ identity matrix, should be described by the spectrum of the stochastic Airy operator
\[
 \mathcal {H}_{\beta}=-\frac {{\rm d}^2}{{\rm d}x^2}+x+\frac {2}{\sqrt{\beta}}b'_{x},
\]
as $n\to \infty$, where $b'_{x}$ is standard Gaussian white noise.

As a result, the following eigenvalue problem was considered in~\cite{10.2307/23072161}
\[
 \mathcal {H}_{\beta}f=\Lambda f \qquad \text{on} \ L^{2}(\mathbb {R}_{+})
\]
with a Dirichlet boundary condition $f(0)=0$. Ram\'irez, Rider, and Vir\'ag~\cite{10.2307/23072161} proved the following theorem.
\begin{Theorem}
 With probability one, for each $k\geq 0$, the set of eigenvalues of $\mathcal {H}_{\beta}$ has a well-defined $(k+1)$st lowest element $\Lambda_{k}$. Moreover, let $\lambda_{1}\geq \lambda_{2}\geq \cdots$ denote the eigenvalues of ${H}^{\beta}_{n}$. Then the vector
\begin{align*}
 \big(n^{1/6}(2\sqrt{n}-\lambda_{l})\big)_{l=1,\dots,k}
\end{align*}
converges in distribution to $(\Lambda_{0},\Lambda_{1},\dots,\Lambda_{k-1})$ as $n\to \infty$.
\end{Theorem}
Ram\'irez, Rider, and Vir\'ag showed that the distribution of $-\Lambda_{0}$ is a consistent definition of Tracy--Widom($\cdot$) for general $\beta>0$.
Based on~\cite{10.2307/23072161, article6}, Bloemendal and Vira\'g~\cite{Bloemendal2011FiniteRP,article7} considered a generalized eigenvalue problem
\[
 \mathcal {H}_{\beta}f=\Lambda f \qquad \text{on} \ L^{2} (\mathbb {R}_{+} )
\]
with boundary condition $f'(0)=\omega f(0)$, where $\omega\in \mathbb {R}$ represents a scaling parameter. Let~$\mathcal {H}_{\beta,\omega}$ denote $\mathcal {H}_{\beta}$ together with this boundary condition. According to~\cite{10.2307/23072161}, the distribution $F_{\beta,\omega}$ of~$-\Lambda_{0}$ in the case $\omega=+\infty$ thus coincides with Tracy--Widom for general $\beta>0$.

Bloemendal and Vir\'ag~\cite{article7} further showed that $F=F_{\beta,\omega}$ can be characterized using the solution of a boundary value problem, as we now discuss in the following section. We also point out that Bloemendal~\cite{Bloemendal2011FiniteRP} provided \textsc{Mathematica} code to approximate $F$. The goal of the current work is to expand and improve upon this scheme.

This paper is laid out as follows. In Section~\ref{s:description}, we outline our two algorithms to compute the Tracy--Widom distribution. In Section~\ref{s:valid}, we validate and compare our methods. In Section~\ref{s:results}, we present a number of additional numerical results. We also include two appendices to discuss some nuances in the numerical algorithms (Appendix~\ref{a:2}) and to discuss the large $\beta$ limit (Appendix~\ref{a:1}). Code to produce all figures in this paper can be found here~\cite{tracywidombeta}.

\section{Algorithm description}\label{s:description}

The Tracy--Widom distribution function $F_{\beta}$ can be characterized as follows~\cite{Bloemendal2011FiniteRP}. Consider
\begin{equation}
\frac {\partial F}{\partial x}+\frac {2}{\beta}\frac {\partial^{2} F}{\partial \omega^{2}}+\big(x-\omega^2\big)\frac {\partial F}{\partial \omega}=0\qquad \text{for} \ (x,\omega)\in \mathbb {R}^{2},
\label{eqn:bv}
\end{equation}
with boundary conditions given by
\begin{align*}
&F(x,\omega)\to 1\qquad \text{as} \ x,\omega\to \infty \ \text{together,}\\
&F(x,\omega)\to 0\qquad \text{as} \ \omega\to -\infty \ \text{with}x \ \text{bounded above.}
\end{align*}
Then~\cite[Theorem 1.7]{article7}
\begin{equation}
 F_{\beta}(x)=\lim_{\omega\to \infty}F(x,\omega).
 \label{eqn:9}
\end{equation}
Using $\omega=-\cot \theta$, we rewrite~\eqref{eqn:bv} as
\begin{equation}
\frac {\partial H}{\partial x}+\left(\frac {2}{\beta}\sin^{4} \theta\right)\frac {\partial^{2} H}{\partial \theta^{2}}+\left[\left(x+\frac {2}{\beta}\sin 2\theta\right)\sin^{2}\theta-\cos^{2}\theta\right]\frac {\partial H}{\partial \theta}=0,
\label{eqn:bv1}
\end{equation}
with boundary condition
\begin{align*}
H(x,0)=0.
\end{align*}
To address the boundary condition as $x\to\infty$ and $\theta\to \pi$, we truncate the domain at a finite value, denoted as $x=x_{0}>0$. We then employ the Gaussian asymptotics established by Bloemendal~\cite[Theorem 4.1.1]{Bloemendal2011FiniteRP} to obtain the following approximate asymptotic initial con\-di\-tion~\mbox{\cite[p.~103]{Bloemendal2011FiniteRP}}:
\begin{align}
H(x_{0},\theta)=\begin{cases}
\Phi\left(\dfrac {x_{0}-\cot^{2}\theta}{\sqrt{(4/\beta)\cot \theta}}\right), &0\leq \theta\leq \pi/2,\\
1, &\pi/2\leq \theta.
\end{cases}\label{initialc}
\end{align}
Here $\Phi$ denotes the standard normal distribution function. Note that $H(x_{0},\theta)$ is continuous in both $x_{0}>0$ and $\theta\geq 0$. One then approximates the undeformed Tracy--Widom distribution~$F_{\beta}(x)$ at $\theta=\pi$ by $\omega=-\cot \theta$ and~\eqref{eqn:9}, i.e., $F_{\beta}(x)\approx H(x,\pi)$.

\subsection{Finite-difference discretization}
One way to solve this boundary-value problem is by discretizing it using finite differences, see, for example,~\cite{10.5555/1355322}. For $0\leq n\leq N$, $0\leq m\leq M$, define
\[
x_{n}=x_{0}+n\Delta x,\qquad \theta_m=mh.
\]
Given that $\beta>0$, one integrates equation~\eqref{eqn:bv1} backward in ``time'' with respect to the time-like variable $x$ to guarantee its well-posedness: $\Delta x < 0$. Denote the approximation of~$H(x,\theta_{m})$ by~$H_{m}(x)$. We then replace partial derivatives of $H$ with respect to $\theta$ by centered differences, with grid spacing $h$. The method of lines formulation reads
\begin{equation}
 \frac {\partial \pmb{H}_{M}}{\partial x}=-\bigg(\frac {2\sin^{4} \theta}{\beta h^2}\bigg)\check T\pmb{H}_{M}+\frac {\big[\big(x+\frac {2}{\beta}\sin 2\theta\big)\sin^{2}\theta-\cos^{2}\theta\big]}{2h}\check U\pmb{H}_{M},
 \label{eqn:fd}
\end{equation}
where $\pmb{H}_{M}(x)= [H_{1}(x),H_{2}(x),\dots,H_{M}(x) ]^{\textsf{T}}$, $Mh=\theta_{M}$, and
\begin{equation}
 \check T=\begin{bmatrix}
-2 & \hphantom{-} 1\\
\hphantom{-}1 & -2 & \hphantom{-}1 \\
& \hphantom{-}1 & -2 & 1\\
&& \ddots & \ddots & \ddots \\
&&&1 & -2 & \hphantom{-}1 \\
&&&& \hphantom{-}1 & -2 \end{bmatrix}, \qquad \check U=\begin{bmatrix}
0 & -1\\
1 & \hphantom{-}0 & -1 \\
& \hphantom{-}1 & \hphantom{-}0 & -1\\
&& \ddots & \ddots & \ddots \\
&&&\hphantom{-}1 & 0 & -1 \\
&&&-1 & 4 & -3 \end{bmatrix}.
\label{eqn:mm}
\end{equation}
$\pmb{H}_{M}$ excludes $H_{0}$ since $H_{0}(x)=0$. The coefficients in the last row of $\check U$ come from the parameters of the two-step backward difference formula (BDF2). Note that there is no need to replace the last row of $\check T$ using a backward difference formula since $\theta_{M}=\pi$ and, mathematically, there is no need to set an additional boundary condition since~\eqref{eqn:fd} has vanishing diffusivity for $\theta = \pi$.

Let
\begin{equation}
 T(\beta,\pmb{\theta}_{M},h):=-\bigg(\frac {2\sin^{4} \pmb{\theta}_{M}}{\beta h^2}\bigg)\check T=\begin{bmatrix}
\dfrac{-2\sin^{4} \theta_{1}}{\beta h^2} & &\\
 & \ddots & \\
& & \dfrac{-2\sin^{4} \theta_{M}}{\beta h^2} \end{bmatrix}\check T,
 \label{eqn:matri}
\end{equation}
and
\begin{align}
 &U(\beta,x,\pmb{\theta}_{M},h):=\frac {\big[\big(x+\frac {2}{\beta}\sin 2\pmb{\theta}_{M}\big)\sin^{2}\pmb{\theta}_{M}-\cos^{2}\pmb{\theta}_{M}\big]}{2h}\check U \nonumber \\
 &=\begin{bmatrix}
\dfrac{(x+\frac {2}{\beta}\sin 2\theta_1)\sin^{2}\theta_{1}-\cos^{2}\theta_{1}}{2h} & &\\
 & \ddots & \\
& & \dfrac{(x+\frac {2}{\beta}\sin 2\theta_M)\sin^{2}\theta_{M}-\cos^{2}\theta_{M}}{2h} \end{bmatrix}\check U,
 \label{tt2}
\end{align}
where $\pmb{\theta}_{M}= [\theta_{1},\theta_{2},\dots,\theta_{M} ]^{\textsf{T}}$. Then
\begin{equation}
\frac {\partial \pmb{H}_{M}}{\partial x}=\left[T(\beta,\pmb{\theta}_{M},h)+U(\beta,x,\pmb{\theta}_{M},h)\right]\pmb{H}_{M}.
 \label{eqn:fi}
\end{equation}
Noting that $x$ is the time-like variable, we apply the trapezoidal rule with time step $\Delta x<0$ to~\eqref{eqn:fi} yielding
\begin{gather*}
 \frac {\pmb{H}^{n+1}_{M}-\pmb{H}^{n}_{M}}{\Delta x}= \frac {1}{2}\big[T(\beta,\pmb{\theta}_{M},h)\pmb{H}^{n}_{M}+U(\beta,x_{n},\pmb{\theta}_{M},h)\pmb{H}^{n}_{M}\big]\\
 \hphantom{\frac {\pmb{H}^{n+1}_{M}-\pmb{H}^{n}_{M}}{\Delta x}=}{}
 +\frac {1}{2}\big[T(\beta,\pmb{\theta}_{M},h)\pmb{H}^{n+1}_{M}+U(\beta,x_{n+1},\pmb{\theta}_{M},h)\pmb{H}^{n+1}_{M}\big],
\end{gather*}
where $\pmb{H}^{n}_{M}\approx \pmb{H}_{M}(x_{n})$. Upon rearranging, this gives
\begin{gather}
 \bigg[I-\frac {\Delta x}{2}T(\beta,\pmb{\theta}_{M},h)-\frac {\Delta x}{2}U(\beta,x_{n+1},\pmb{\theta}_{M},h) \bigg] \pmb{H}^{n+1}_{M} \nonumber\\
 \quad\qquad= \bigg[I+\frac {\Delta x}{2}T(\beta,\pmb{\theta}_{M},h)+\frac {\Delta x}{2}U(\beta,x_{n},\pmb{\theta}_{M},h)\bigg]\pmb{H}^{n}_{M}.\label{trap2}
\end{gather}
Finally, we obtain $F_{\beta}(x)\approx \tilde{F}_{\beta}(x):=H_{M}(x) \approx H^n_M$, $n=0,1,\dots, N$. Below, we also explore other time integration methods, in addition to the trapezoidal method. We use the trapezoidal method as our default due to its A-stability, but it can be outperformed by BDF methods in this context.

\subsection{Spectral discretization}
To obtain better accuracy, we apply a Fourier spectral method. As before, for $0\leq n\leq N$, define $x_{n}=x_{0}+n\Delta x$. Suppose $H(x,\theta)=\int_{0}^{\theta} \rho(x,\theta')\, {\rm d}\theta'$, then we rewrite~\eqref{eqn:bv1} as
\[
 \frac {\partial H}{\partial x}+\bigg(\frac {2}{\beta}\sin^{4} \theta\bigg)\frac {\partial \rho}{\partial \theta}+\bigg[\bigg(x+\frac {2}{\beta}\sin 2\theta\bigg)\sin^{2}\theta-\cos^{2}\theta\bigg]\rho=0.
\]
Upon taking the derivative with respect to $\theta$ on both sides, we arrive at a partial differential equation for $\rho(x,\theta)$,
\begin{align}
\frac {\partial \rho}{\partial x}&+\bigg(\frac {8}{\beta}\sin^{3}\theta \cos\theta\bigg)\frac {\partial \rho}{\partial \theta}+\bigg(\frac {2}{\beta}\sin^{4} \theta\bigg)\frac {\partial^{2} \rho}{\partial \theta^{2}} \nonumber\\
&+\bigg[(2x+2)\sin \theta \cos \theta+\frac {2}{\beta}\big(2\sin^{2}\theta \cos {2\theta}+2\sin \theta\cos \theta \sin {2\theta}\big)\bigg]\rho \nonumber\\
&+\bigg[\bigg(x+\frac {2}{\beta}\sin 2\theta\bigg)\sin^{2}\theta-\cos^{2}\theta\bigg]\frac {\partial \rho}{\partial \theta}=0.\label{spect}
\end{align}
Now, suppose
\begin{equation}
 \rho(x,\theta)\approx \sum_{m=-M}^{M}a_{m}(x){\rm e}^{2{\rm i}m\theta/l},\qquad \theta\in[0,l\pi),\quad \theta_{M}=l\pi.
 \label{eqn:perio}
\end{equation}
Substituting~\eqref{eqn:perio} in~\eqref{spect} gives a system of ordinary differential equations for $a_{m}(x)$, $m=-M,\dots,M$, after truncation
\begin{equation}
\frac {{\rm d} \pmb{a}_{M}(x)}{{\rm d} x}=(A+xB)\pmb{a}_{M}(x),\qquad\pmb{a}_{M}(x)=\begin{bmatrix}
a_{-M}(x)\\
\vdots\\
a_{M}(x) \end{bmatrix},
\label{eqn:bdf}
\end{equation}
where
\begin{align}
A={}&-\frac {2}{\beta}\bigg(\frac {3}{8}I-\frac {1}{4}S_{-l}-\frac {1}{4}S_{+l}+\frac {1}{16}S_{+2l}+\frac {1}{16}S_{-2l}\bigg)D_{2} \nonumber \\
&{}+\bigg[\frac {1}{2}I+\frac {1}{4}S_{-l}+\frac {1}{4}S_{+l}-\frac {2}{\beta}\bigg(\frac {1}{2{\rm i}}S_{-l}-\frac {1}{2{\rm i}}S_{+l}\bigg)\bigg(\frac {1}{2}I-\frac {1}{4}S_{+l}-\frac {1}{4}S_{-l}\bigg)\bigg]D_{1} \nonumber\\
&{}-\bigg[\frac {8}{\beta}\bigg(\frac {1}{-8{\rm i}}S_{+3l/2}+\frac {1}{8{\rm i}}S_{-3l/2}+\frac {3}{8{\rm i}}S_{-l/2}-\frac {3}{8{\rm i}}S_{+l/2}\bigg)\bigg(\frac {1}{2}S_{-l/2}+\frac {1}{2}S_{+l/2}\bigg)\bigg]D_{1}\nonumber\\
&{}-\frac {4}{\beta}\bigg(\frac {1}{2}I-\frac {1}{4}S_{+l}-\frac {1}{4}S_{-l}\bigg)\bigg(\frac {1}{2}S_{+l}+\frac {1}{2}S_{-l}\bigg)\nonumber\\
&{}-\frac {4}{\beta}\bigg(\frac {1}{2{\rm i}}S_{-l/2}-\frac {1}{2{\rm i}}S_{+l/2}\bigg)\bigg(\frac {1}{2}S_{+l/2}+\frac {1}{2}S_{-l/2}\bigg)\bigg(\frac {1}{2{\rm i}}S_{-l}-\frac {1}{2{\rm i}}S_{+l}\bigg)\nonumber\\
&{}-2\bigg(\frac {1}{2{\rm i}}S_{-l/2}-\frac {1}{2{\rm i}}S_{+l/2}\bigg)\bigg(\frac {1}{2}S_{+l/2}+\frac {1}{2}S_{-l/2}\bigg),\label{A}
\end{align}
and
\begin{equation}
B=\bigg(-\frac {1}{2}I+\frac {1}{4}S_{-l}+\frac {1}{4}S_{+l}\bigg)D_{1}-2\bigg(\frac {1}{2{\rm i}}S_{-l/2}-\frac {1}{2{\rm i}}S_{+l/2}\bigg)\bigg(\frac {1}{2}S_{+l/2}+\frac {1}{2}S_{-l/2}\bigg).
\label{eqn:B}
\end{equation}
Here $S_{\pm k}$ represent the (Fourier modes) shift matrices, $S_{\pm k}=S^{k}_{\pm 1}$, where
\begin{align*}
 S_{+1}=\begin{bmatrix}
0 & 1 & & & & 0\\
& 0 & 1 \\
& & 0 & 1\\
&& & \ddots & \ddots \\
&&& & 0 & 1 \\
1&&&& & 0 \end{bmatrix},\qquad S_{-1}=\begin{bmatrix}
0 &&&&& 1\\
1 & 0 & &&\\
0 &1 & 0 & &\\
&& \ddots & \ddots & \\
&&&1 & 0 & 0 \\
0&&&& 1 & 0 \end{bmatrix}.
\end{align*}
Also $D_{1}$ and $D_{2}$ represent the first and second-order differentiation matrices in the Fourier space, i.e., $D_{2}=D^{2}_{1}$, where
\begin{align*}
 D_{1}=\begin{bmatrix}
-2{\rm i}M/l &&&& \\
 & -2{\rm i}(M-1)/l & &&\\
&& \ddots & & \\
&&& 2{\rm i}(M-1)/l & \\
&&&& & 2{\rm i}M/l \end{bmatrix}.
\end{align*}
The Fourier coefficients of the initial condition are obtained via
\begin{equation}
 a_{m}(x_{0})=\frac {1}{l\pi}\int_{0}^{l\pi}\rho(x_{0},\theta){\rm e}^{-2{\rm i}m\theta/l}\, {\rm d}\theta.
 \label{eqn:eq2}
\end{equation}
Instead of applying the second-order accurate trapezoidal rule to integrate the system of ODEs for $\rho(x,\theta)$, we suggest the use of a five-step backward differentiation formula method (BDF5), see~\cite[Section 8.4]{10.5555/1355322}. To use BDF5, we need four more starting conditions, i.e., $\pmb{a}_{M}(x_{0}-i\Delta x)$, $i=1,2,3,4$, which can be obtained by~\eqref{initialc} and~\eqref{eqn:eq2}.
We then proceed with BDF5 to solve for $\pmb{a}^{n}_{M}\approx\pmb{a}_{M}(x_{n})$
\begin{align}
 \left[137I-60\Delta x \left(A+x_{n}B\right)\right]\pmb{a}^{n}_{M}= 300\pmb{a}^{n-1}_{M}&-300\pmb{a}^{n-2}_{M}
 +200\pmb{a}^{n-3}_{M}
 -75\pmb{a}^{n-4}_{M}+12\pmb{a}^{n-5}_{M}.\label{spect2}
\end{align}
Other time integration methods can be used, and this will be discussed further below.

Finally, the value of $H(x,\theta)$ can be recovered from
\[
 H(x,\theta)\approx\int_{0}^{\theta}\sum_{m=-M}^{M}a_{m}(x){\rm e}^{2{\rm i}m\theta'/l}\, {\rm d}\theta'=\sum_{m=-M}^{M}a_{m}(x)\int_{0}^{\theta}{\rm e}^{2{\rm i}m\theta'/l}\, {\rm d}\theta'.
\]
The Tracy--Widom distribution can then be approximated by setting $\theta=\pi$,
\[
F_{\beta}(x)\approx\tilde{F}_{\beta}(x)=\sum_{m=-M}^{M}a_{m}(x)\int_{0}^{\pi}{\rm e}^{2{\rm i}m\theta'/l}\, {\rm d}\theta'=\sum_{m=-M}^{M}a_{m}(x)\frac {l}{2{\rm i}m}\big({\rm e}^{2{\rm i}m\pi/l}-1\big),
\]
from which we find
\[
 F_{\beta}(x_n)\approx\tilde F_\beta(x_n) \approx \sum_{m=-M}^{M}a^{n}_{m}\frac {l}{2{\rm i}m}\big({\rm e}^{2{\rm i}m\pi/l}-1\big),\qquad n=0,1,\dots,N.
\]

\section{Algorithm validation and comparison}\label{s:valid}

At this point, we have approximated the values of $F_{\beta}(x)$ on equally-spaced grid points $x_{0},x_{1},\allowbreak \ldots,x_{N}$. To obtain a continuous function, we perform Fourier interpolation. Since Fourier interpolation has high accuracy when the function being interpolated is periodic, we define
\[
 \varphi(x)=\frac {\mathrm{erf}(x)+1}{2},\qquad x\in \mathbb {R},
\]
where $\mathrm{erf}$ denotes the error function~\cite{NIST:DLMF}. Then consider
\[
 G_{\beta}(x):=\tilde{F}_{\beta}(x)-\varphi(x),
\]
which is nearly a periodic function. Note that other functions can also be used instead of the error function $\mathrm{erf}(x)$ in constructing a periodic function. We perform Fourier interpolation to~$G_{\beta}(x)$ on the grid $x_{1},\dots,x_{N}$. The resulting interpolant is then evaluated on a shifted and scaled Chebyshev grid on $[x_{N},x_{0}]$. By adding back in $\varphi(x)$ evaluated on the same Chebyshev grid and then interpolating with Chebyshev polynomials, we obtain a useful high-accuracy approximation of $F_{\beta}(x)$. The number of Chebyshev coefficients is user-decided (we use $10^3$ by default).

We point out that using either~\eqref{eqn:fd} or~\eqref{eqn:bdf}, we can also obtain the approximation of $F'_{\beta}(x)$ on the grid $x_{0},x_{1},\ldots,x_{N}$ at nearly no extra cost. We then perform the same procedure of interpolation to get an approximation of $F'_{\beta}(x)$, without using $\varphi(x)$ since the function being interpolated is already nearly periodic.

Algorithm~\ref{alg:11} shows the pseudocode for the Fourier interpolation, where $\texttt{T}$ is the grid points $x_{N},x_{N-1},\dots,x_{0}$ with grid spacing $\Delta x$, $\chebpts$ is an integer that gives the number of Chebyshev coefficients, $\texttt{F}$ is the value of $\tilde{F}_{\beta}(x_{j})$, $j=1,\dots,N$, and $\texttt{f}$ is the value of $\tilde{F}'_{\beta}(x_{j})$. We think of~\texttt{F},~\texttt{f} as functions on \texttt{T} and need to extend them to all of $[x_{N},x_{0}]$.

\begin{algorithm}
 \caption{\texttt{Fourier}$\underline{~}$\texttt{interp}(\texttt{T}, $\chebpts$, \texttt{F}, \texttt{f})} \label{alg:11}
 \begin{algorithmic}[1]
 \State Set $\texttt{D=}[\min\texttt{T},\max\texttt{T}]$
 \State Set $\texttt{G}(\texttt{T})\texttt{=F}(\texttt{T})\texttt{-}\varphi(\texttt{T})$
\State Perform Fourier interpolation to $\texttt{G}(\texttt{T})$ and \texttt{f}(\texttt{T}) to obtain the interpolant $\texttt{G}(x)$ and $\texttt{f}(x)$, $x\in \texttt {D}$
\State Set $\tilde{\texttt{F}}\texttt{=G}(\bar{x})\texttt{+}\varphi(\bar{x})$ and $\tilde{\texttt{f}}(\bar{x})\texttt{=f}(\bar{x})$, where $\bar{x}$ is a Chebyshev grid on $\texttt{D}$ with $\chebpts$ points
\State Perform Chebyshev interpolation to $\tilde{\texttt{F}}$ and $\tilde{\texttt{f}}$ and return the interpolants
	\end{algorithmic}
\end{algorithm}

For the finite-difference discretization using trapezoidal method to compute either the cumulative distribution function (cdf) or the probability density function (pdf) of $F_{\beta}$ using \texttt{TW}$(\beta)$, as provided in {\tt TracyWidomBeta.jl}, the following default values for the parameters are used:
\begin{equation}
 x_{0}=\bigg\lfloor \frac {13}{\sqrt{\beta}} \bigg\rfloor ,\quad x_{N}=-10,\quad \Delta x=-10^{-3},\quad \theta_{M}=\pi,\quad \chebpts=10^3,\quad M=10^3,
 \label{eqn:paraf}
\end{equation}
where $\lfloor \cdot \rfloor$ denotes the floor function.

\begin{Notation} Throughout we use \texttt{TW}($\beta$; \texttt{params}) to refer to our implementation with different choices of parameters. For example,
\[
\text{\texttt{TW}$(\beta; \texttt{method="finite"}, \texttt{step="trapz"}, \texttt{pdf=true})$},
\]
refers to using the finite-difference discretization, time-stepped with the trapezoidal method, outputting an approximation of $F'_\beta$ with the default parameters~\eqref{eqn:paraf}.
\end{Notation}

The default values of $x_{0}$ and $x_{N}$ are chosen to be $\lfloor 13/\sqrt{\beta} \rfloor$ and $-10$ respectively so that $F_{\beta}(x_{0})\approx 1$ and $F_{\beta}(x_{N})\approx 0$ for $\beta\geq 1$. Though not optimal, larger values of $x_{0}$ and smaller values of $x_{N}$ can also be used. See Section~\ref{s:x0} for a discussion on the selection of the default value for $x_{0}$. For $0<\beta<1$, a larger domain for $x$ should be used. The values of $\Delta x$ and $M$ are chosen so that $M= \lfloor{-1/\Delta x}\rfloor$. In this way, the local truncation error of the trapezoidal method is of the optimal order. Algorithm~\ref{alg:1} shows the pseudocode for \texttt{TW}$(\beta; \texttt{pdf=true})$. Step 6 may be replaced by solving an alternate discretization of~\eqref{eqn:fi}.
\begin{algorithm}
 \caption{\texttt{TW}($\beta$; \texttt{pdf=true})}
 \begin{algorithmic}[1]
 \State Set the values of $x_{0}$, $x_{N}$, $\Delta x$, $\theta_{\text{M}}$, \chebpts, and M as in~\eqref{eqn:paraf}
 \State Set $x\texttt{=}x_{0}:\Delta x:x_{N}$
 \State Set up the initial condition as in~\eqref{initialc}
 \State Set up the matrices as in (\ref{eqn:mm}), (\ref{eqn:matri}), and~\eqref{tt2}
		\For {$\texttt{m=}1,\ldots,\texttt{length}(x)\texttt{-}1$}
			\State Solve~\eqref{trap2} and denote the result by $\texttt{H}$
 \State Use~\eqref{eqn:fd} and denote the result by $\texttt{h}$
		\EndFor
 \State Take the values of $\texttt{H}$ and $\texttt{h}$ for $\theta\texttt{=}\pi$ and denote the results by $\texttt{F}$ and $\texttt{f}$
 \State Perform Algorithm~\ref{alg:11} with $\texttt{F}$, \texttt{f}, and $\texttt{T=}x$
 \State Return the interpolant for $\tilde{\texttt{f}}$
	\end{algorithmic}
 \label{alg:1}
\end{algorithm}

Similarly, for the spectral discretization using BDF5 method to compute either the cdf or the pdf using
\begin{gather*}
\begin{split}
&\text{\texttt{TW}$(\beta; \texttt{method="spectral"}, \texttt{step="bdf5"})$} \quad \text{or}\\
&\text{\texttt{TW}$(\beta; \texttt{method="spectral"}, \texttt{step="bdf5"}, \texttt{pdf=true})$},
\end{split}
\end{gather*}
 the following values for the parameters are used:
\begin{gather}
 x_{0}=\bigg\lfloor \frac {13}{\sqrt{\beta}} \bigg\rfloor,\quad x_{N}=-10,\quad \Delta x=-10^{-3},\quad \theta_{M}=20\pi,\quad \chebpts=10^3,\quad M=8\times 10^{3}.
 \label{eqn:paras}
\end{gather}
One thing to note here is that $\theta_{M}$ is set to be $20\pi$, which is much larger than $\pi$ as used for finite-difference discretization. This setting is necessary since a periodic boundary condition~\eqref{eqn:perio} is imposed on $\rho(x,\theta)$, or equivalently, on $H(x,\theta)$. If the length of the domain, $\theta_{M}$, is not large enough, due to periodicity, the approximation to $H(x,\theta)$ will propagate to the end of the domain and reappear at the lower boundary $\theta=0$. In other words, $\theta_{M}$ depends on the speed of propagation of the approximation to $H(x,\theta)$, and it turns out that setting $\theta_{M}=20\pi$ is sufficient. $M$ is set to be $8\times 10^3$ regardless of the value of $\Delta x$ to make sure we have large enough number of Fourier modes to represent the initial condition. See Section~\ref{s:22} for details on which value of $M$ to use for each method in terms of both accuracy and computation time. Algorithm~\ref{alg:2} shows the pseudocode for \texttt{TW}$(\beta; \texttt{method="spectral"}, \texttt{step="bdf5"}, \texttt{pdf=true})$. As with the finite-difference method, step 11 can be replaced with solving an alternate discretization of~\eqref{eqn:bdf}.
\begin{algorithm}
 \caption{\texttt{TW}($\beta$; \texttt{method="spectral"}, \texttt{step="bdf5"}, \texttt{pdf=true})}
 \begin{algorithmic}[1]
 \State Set the values of $x_{0}$, $x_{N}$, $\Delta x$, $\theta_{\text{M}}$, \chebpts, and M as in~\eqref{eqn:paras}
 \State Set $x\texttt{=}x_{0}:\Delta x:x_{N}$
 \State Set \texttt{D=}$0:h:\theta_{\text{M}}-h$ with $h\texttt{=}\theta_{\text{M}}/\text{M}$
 \State Set up the initial conditions using~\eqref{initialc} and~\eqref{eqn:eq2}.
 \State Construct matrices \texttt{A} and \texttt{B} as in~\eqref{A} and~\eqref{eqn:B}
 \State Set $\texttt{N=-}\texttt{floor}(\text{M}/2):1:\texttt{floor}(({\text{M}-}1)/2)$
\For {$\texttt{m} \in \texttt{N}$}
\State Compute the vector $\texttt{integ=}\frac {l}{2{\rm i}m}\big({\rm e}^{2{\rm i}m\pi/l}\texttt{-}1\big)$ with $l\texttt{=}\theta_{\text{M}}/\pi$
\EndFor
 \For {$\texttt{k=}2:\texttt{length}(x)$}
 \State Solve~\eqref{spect2} to find $\textbf{a}^{n}_{\text{M}}$
 \State Use~\eqref{eqn:bdf} and denote the result by $\textbf{b}^{n}_{\text{M}}$
 \EndFor
 \State Compute \texttt{F=}$\big\langle\textbf{a}^{n}_{\text{M}},\texttt{integ}\big\rangle$ and \texttt{f=}$\big\langle\textbf{b}^{n}_{\text{M}},\texttt{integ}\big\rangle$
 \State Perform Algorithm~\ref{alg:11} with $\texttt{F}$, \texttt{f}, and $\texttt{T=}x$
 \State Return the interpolant for $\tilde{\texttt{f}}$
	\end{algorithmic}
 \label{alg:2}
\end{algorithm}

\begin{Remark}
To optimize the use of the spectral method, one should consider $\theta$ from the interval $[k_1(x) \pi, k_2(x) \pi]$, where $k_1$, $k_2$ are $x$-dependent integers that adapt to where the solution is non-zero, see Figure~\ref{fig:pp}.
\end{Remark}

Note that when $\beta$ is very large, numerical instabilities can occur. This instability arises primarily because the initial condition closely resembles a step function. To get an accurate result in this case, first of all, one needs to use finer grid for $x$ and larger value for $M$, i.e., refinement in both time and space. Also, one needs to use more Chebyshev coefficients. With current values for the parameters, \texttt{TW}$(\beta)$ and
\texttt{TW}$(\beta; \texttt{method="spectral"}, \texttt{step="bdf5"})$ exhibits stability roughly for $1\leq \beta\leq 30$.

\subsection[Eigenvalues of Delta x(T(beta,theta\_M,h)+U(beta,x,theta\_M,h)) for finite-difference discretization]{Eigenvalues of $\boldsymbol{\Delta x[T(\beta,\theta_{M},h)+U(\beta,x,\theta_{M},h)]}$\\ for finite-difference discretization}

\begin{figure}[t]
 \centering\vspace{-2mm}
 \subfloat[\vspace*{-5.5mm}\centering Zoomed-out view] {{\includegraphics[width=0.36\linewidth]{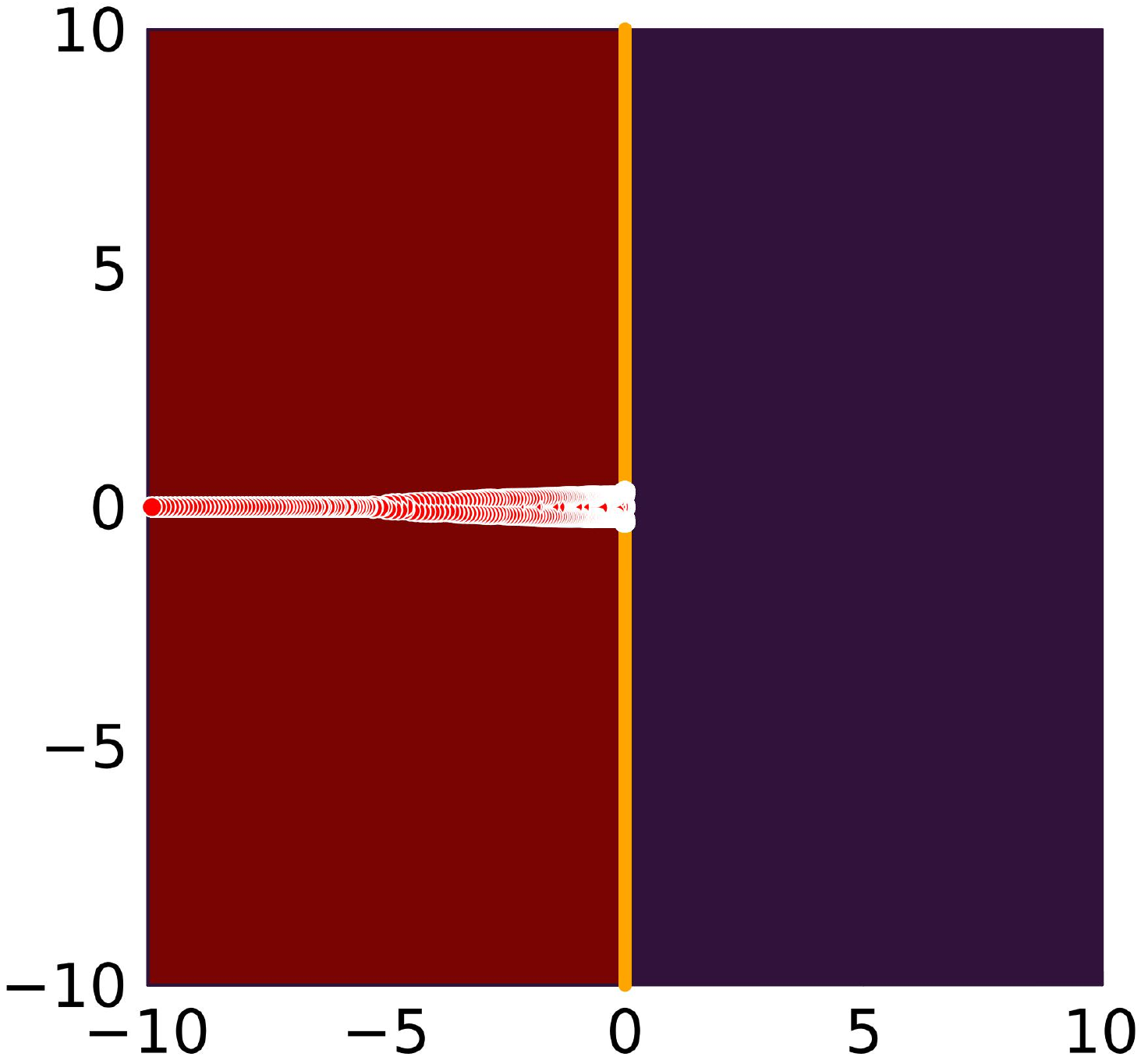} }}%
 \put(-178,67){\rotatebox{90}{\color{black}\small $\operatorname{Im}(z)$}}
 \put(-90,-7){\color{black}\small $\operatorname{Re}(z)$}
 \qquad\qquad
 \vspace*{5mm}
 \subfloat[\vspace*{-5.5mm}\centering Zoomed-in view]{{\includegraphics[width=0.365\linewidth]{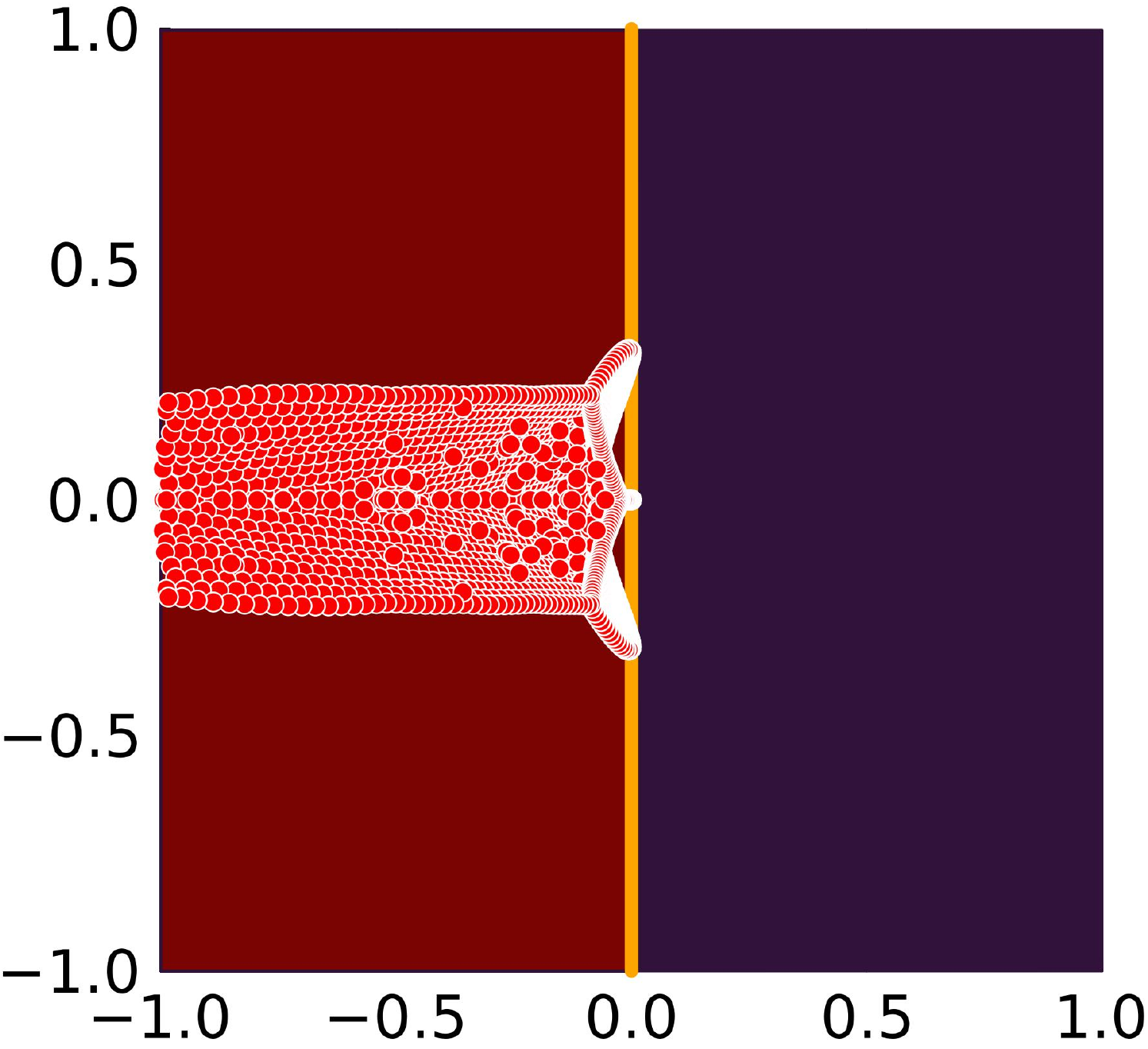} }}%
 \put(-178,67){\rotatebox{90}{\color{black}\small $\operatorname{Im}(z)$}}
 \put(-90,-7){\color{black}\small $\operatorname{Re}(z)$}

 \vspace{-1mm}

 \caption{The absolute stability region of trapezoidal method applied to the finite-difference discretization (the maroon-shaded region along with the orange boundary curve) and the eigenvalues of $\Delta x (T+U)$ (the red dots with white boundary) for $\beta=2$, $x=x_0,x_0-1,\dots,x_{N}$, $\Delta x=-10^{-3}$, and $M=10^3$.} \label{p:bdf5_fd}
\end{figure}

\begin{figure}[t]
 \centering\vspace{-3mm}

 \subfloat[\vspace*{-5.5mm}\centering Zoomed-out view] {{\includegraphics[width=0.36\linewidth]{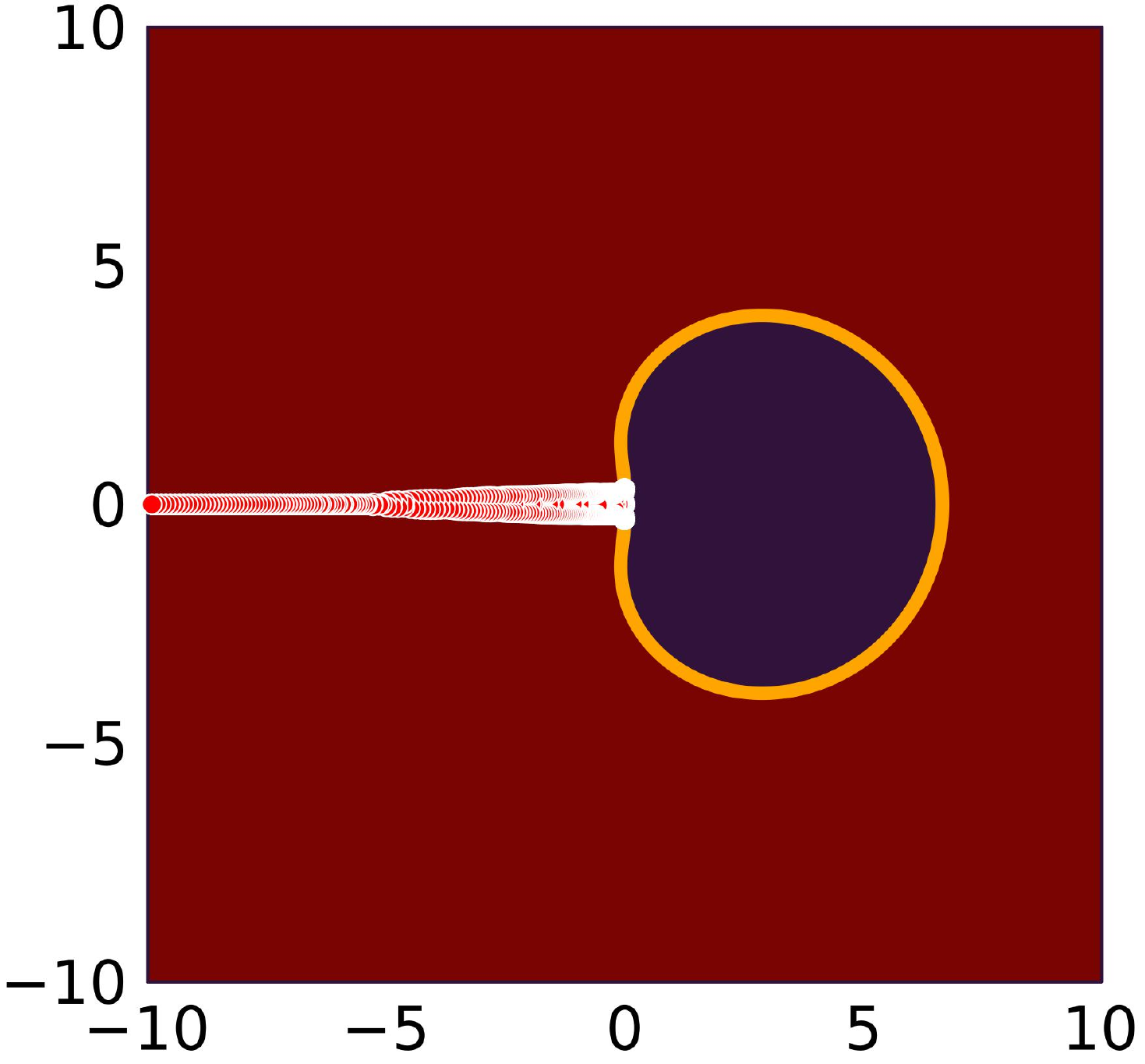} }}%
 \put(-178,67){\rotatebox{90}{\color{black}\small $\operatorname{Im}(z)$}}
 \put(-90,-7){\color{black}\small $\operatorname{Re}(z)$}
 \qquad\qquad
 \vspace*{3mm}
 \subfloat[\vspace*{-5.5mm}\centering Zoomed-in view]{{\includegraphics[width=0.365\linewidth]{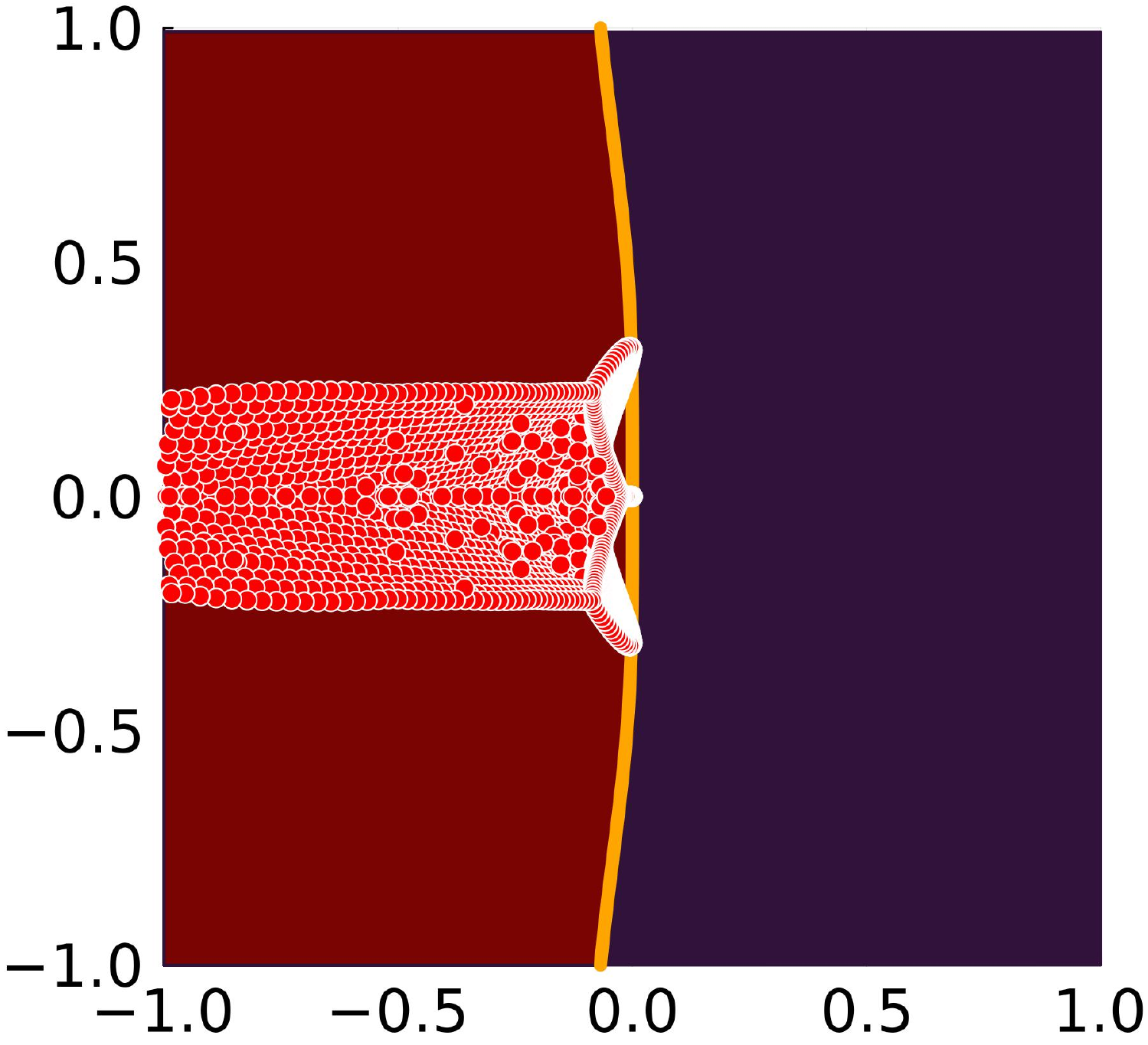} }}%
 \put(-178,67){\rotatebox{90}{\color{black}\small $\operatorname{Im}(z)$}}
 \put(-90,-7){\color{black}\small $\operatorname{Re}(z)$}

  \vspace{-1mm}

 \caption{The absolute stability region of BDF3 applied to the finite-difference discretization (the maroon-shaded region along with the orange boundary curve) and the eigenvalues of $\Delta x (T+U)$ (the red dots with white boundary) for $\beta=2$, $x=x_0,x_0-1,\dots,x_{N}$, $\Delta x=-10^{-3}$, and $M=10^3$.} \label{p:bdf6_fd}
 \vspace{-2mm}
\end{figure}

The default time-stepping routine for the finite-difference discretization is the trapezoidal me\-thod but BDF3, BDF4, BDF5, and BDF6 can also be used on~\eqref{eqn:fd} to solve for $\pmb{H}_{M}$. It turns out that with values for the parameters as in~\eqref{eqn:paraf}, convergence occurs for all these methods. Since with the finite-difference discretization, the trapezoidal method and BDF3 are usually used, Figures~\ref{p:bdf5_fd} and~\ref{p:bdf6_fd} show the absolute stability regions of trapezoidal method and BDF3 along with eigenvalues of $\Delta x( T + U)$ for $\beta=2$, $x=x_0,x_0-1,\dots,x_{N}$, $\Delta x=-10^{-3}$, and $M=10^3$. The eigenvalues in the left-half plane are firmly within the region of absolute stability in each case.\looseness=1

\subsection[Eigenvalues of Delta x(A+xB) for the spectral discretization]{Eigenvalues of $\boldsymbol{\Delta x(A+xB)}$ for the spectral discretization}
\begin{figure}[t]\vspace{-2mm}
 \centering
 \subfloat[\vspace*{-5.5mm}\centering Zoomed-out view] {{\includegraphics[width=0.36\linewidth]{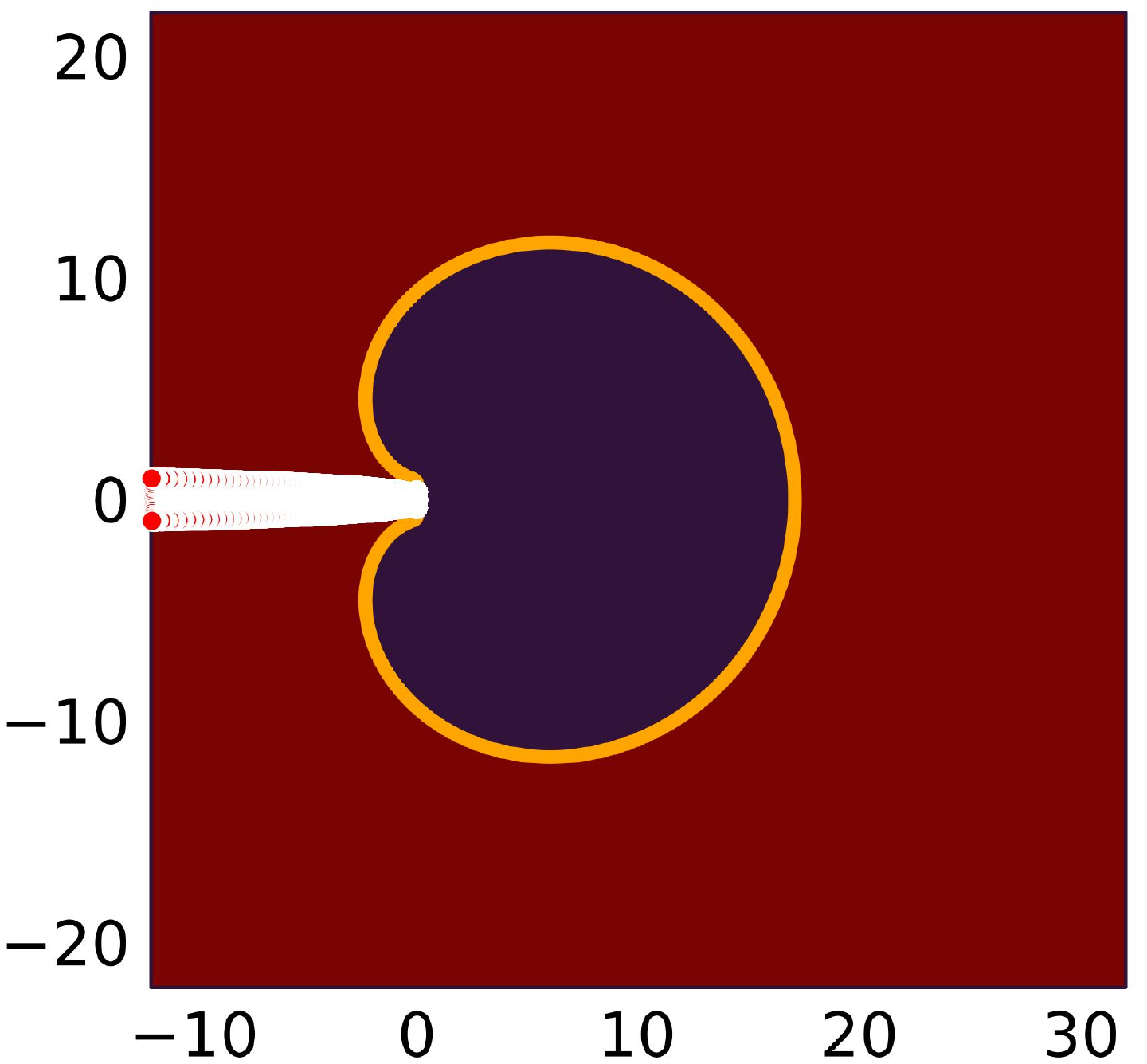} }}%
 \put(-178,67){\rotatebox{90}{\color{black}\small $\operatorname{Im}(z)$}}
 \put(-90,-7){\color{black}\small $\operatorname{Re}(z)$}
 \qquad\qquad
 \vspace*{5mm}
 \subfloat[\vspace*{-5.5mm}\centering Zoomed-in view]{{\includegraphics[width=0.37\linewidth]{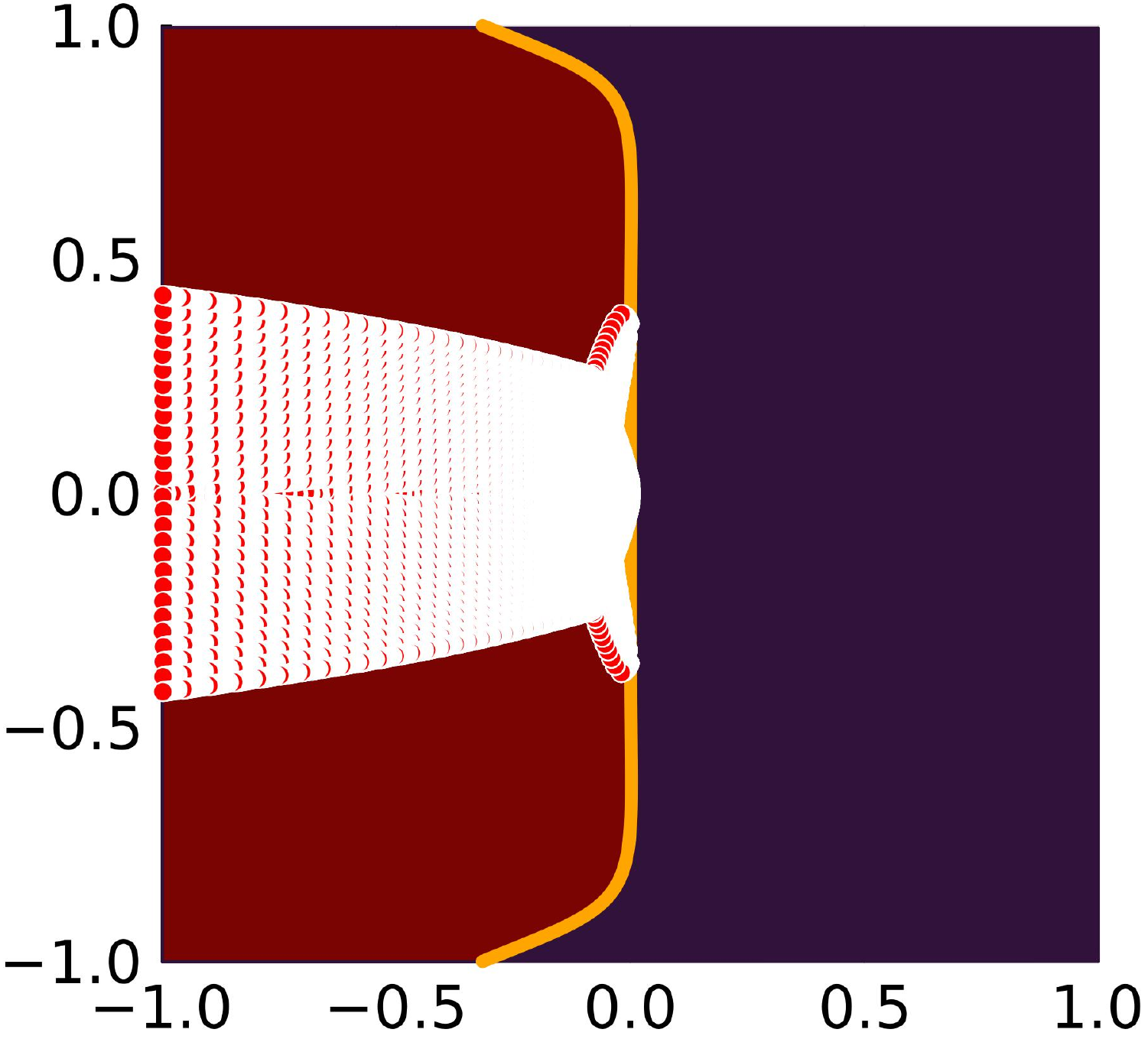} }}%
 \put(-178,67){\rotatebox{90}{\color{black}\small $\operatorname{Im}(z)$}}
 \put(-90,-7){\color{black}\small $\operatorname{Re}(z)$}
 \caption{The absolute stability region of BDF5 applied to the spectral discretization (the maroon-shaded region along with the orange boundary curve) and the eigenvalues of $\Delta x (A+xB)$ (the red dots with white boundary) for $\beta=2$, $x=x_0,x_0-1,\dots,x_{N}$, $\Delta x=-10^{-3}$, and $M=8\times 10^3$.}%
 \label{p:bdf5_sd}%
\end{figure}
\begin{figure}[t]\vspace{-3mm}
 \centering
 \subfloat[\vspace*{-5.5mm}\centering Zoomed-out view] {{\includegraphics[width=0.36\linewidth]{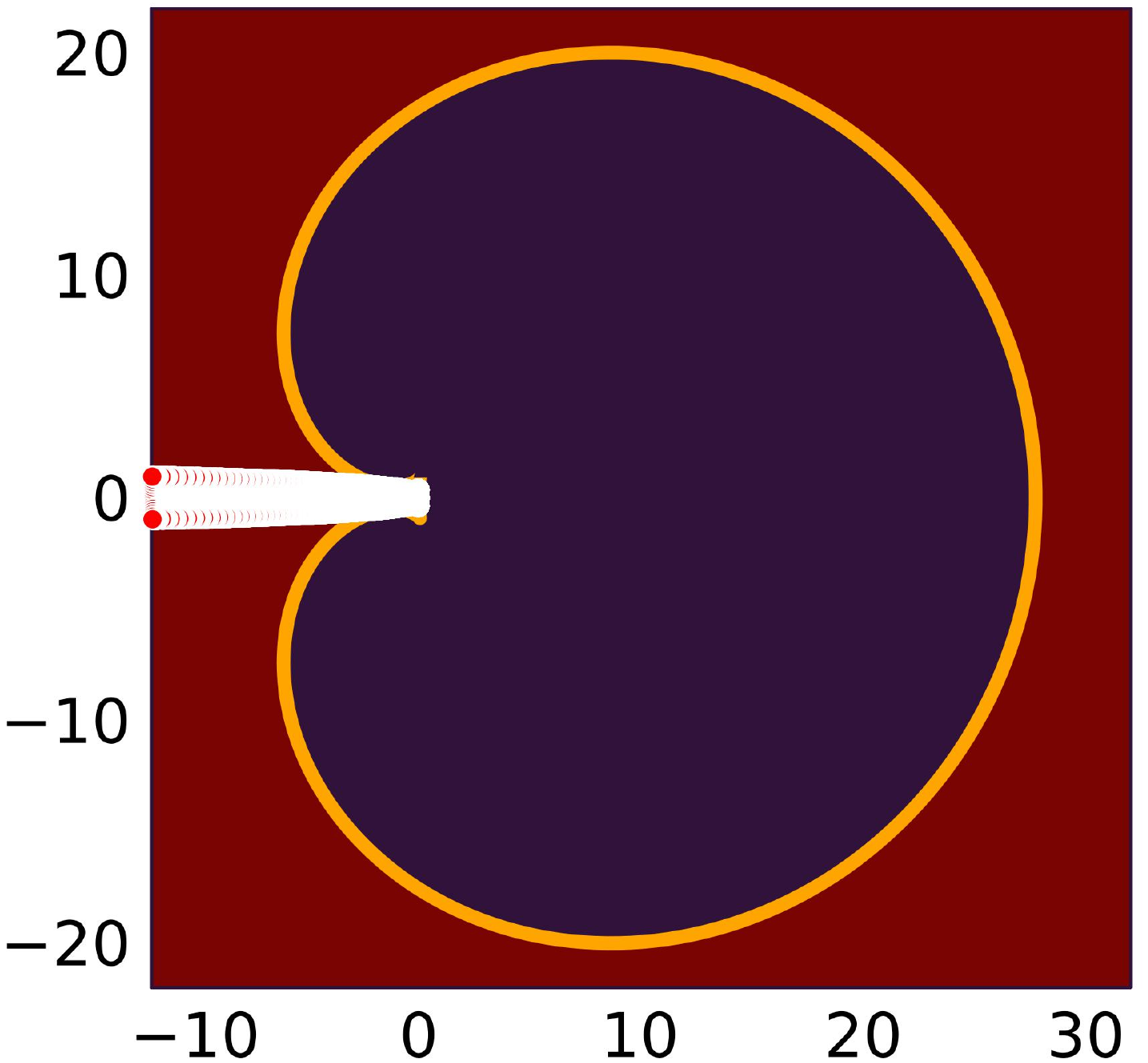} }}%
 \put(-178,67){\rotatebox{90}{\color{black}\small $\operatorname{Im}(z)$}}
 \put(-90,-7){\color{black}\small $\operatorname{Re}(z)$}
 \qquad\qquad
 \vspace*{5mm}
 \subfloat[\vspace*{-5.5mm}\centering Zoomed-in view]{{\includegraphics[width=0.37\linewidth]{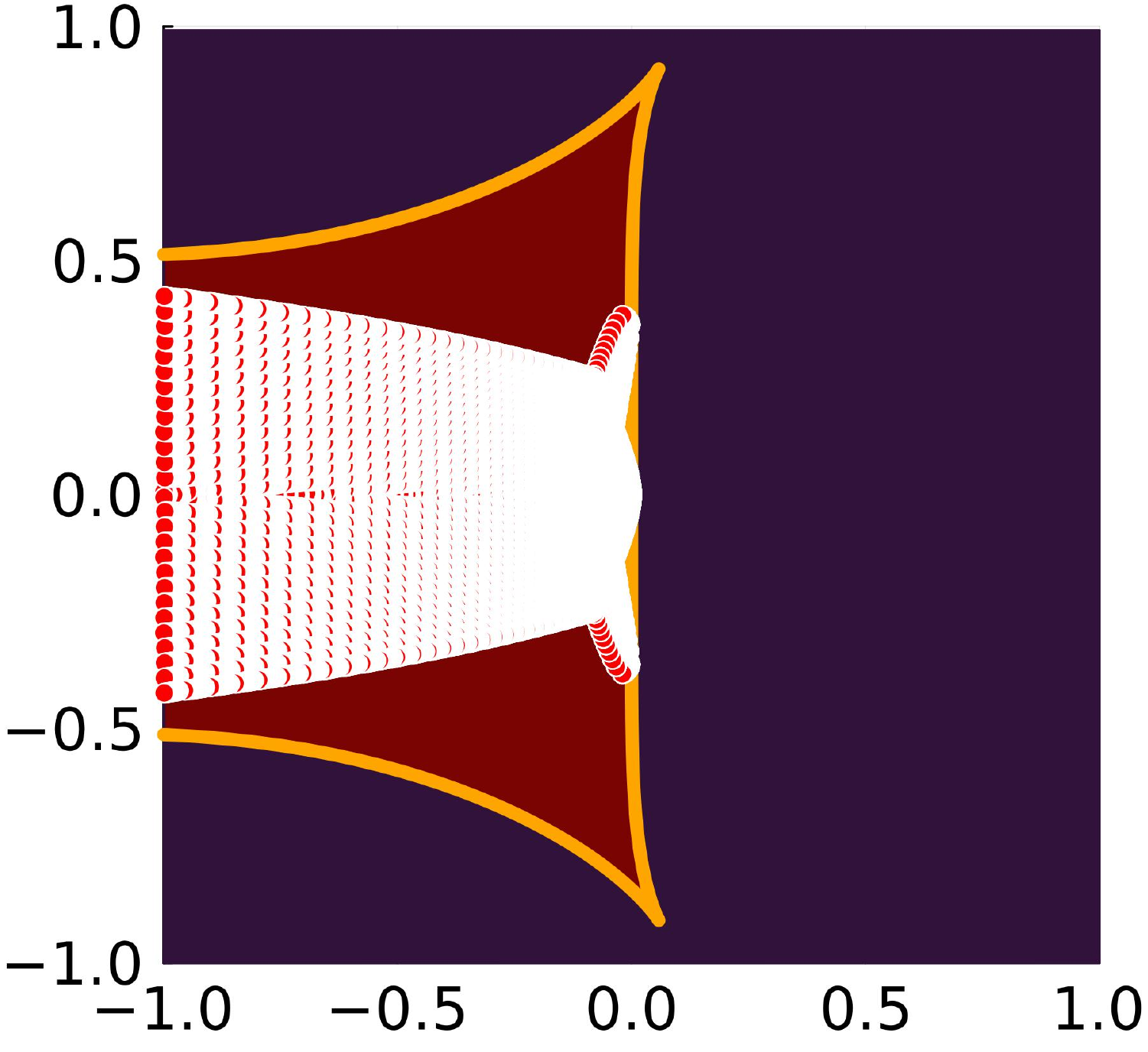} }}%
 \put(-178,67){\rotatebox{90}{\color{black}\small $\operatorname{Im}(z)$}}
 \put(-90,-7){\color{black}\small $\operatorname{Re}(z)$}
 \caption{The absolute stability region of BDF6 applied to the spectral discretization (the maroon-shaded region along with the orange boundary curve) and the eigenvalues of $\Delta x (A+xB)$ (the red dots with white boundary) for $\beta=2$, $x=x_0,x_0-1,\dots,x_{N}$, $\Delta x=-10^{-3}$, and $M=8\times 10^3$.}%
 \label{p:bdf6_sd}\vspace{-1mm}
\end{figure}

Figures~\ref{p:bdf5_sd} and~\ref{p:bdf6_sd} show the absolute stability regions of BDF5 and BDF6 along with eigenvalues of $\Delta x (A + x B)$ for $\beta=2$, $x=x_0,x_0-1,\dots,x_{N}$, $\Delta x=-10^{-3}$, and $M=8\times 10^3$. It is clear that, with the values of the parameters given as in~\eqref{eqn:paras}, the eigenvalues in the left-half plane are firmly within the region of absolute stability for these methods. We choose BDF5 over BDF6 because with the same values of these parameters, the performance is roughly the same yet BDF5 has a larger absolute stability region. Selecting $\Delta x$ is more nuanced than one might think, see Appendix~\ref{a:2} for more details.

\subsection{Error analysis}

We now compare the accuracy of these two algorithms when computing the cdf for ${\beta=1,2,4}$. For reference solutions, we implement the Fredholm determinant representations for $\beta = 1,2,4$ \mbox{\cite{Bornemann2009OnTN,BornemannFredholm, Ferrari2005}} by porting the code in the \textsc{Julia} package \texttt{RandomMatrices} to \textsc{Mathematica} and implementing it in high-precision arithmetic. This ensures that our reference solutions are accurate to beyond the $\approx 10^{-16}$ machine precision for standard double precision.
Observe that for~$F_4(x)$,~$x$ should be divided by a factor of $2^{1/6}$. This difference in variance convention has also been underscored in~\cite[p.~47]{nadalc}, \cite{Mays2021,Nadal2011}, and~\cite[Remark 5.1.4]{Bloemendal2011FiniteRP}.

\subsubsection{Error across the domain}
\begin{figure}[!ht]
 \centering
 \subfloat[\vspace*{-5.5mm}\centering FD (Trapz)] {{\includegraphics[width=0.46\linewidth]{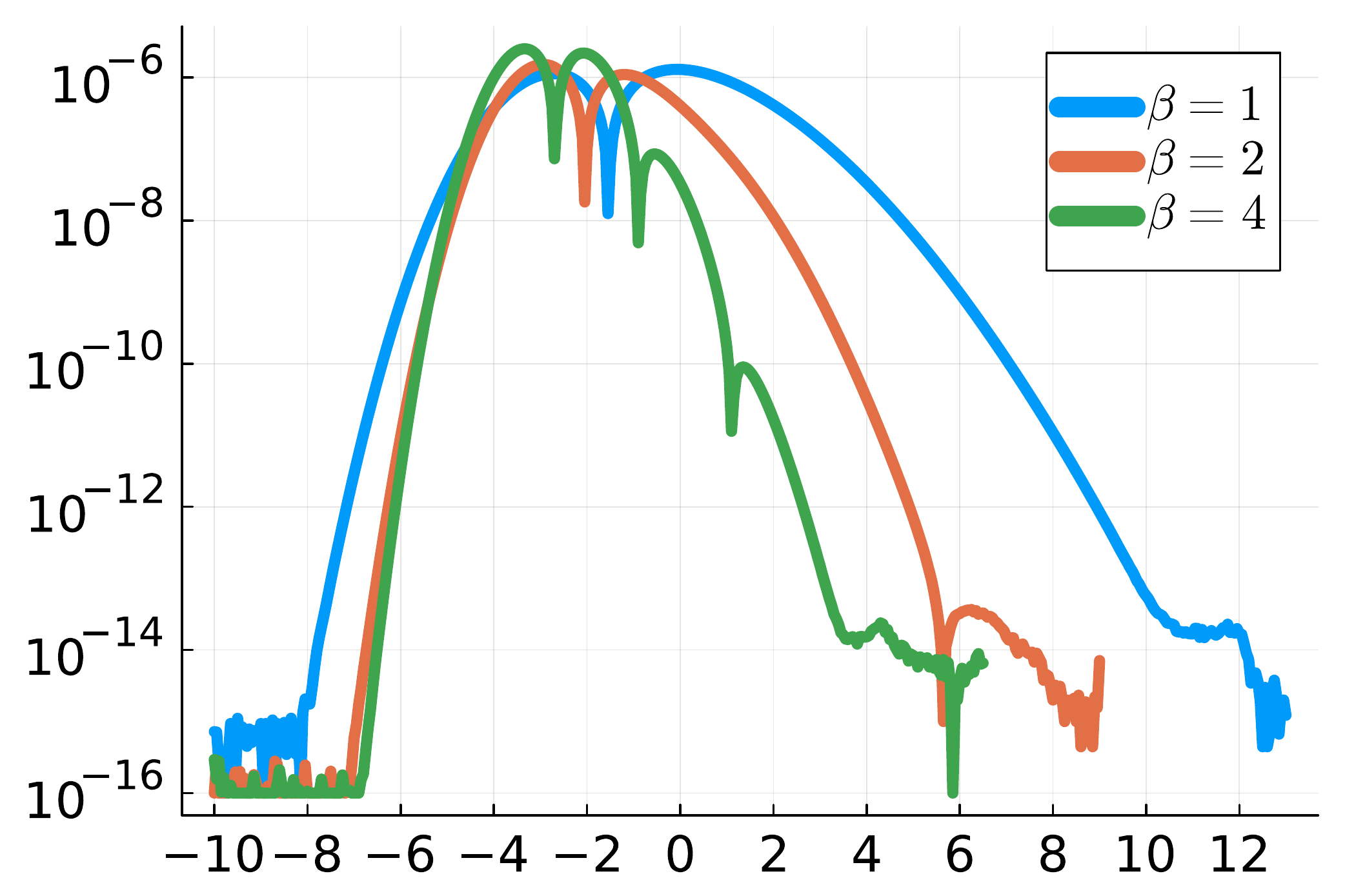} }}%
 \put(-220,63){\rotatebox{90}{\small Error}}
 \put(-103,-6){\color{black}\small $x$}
 \qquad
 \vspace*{5mm}
 \subfloat[\vspace*{-5.5mm}\centering SD (BDF5)]{{\includegraphics[width=0.46\linewidth]{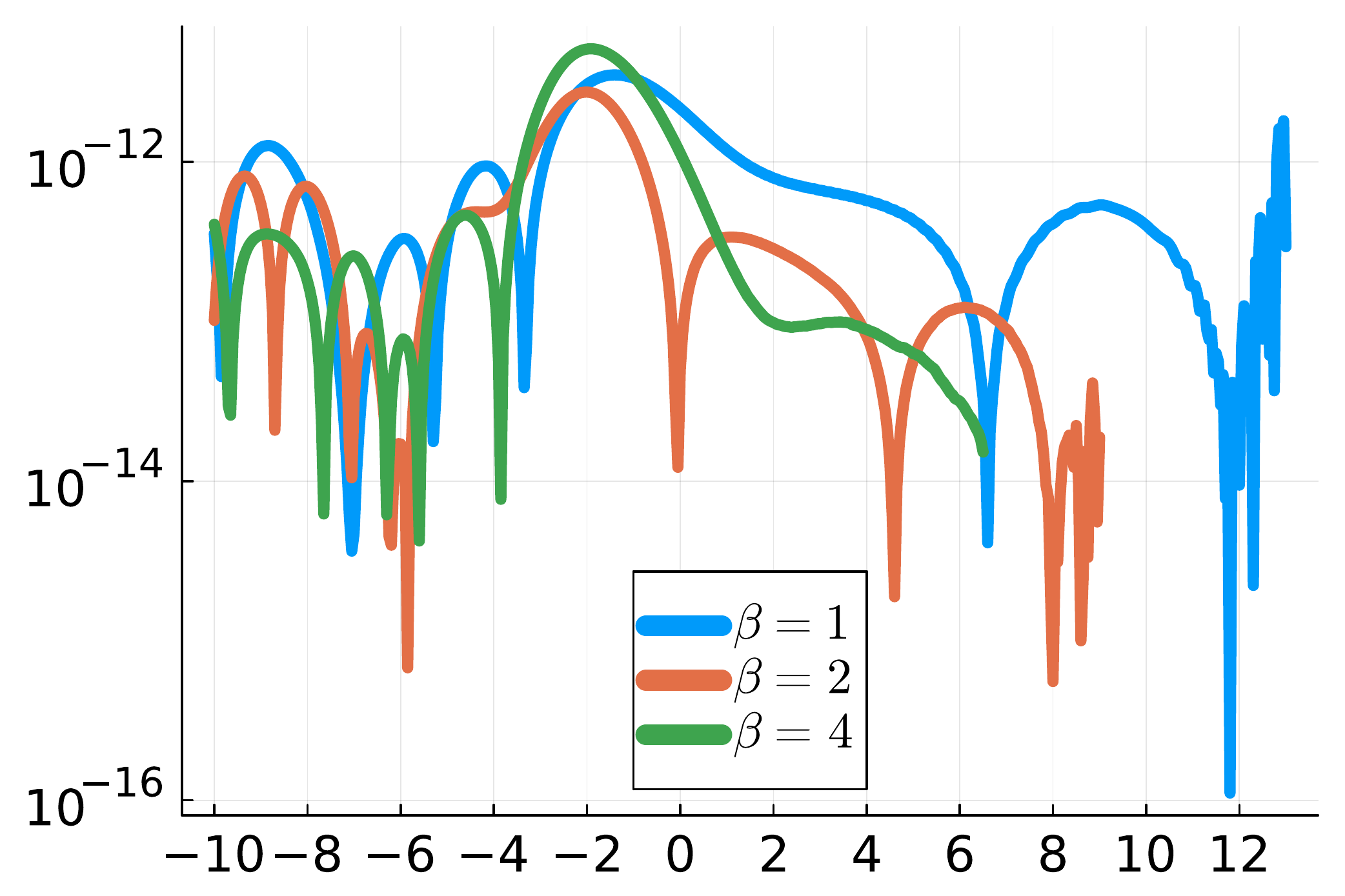} }}%
 \put(-220,63){\rotatebox{90}{\small Error}}
 \put(-103,-6){\color{black}\small $x$}
 \caption{Absolute errors from the two algorithms are presented over the domain $x\in [-10,13]$ for $\beta=1,2,4$ with $x_{0}=\lfloor 13/\sqrt{\beta}\rfloor$ and $\Delta x=-10^{-3}$. For the finite-difference discretization, $M=10^3$ is used, while for the spectral discretization, $M=8\times 10^3$ is employed.}%
 \label{p:erroracross}%
\end{figure}
With values of the parameters given as in~\eqref{eqn:paraf} and~\eqref{eqn:paras}, Figure~\ref{p:erroracross} shows the absolute errors of the two algorithms across the domain. Tables~\ref{tab1} and \ref{tab2} show the absolute errors of some selected $x$-values with same values for the parameters.

For finite-difference discretization, from either Figure~\ref{p:erroracross}\,(a) or Table~\ref{tab1}, the error is roughly on the order of $10^{-7}$ around the peak of the distribution, and it improves when $x$ approaches either end of the domain. This is due to the fact that exact values, $0$ and $1$, for the initial condition are imposed on the endpoints of the domain, and the Dirichlet boundary condition ensures that the solution tends to zero.

For spectral discretization, from either Figure~\ref{p:erroracross}\,(b) or Table~\ref{tab2}, the error is roughly on the order of $10^{-13}$ throughout the entire domain. Regardless of the order of error, the reason for this discrepancy from finite-difference discretization near the edges of the domain is that an error is introduced for the initial condition when we represent $\rho(x,\theta)$ in terms of a truncated Fourier series.
\begin{table}[t]
 \centering
 \begin{tabular}{|c|c|c|c|}
 \hline
 \diagbox[width=5em]{$x$}{$\beta$} &\makebox[9em]{$1$}&\makebox[9em]{$2$} &\makebox[9em]{$4$} \\ \hline
 $-8$ & $1.832(-15)$ & $5.024(-17)$ & $1.950(-17)$\\ \hline
 $-6$ & $7.278(-10)$ & $9.602(-12)$ & $4.766(-16)$ \\ \hline
 $-4$ & $3.715(-7)$ & $3.600(-7)$ & $1.807(-7)$ \\ \hline
 $-2$ & $6.553(-7)$ & $1.027(-7)$ & $2.020(-6)$ \\ \hline
 $0$ & $1.295(-6)$ & $4.014(-7)$ & $3.319(-8)$ \\ \hline
 $2$ & $4.116(-7)$ & $1.137(-8)$ & $6.185(-12)$ \\ \hline
 $4$ & $3.352(-8)$ & $3.286(-11)$ & $1.565(-14)$ \\ \hline
 $6$ & $9.575(-10)$ & $3.186(-14)$ & $5.995(-15)$ \\ \hline
 \end{tabular}

 \caption{Absolute errors of finite-difference discretization using trapezoidal method with values of the parameters as in~\eqref{eqn:paraf}.}
 \label{tab1}
\end{table}
\begin{table}[t]
 \centering
 \begin{tabular}{|c|c|c|c|}
 \hline
 \diagbox[width=5em]{$x$}{$\beta$} &\makebox[9em]{$1$}&\makebox[9em]{$2$} &\makebox[9em]{$4$} \\ \hline
 $-8$ & $5.912(-13)$ & $6.870(-13)$ & $3.509(-13)$\\ \hline
 $-6$ & $3.278(-13)$ & $3.470(-14)$ & $2.021(-13)$ \\ \hline
 $-4$ & $9.057(-13)$ & $4.580(-13)$ & $4.525(-13)$ \\ \hline
 $-2$ & $2.663(-12)$ & $2.740(-12)$ & $4.809(-12)$ \\ \hline
 $0$ & $2.162(-12)$ & $1.090(-13)$ & $1.135(-12)$ \\ \hline
 $2$ & $7.924(-13)$ & $3.514(-13)$ & $9.293(-14)$ \\ \hline
 $4$ & $5.719(-13)$ & $1.423(-13)$ & $7.527(-14)$ \\ \hline
 $6$ & $1.819(-13)$ & $7.094(-14)$ & $6.439(-15)$ \\ \hline
 \end{tabular}

 \caption{Absolute errors of spectral discretization using BDF5 with values of the parameters as in~\eqref{eqn:paras}.}
 \label{tab2}
\end{table}

\subsubsection{Order of error}
\begin{figure}[t]
 \centering
 \subfloat [\vspace*{-6.5mm}\centering FD discretization]{{\includegraphics[width=0.46\linewidth]{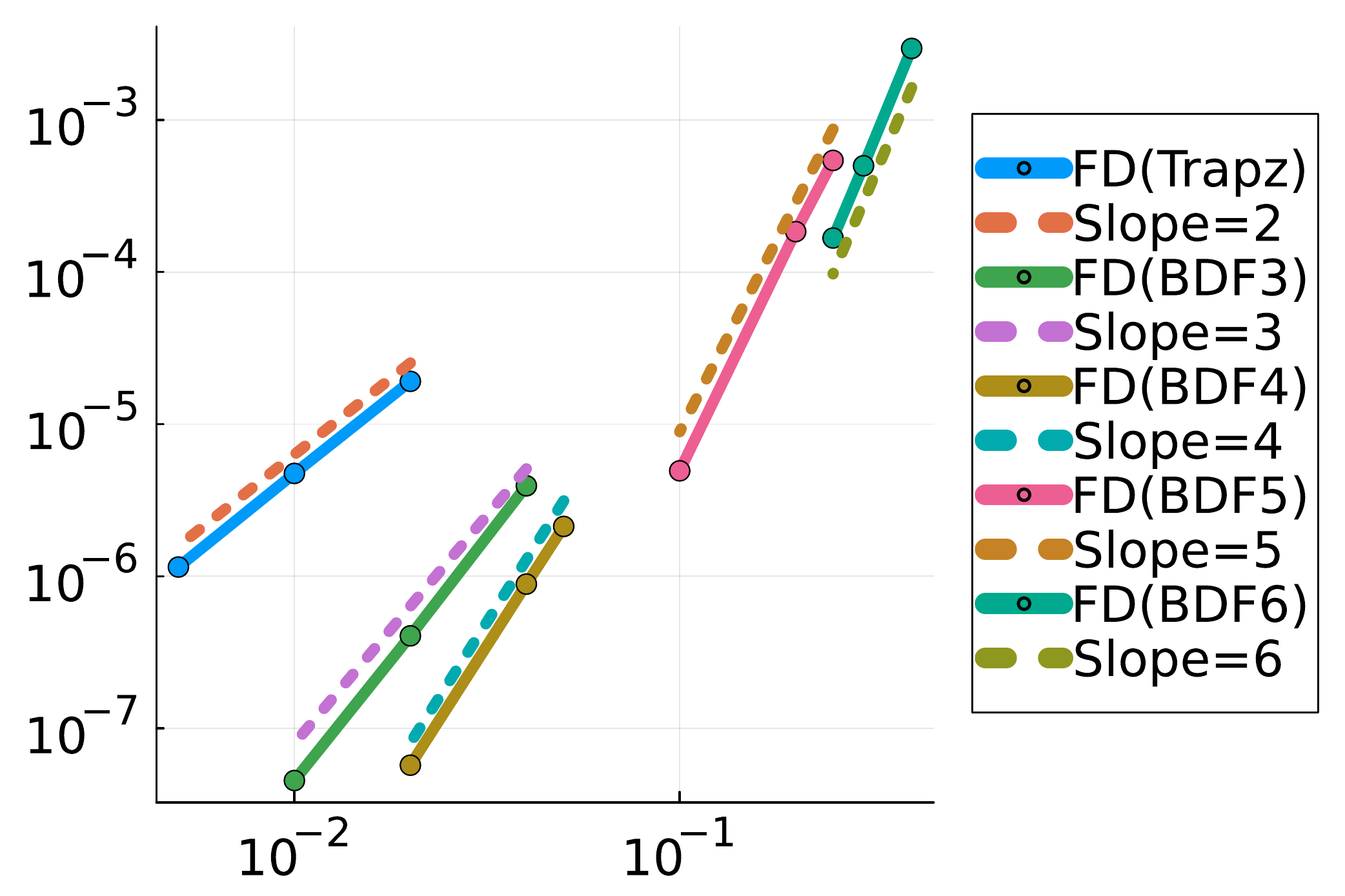} }}%
 \put(-220,63){\rotatebox{90}{\small Error}}
 \put(-120,-8){\color{black}\small $|\Delta x|$}
 \qquad
 \subfloat [\vspace*{-6.5mm}\centering Spectral discretization]{{\includegraphics[width=0.46\linewidth]{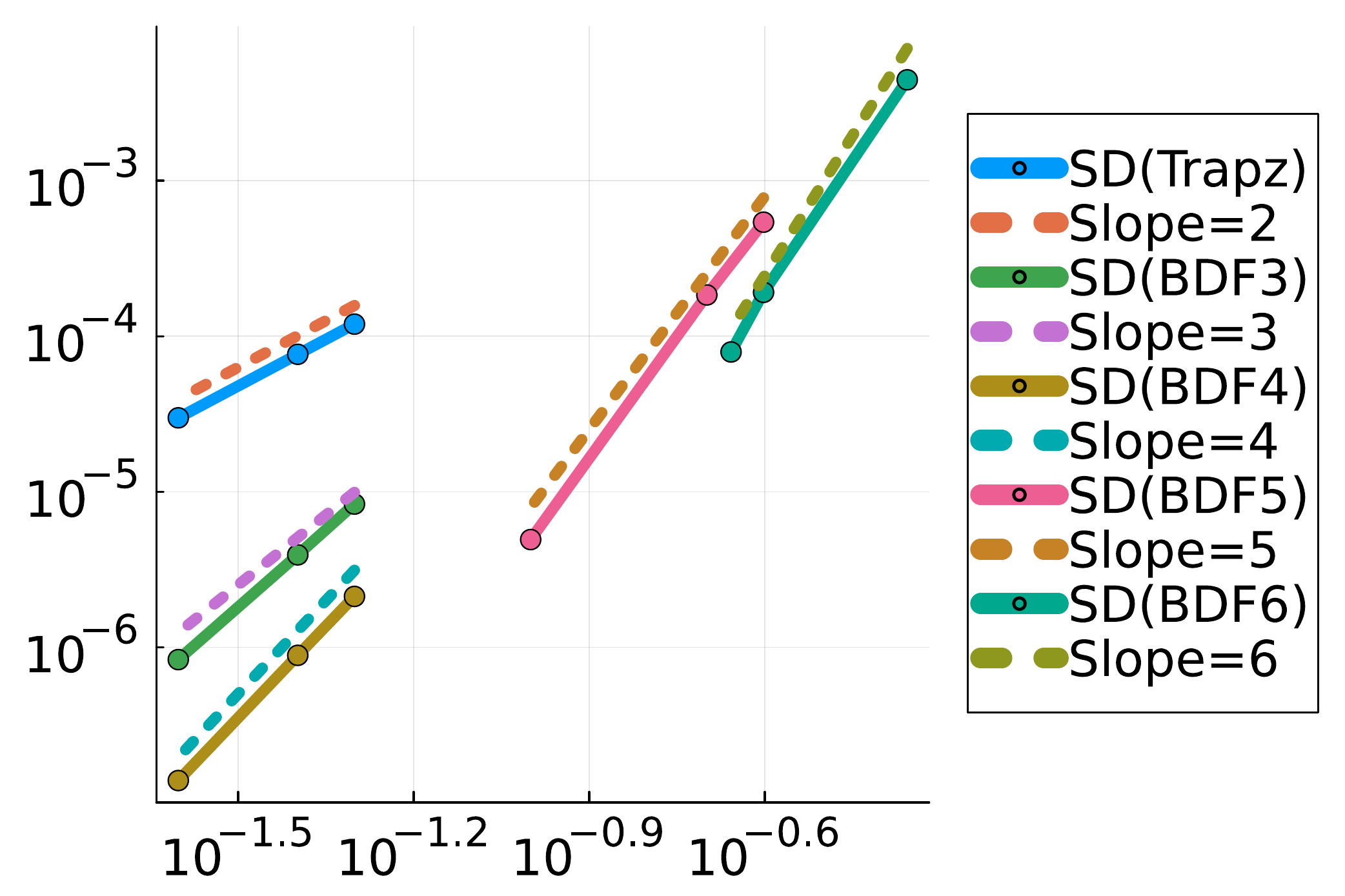} }}%
 \put(-220,63){\rotatebox{90}{\small Error}}
 \put(-110,-8){\color{black}\small $|\Delta x|$}
 \vspace*{5mm}\caption{The order of accuracy for the finite-difference (FD) discretization and the spectral discretization.  (a)
 The finite-difference discretization. The errors for trapezoidal (Trapz) method are computed by treating \texttt{TW}\big($\beta\texttt{=}2$; $x_{0}=\lfloor13/\sqrt{\beta}\rfloor$, \texttt{step="trapz"}, \texttt{interp=false}, $\Delta x\texttt{=}-10^{-3}$, $M\texttt{=}10^3$\big) as the reference and compared with $\Delta x=-0.005,-0.01,-0.02$ at $x=-2$. For BDF methods, the same reference is used with $\texttt{step}$ changed accordingly. For BDF3, $\Delta x=-0.01,-0.02,-0.04$ are used. For BDF4, $\Delta x=-0.02,-0.04,-0.05$ are used. For BDF5, $\Delta x=-0.1,-0.2,-0.25$ are used. For BDF6, $\Delta x=-0.25,-0.3,-0.4$ are used. (b) The spectral discretization. The errors for trapezoidal (Trapz) method are computed by treating \texttt{TW}\big($\beta\texttt{=}2$; $x_{0}=\lfloor 13/\sqrt{\beta}\rfloor$, \texttt{method="spectral"}, \texttt{step="trapz"}, \texttt{interp=false}, $\Delta x\texttt{=}-10^{-3}$, $M\texttt{=}8000$\big) as the reference and compared with $\Delta x=-0.025,-0.04,-0.05$ at $x=-2$. For BDF methods, the same reference is used with $\texttt{step}$ changed accordingly. For BDF3 and BDF4, $\Delta x=-0.025,-0.04,-0.05$ are used. For BDF5, $\Delta x=-0.1,-0.2,-0.25$ are used. For BDF6, $\Delta x=-0.22,-0.25,-0.44$ are used.}
 \label{p:order1}
\end{figure}

\begin{figure}[!ht]
 \centering
 \subfloat[\vspace*{-5.5mm}\centering FD (Trapz)] {{\includegraphics[width=0.46\linewidth]{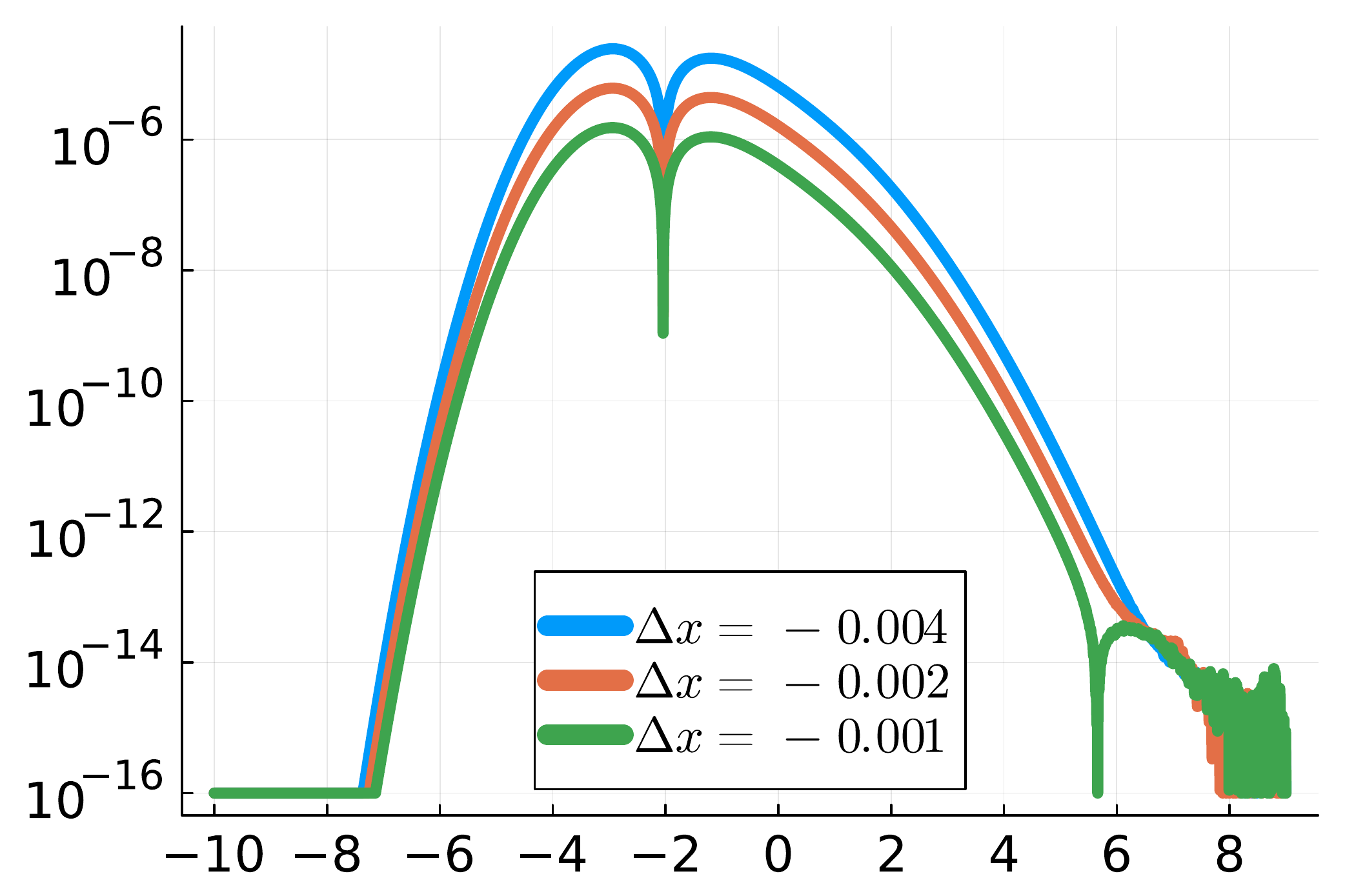} }}%
 \put(-220,63){\rotatebox{90}{\small Error}}
 \put(-103,-6){\color{black}\small $x$}
 \qquad
 \vspace*{5mm}
 \subfloat[\vspace*{-5.5mm}\centering SD (BDF5)] {{\includegraphics[width=0.46\linewidth]{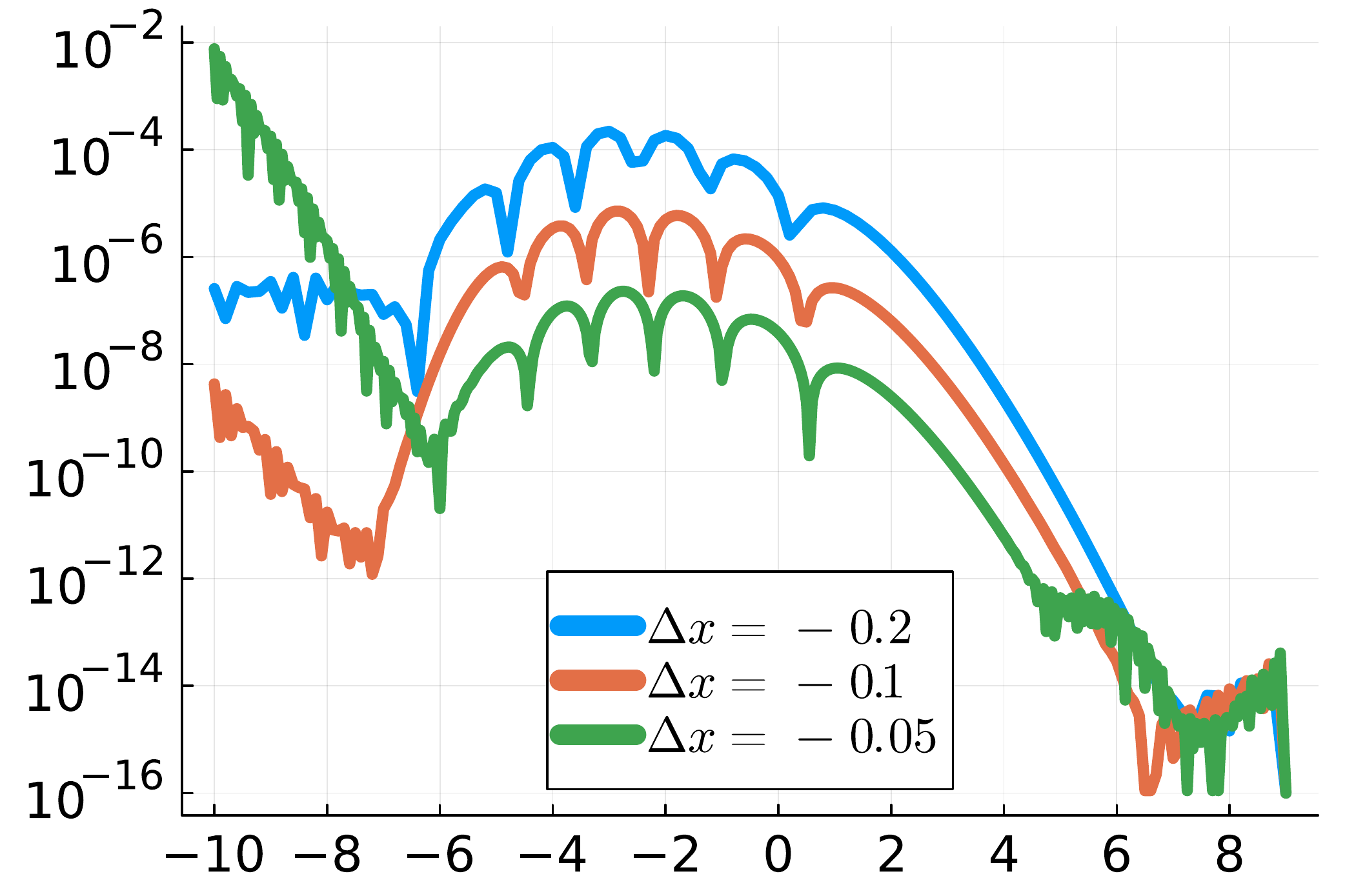} }}%
 \put(-220,63){\rotatebox{90}{\small Error}}
 \put(-103,-6){\color{black}\small $x$}
 \caption{Change of error of finite-difference discretization using trapezoidal method with $\Delta x=-0.004,-0.002,-0.001$ and spectral discretization using BDF5 with $\Delta x=-0.2,-0.1,-0.05$. Values of the other parameters are used as in~\eqref{eqn:paraf} and~\eqref{eqn:paras}. When applying BDF5 to the spectral discretization, errors are more pronounced for smaller time steps at the left edge due to the onset of instability discussed in the appendix.}%
 \label{p:errorchange}%
\end{figure}
For $\beta=2$ evaluated at $x=-2$, Figure~\ref{p:order1} shows the order of error plots of finite-difference discretization and spectral discretization. We choose $x=-2$ since it is near the peak of the distribution. With $\Delta x=-0.004,-0.002,-0.001$ for finite-difference discretization and $\Delta x=-0.2,-0.1,-0.05$ for spectral discretization, Figure~\ref{p:errorchange} shows the change of error as the value of~$|\Delta x|$ decreases.

For both the finite-difference discretization and the spectral discretization, from Figure~\ref{p:order1}, the error has the expected order with respect to $|\Delta x|$. However, if we decrease the value of $|\Delta x|$ further, each BDF method has a corresponding range for $\Delta x$ that causes instability. Moreover, once the value of~$\Delta x$ exits this range, the corresponding BDF method becomes accurate again. See Appendix~\ref{a:2} for more details.

\begin{Remark}
In addition to the trapezoidal and BDF methods, one could consider $A$-stable and $L$-stable diagonally implicit Runge--Kutta methods as alternative options for time-stepping. Furthermore, one could even consider adaptive time-stepping. While these alternative methods can effectively address the stability concerns arising from the spectrum of $\Delta x(A+xB)$ and $\Delta x(T+U)$, they do tend to increase the computation time. These methods likely require more than one linear solve at each time step (possibly one matrix factorization and multiple uses of it). And our target time step is $|\Delta x | = 10^{-3}$ and, with this time step, BDF5 is stable with the default parameters, requiring only one (sparse) linear solve per time step. Then the comparison of methods becomes a complicated competition between more expensive, higher-order methods that allow a larger time step and less expensive, but still fairly high-order, methods requiring a~smaller time step. It is possible that savings can be achieved here, but we did not see it in our experiments.
\end{Remark}

\subsubsection[Error with respect to x\_0]{Error with respect to $\boldsymbol{x_{0}}$}\label{s:x0}

\begin{figure}[t]
 \centering
 \subfloat[\vspace*{-5.5mm}\centering FD (Trapz)] {{\includegraphics[width=0.46\linewidth]{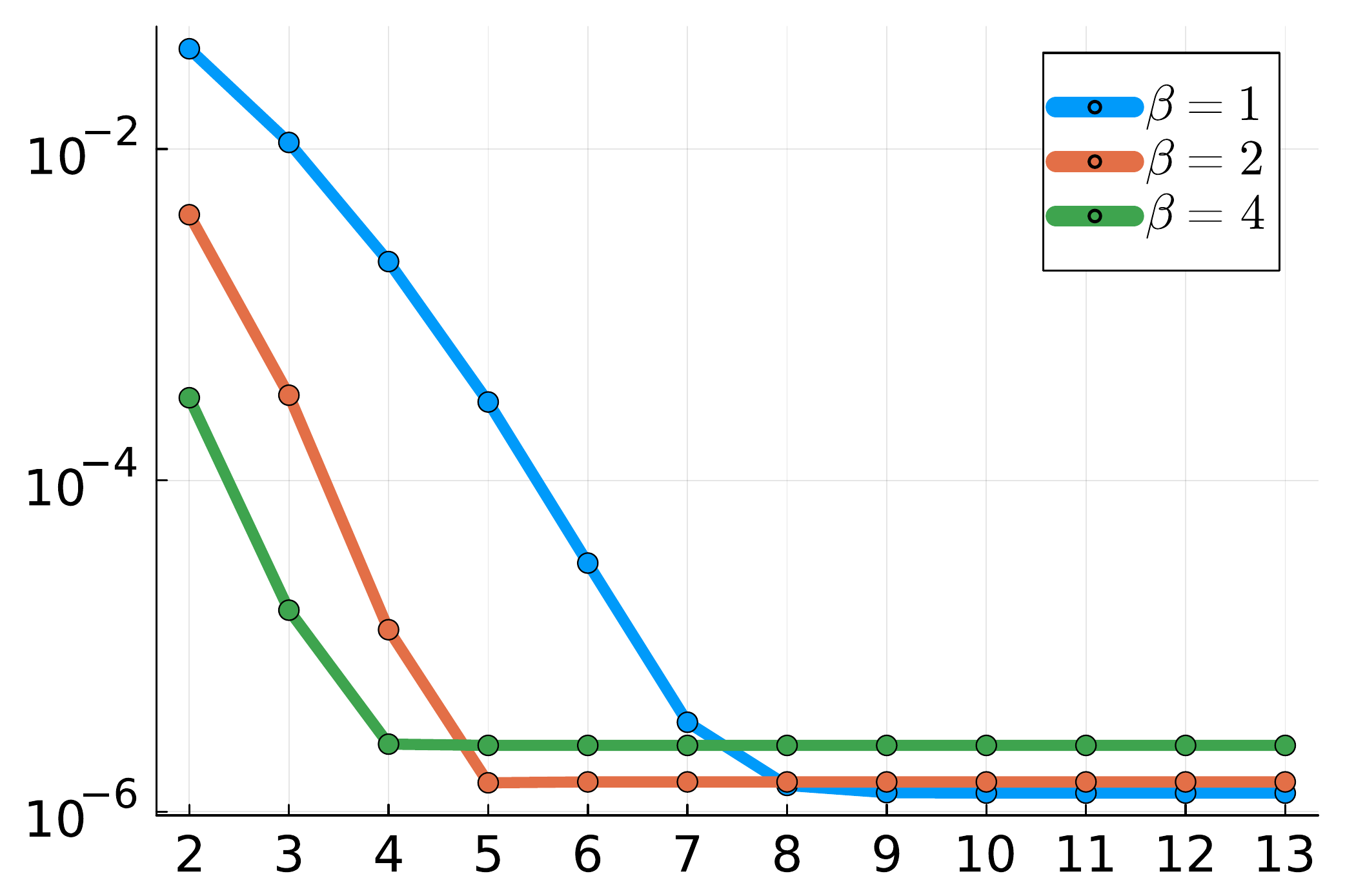} }}%
 \put(-220,63){\rotatebox{90}{\small Error}}
 \put(-103,-6){\color{black}\small $x_{0}$}
 \qquad
 \vspace*{5mm}
 \subfloat[\vspace*{-5.5mm}\centering SD (BDF5)] {{\includegraphics[width=0.46\linewidth]{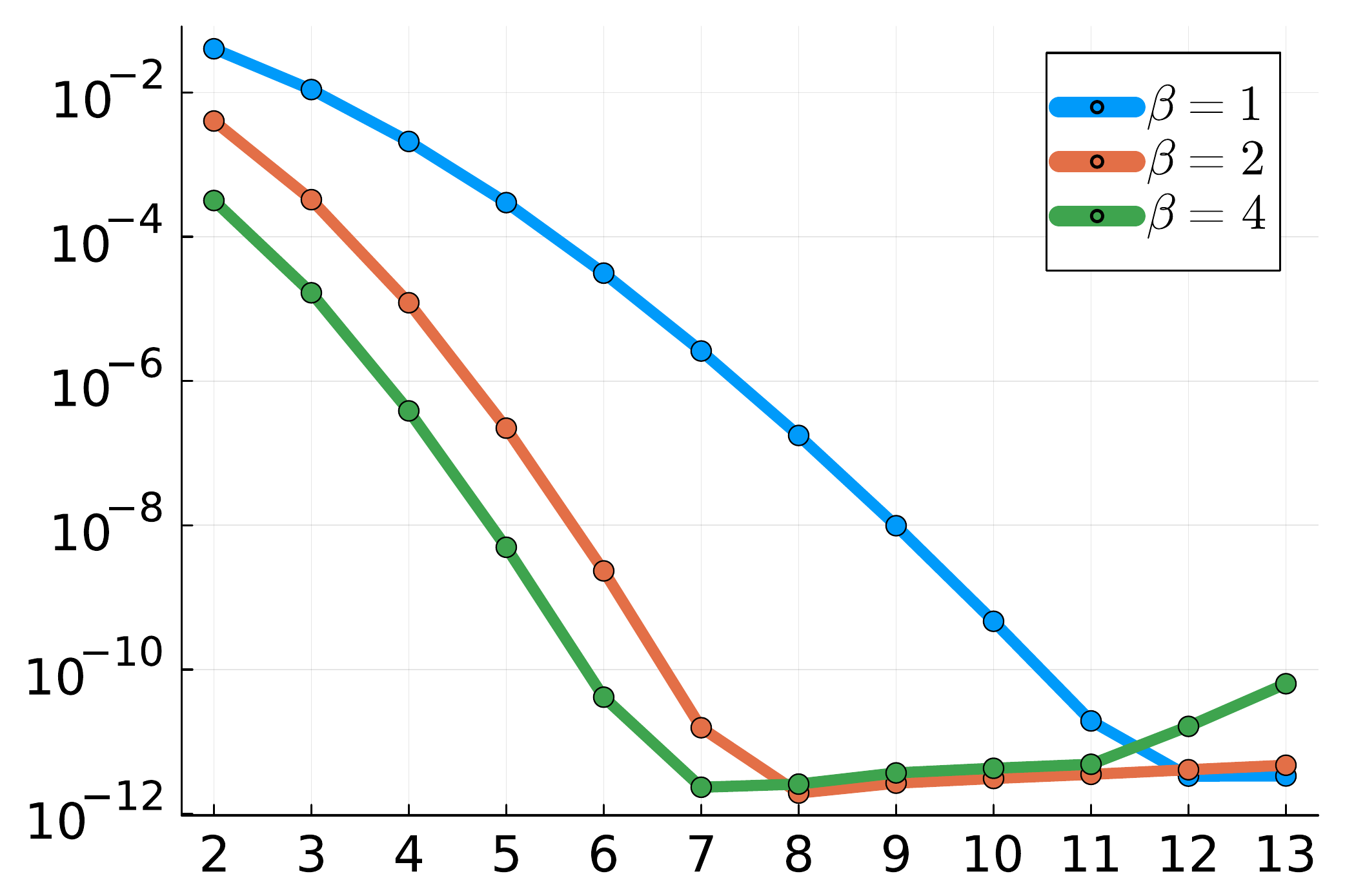} }}%
 \put(-220,63){\rotatebox{90}{\small Error}}
 \put(-103,-6){\color{black}\small $x_{0}$}
 \caption{Change of error of finite-difference discretization using trapezoidal method and spectral discretization using BDF5 with respect to the value of $x_{0}$. $\beta=1,2,4$ and $x_{0}=13,12,\dots,2$ are used. Values of the other parameters are the same as in~\eqref{eqn:paraf} and~\eqref{eqn:paras}. For each value of $\beta$ and $x_{0}$, the error is taken to be the maximum over $x=-8,-7,\dots,x_{0}$.} \label{p:x0}
\end{figure}

Figure~\ref{p:x0} shows how the maximum error over $x=-8,-7,\dots,x_{0}$ changes with respect to the value of $x_{0}$. For $\beta=4$ with spectral discretization using BDF5, the error for $x_{0}=12,13$ is roughly $\mathcal {O}\big(10^{-11}\big)$, which can be improved to $\mathcal {O}\big(10^{-12}\big)$ using more Fourier modes. Except for these two cases, for both methods, as the value of $\beta$ increases, the value of $\tilde{x}_{0}$, the minimum value of $x_{0}$ that can be used without affecting the accuracy, decreases. This is exactly what we expect since the larger value of $\beta$ is, the more concentrated the distribution becomes. Based on~\cite{Borot2012,Dumaz2013}, one expects $\tilde{x}_{0}\approx C/\beta^{\frac {2}{3}}$ for some constant $C$. This suggests a way to choose the optimal value of $x_{0}$ in terms of accuracy and computation time. For our algorithms, we choose $C=13$ since for $\beta=1$, $F_{\beta}(13)\approx 1$. In practice, instead of using $\beta^{\frac {2}{3}}$ for the denominator, we are more conservative and use $\beta^{\frac {1}{2}}$ since, based on Figure~\ref{p:x0}\,(b),
 when $\beta=4$, using $x_0=13/\beta^{\frac {2}{3}}\approx 5.16$ brings in an error of $\mathcal {O}\big(10^{-9}\big)$.

\subsection{Computation time}\label{s:22}
The values in Table~\ref{tab3} are obtained by running on a computer with processor: $11$th Gen Intel(R) Core(TM) i7-11800H $2.30$GHz with $16.0$GB of RAM.

\begin{table}[t]
 \centering
 \begin{tabular}{|c|c|c|c|c|}
 \hline
 \makebox[6em]{Discretization} &\makebox[4em]{Method}&\makebox[8em]{Time (default)} &\makebox[8em]{Error (default)}&\makebox[6em]{Time \big($10^{-6}$\big)} \\ \hline
 \multirow{5}{4.2em}{\centering Finite difference} & Trapz & $12.684$s & $1.027(-7)$ & $0.203$s \\
 & BDF3 & $10.731$s & $5.461(-8)$ & $0.138$s \\
 & BDF4 & $10.458$s & $5.457(-8)$ & $0.142$s \\
 & BDF5 & $10.947$s & $5.538(-8)$ & $0.141$s \\
 & BDF6 & $11.067$s & $5.539(-8)$ & $0.140$s \\ \hline
 \multirow{5}{4em}{\centering Spectral} & Trapz & $222.135$s & $4.790(-8)$ & $10.906$s \\
 & BDF3 & $205.888$s & $4.298(-11)$ & $1.949$s \\
 & BDF4 & $209.175$s & $2.346(-12)$ & $1.388$s \\
 & BDF5 & $201.819$s & $2.663(-12)$ & $1.008$s \\
 & BDF6 & $204.295$s & $2.677(-12)$ & $204.295$s \\\hline
 \end{tabular}

 \caption{For each discretization in the first column and time-stepping method in the second column, computation time to get the interpolated cdf for the Tracy--Widom distribution and the corresponding error at $x=-2$ for $\beta=2$ are recorded in the next two columns. The computation times in the third column are generated using the default parameters as in~\eqref{eqn:paraf} and~\eqref{eqn:paras}, with the corresponding errors in the fourth column. If we aim for an error of $\mathcal {O}\big(10^{-6}\big)$, the last column shows the corresponding minimum observed computation times for each discretization and time-stepping method to achieve this error.}
 \label{tab3}
\end{table}

Table~\ref{tab3} shows the computation time with default values of the parameters as in~\eqref{eqn:paraf} and~\eqref{eqn:paras} for $\beta=2$ along with the corresponding error at $x=-2$. The last column provides the computation time if we aim for an error of $\mathcal {O}\big(10^{-6}\big)$ at $x=-2$. For finite-difference discretization, to have the error of $ \mathcal {O}\big(10^{-6}\big)$ at $x=-2$, $\Delta x=-0.01$ can be used instead of $\Delta x=-0.001$. For spectral discretization, to have the error of $ \mathcal {O}\big(10^{-6}\big)$ at $x=-2$, $M=4\times 10^3$ can be used instead of $M=8\times 10^3$ with $\Delta x=-0.01$ for trapezoidal method, $\Delta x=-0.05$ for BDF3, $\Delta x=-0.07$ for BDF4, and $\Delta x=-0.1$ for BDF5. It turns out that the error using BDF6 will jump from $\mathcal {O}\big(10^{-5}\big)$ to $\mathcal {O}\big(10^{-12}\big)$. From Table~\ref{tab3}, it takes about $204.295$s using BDF6 to have the error of $\mathcal {O}(10^{-12})$ with parameters as in~\eqref{eqn:paras}. To have the error of $\mathcal {O}\big(10^{-5}\big)$, it takes about $0.639$s with $M=5\times 10^3$ and $\Delta x=-0.2$. Work to improve the speed of the spectral discretization is ongoing.\looseness=-1

Using the same machine, Bloemendal's code (see~\cite[Section 6]{Bloemendal2011FiniteRP}), which uses \textsc{Mathematica}'s \texttt{NDSolve}, takes approximately $3$ seconds to compute the approximation to the function $H$. We point out that \textsc{Mathematica} 13.2, generates warnings from \texttt{NDSolve} regarding the errors in the approximate solution. The accuracy of this approach, while simple, appears to be limited to a~maximum of 7 digits. Yet, the fact that this is indeed so simple, and works, is an important reminder of the conceptual simplicity of this representation of the Tracy--Widom distribution function.

Our methods expand upon this, providing both the cdfs and pdfs as output and producing high-accuracy interpolants while leaving the trade-off between accuracy and speed to the user's discretion. And since we consider the general-$\beta$ Tracy--Widom distribution functions as important nonlinear special functions, developing methods, with the highest possible accuracy, is critical. See Section~\ref{sec:outlook} for some thoughts on further improvements.

\section{Additional numerical results}\label{s:results}

In this section, we present additional plots to demonstrate the power and flexibility of the code.

\subsection{Comparison with large random matrices}

We verify numerically that the pdf generated by our algorithm agrees with the model presented by Dumitriu and Edelman~\cite{article}. Recall that $H_{n}^{\beta}$ in~\eqref{eqn:matrix} has~\eqref{eqn:density} as the jpdf for its eigenvalues, and the distribution of its largest eigenvalue, after rescaling, converges to $F_{\beta}$ for any $\beta>0$ as $n\to \infty$.

Histograms in Figure~\ref{fig:veri} are the normalized histograms for $n^{1/6}\big(\lambda_{\text{max}}\big(H^{\beta}_{n}\big)-2\sqrt{n}\big)$, where $\lambda_{\text{max}}\big(H^{\beta}_{n}\big)$ denotes the largest eigenvalue of the
$\beta$-Hermite ensemble $H^{\beta}_{n}$.

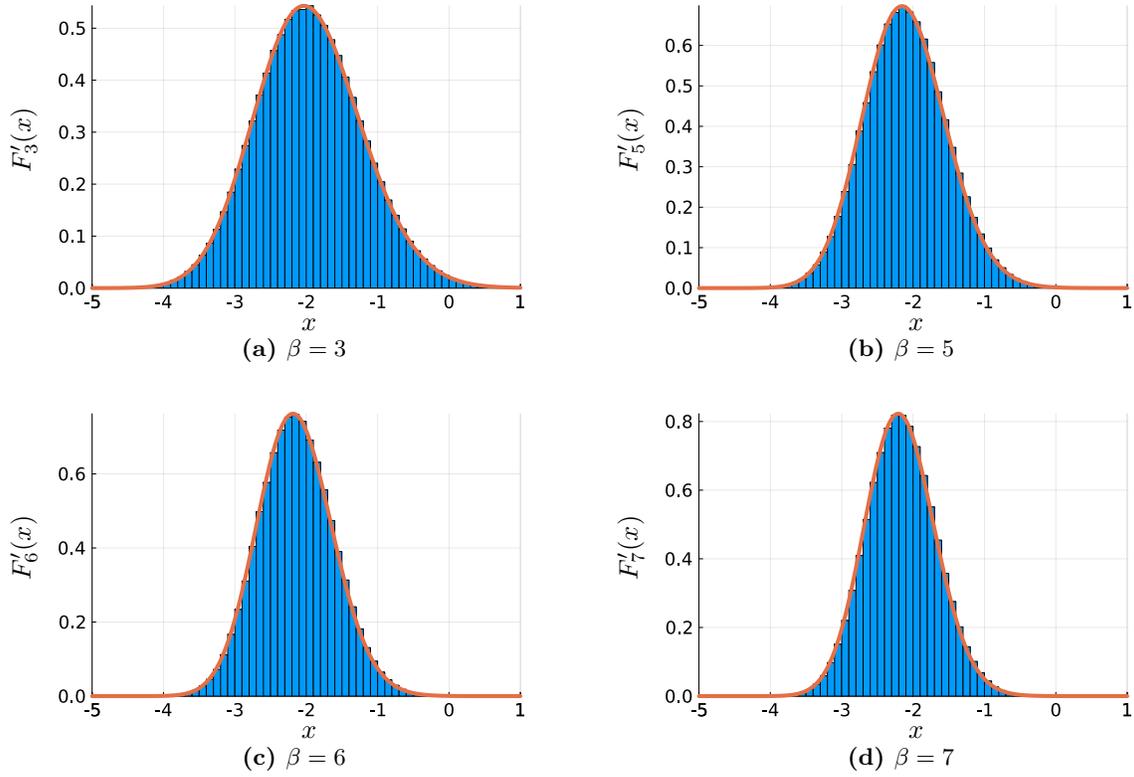
\begin{figure}[th!]\vspace*{-3mm}
 \centering
 \subfloat[\vspace*{-5mm}\centering $\beta=3$] {{\includegraphics[width=0.40\linewidth]{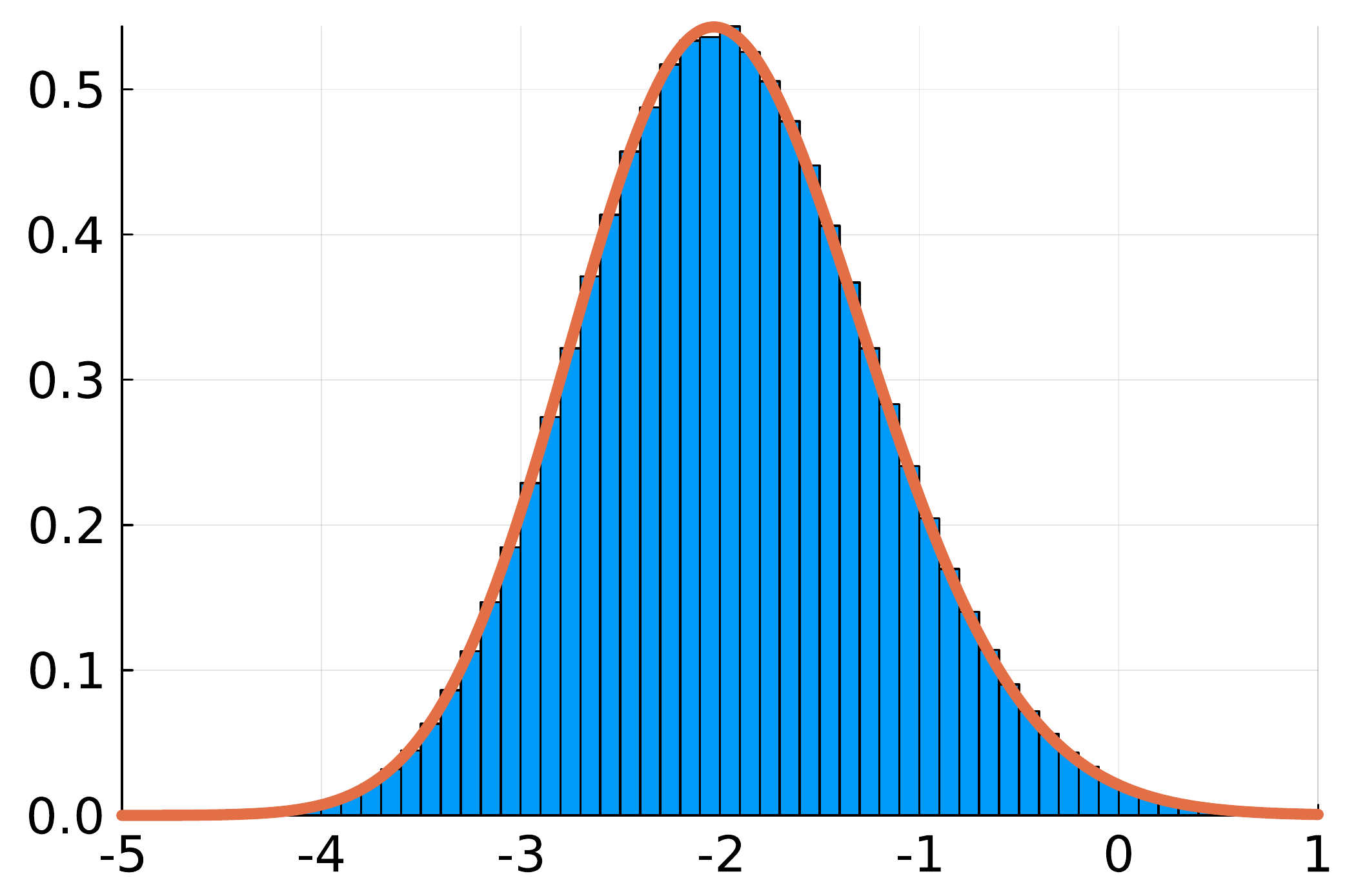} }}%
 \put(-200,55){\rotatebox{90}{\color{black}\small $F_{3}'(x)$}}
 \put(-90,-5){\color{black}\small $x$}
 \qquad\qquad
 \vspace*{6mm}
 \subfloat[\vspace*{-5mm}\centering $\beta=5$ ]{{\includegraphics[width=0.40\linewidth]{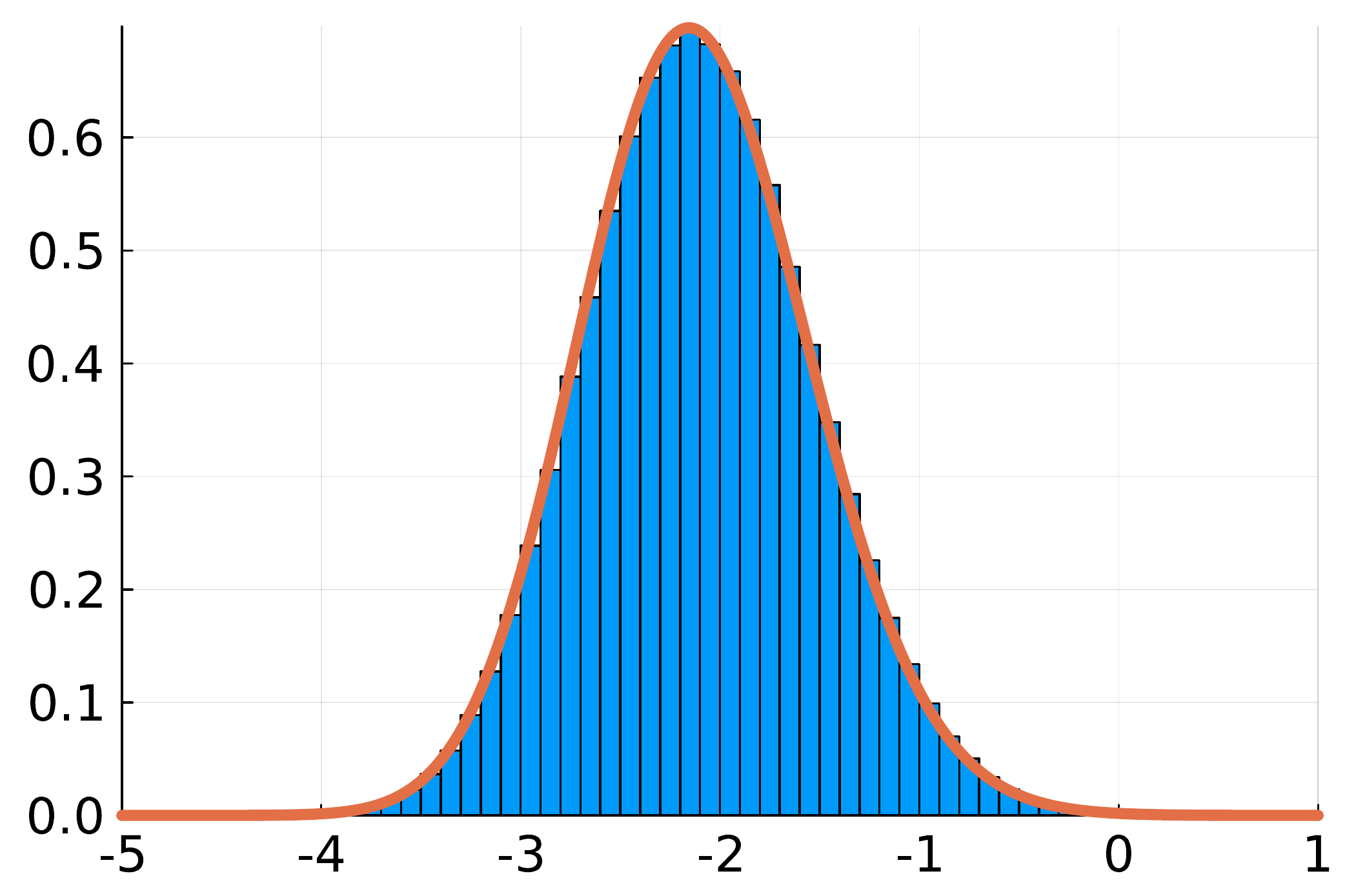} }}%
 \put(-200,55){\rotatebox{90}{\color{black}\small $F_{5}'(x)$}}
 \put(-90,-5){\color{black}\small $x$}
 \qquad
 \vspace*{6mm}
 \subfloat[\vspace*{-5mm}\centering $\beta=6$ ]{{\includegraphics[width=0.40\linewidth]{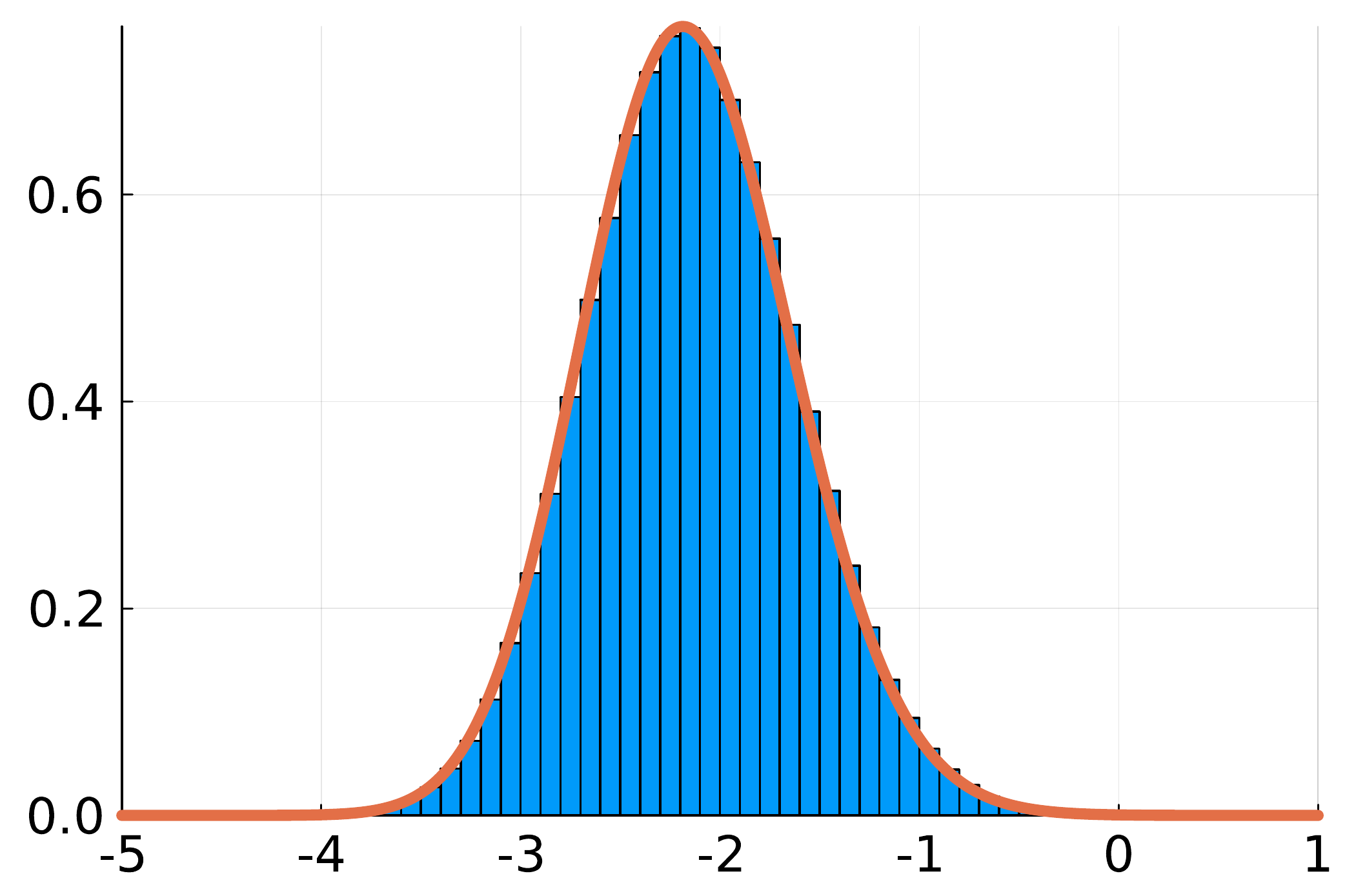} }}%
 \put(-200,55){\rotatebox{90}{\color{black}\small $F_{6}'(x)$}}
 \put(-90,-5){\color{black}\small $x$}
 \qquad\qquad
 \subfloat[\vspace*{-5mm}\centering $\beta=7$ ]{{\includegraphics[width=0.40\linewidth]{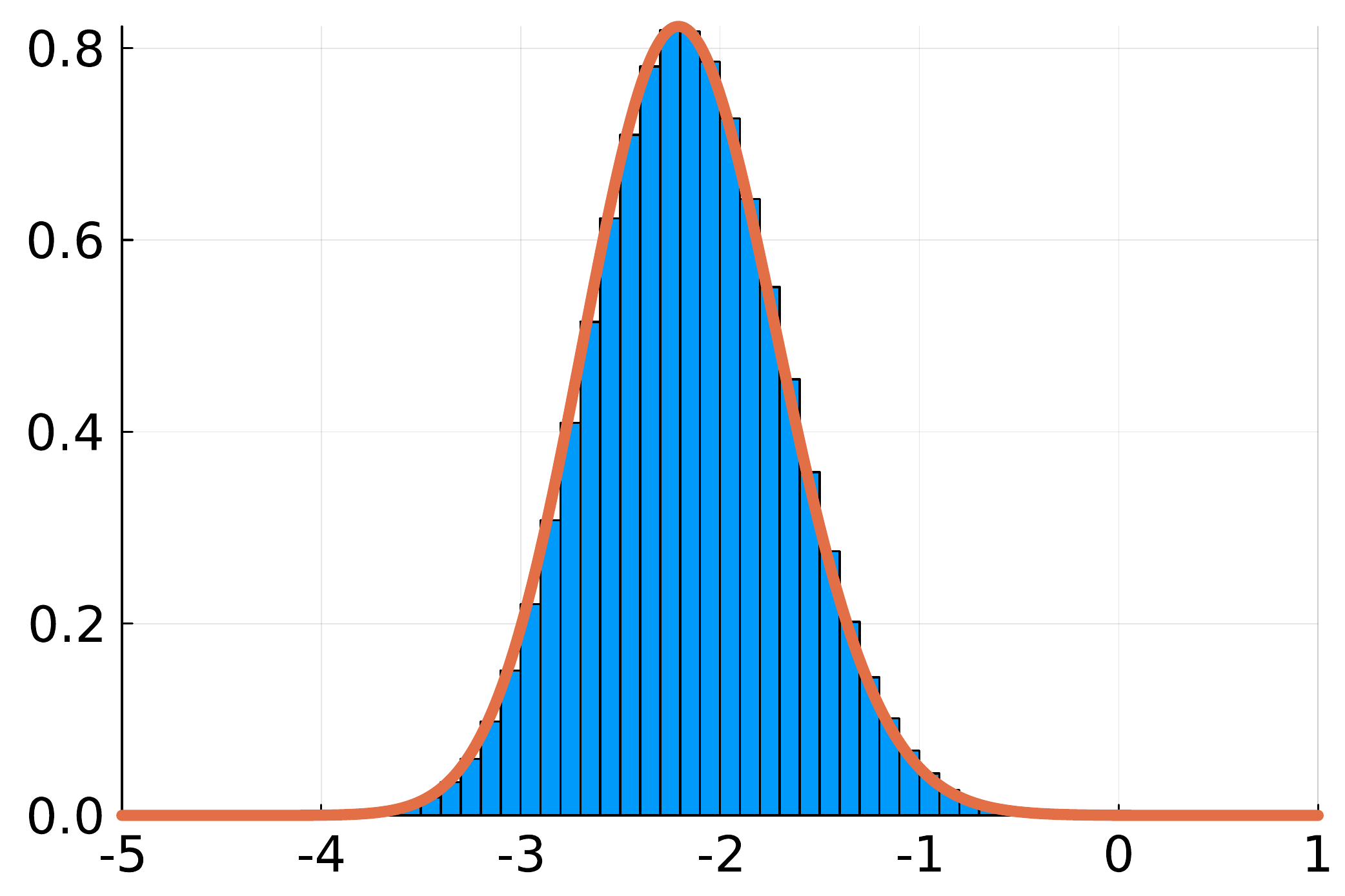} }}%
 \put(-200,55){\rotatebox{90}{\color{black}\small $F_{7}'(x)$}}
 \put(-90,-5){\color{black}\small $x$}
 \caption{The histograms for $n^{1/6}\big(\lambda_{\text{max}}\big(H^{\beta}_{n}\big)-2\sqrt{n}\big)$ using $10^6$ samples with $n=10^4$. The density~$F'_{\beta}(x)$ is generated by \texttt{TW}$(\beta; \texttt{pdf=true})$ for $\beta=3,5,6,7$. The small positive bias in each histogram can be improved by using larger values of $n$.}\label{fig:veri}
\end{figure}

\begin{figure}[th!]\vspace*{-5mm}
 \centering
 \subfloat[\centering $\beta=3$] {{\includegraphics[width=0.40\linewidth]{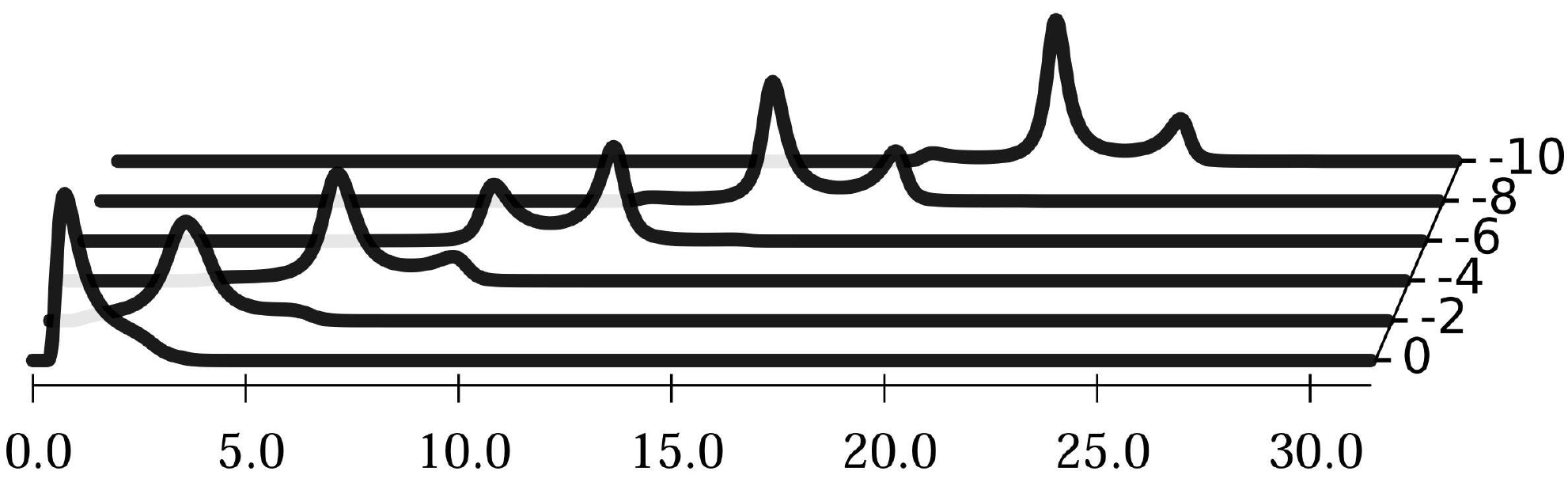} }}%
 \put(-120,-2){\color{black}\small $\theta$}
 \put(-6,28){\color{black}\small $x$}
 \qquad\qquad
 \subfloat[\centering $\beta=5$ ]{{\includegraphics[width=0.40\linewidth]{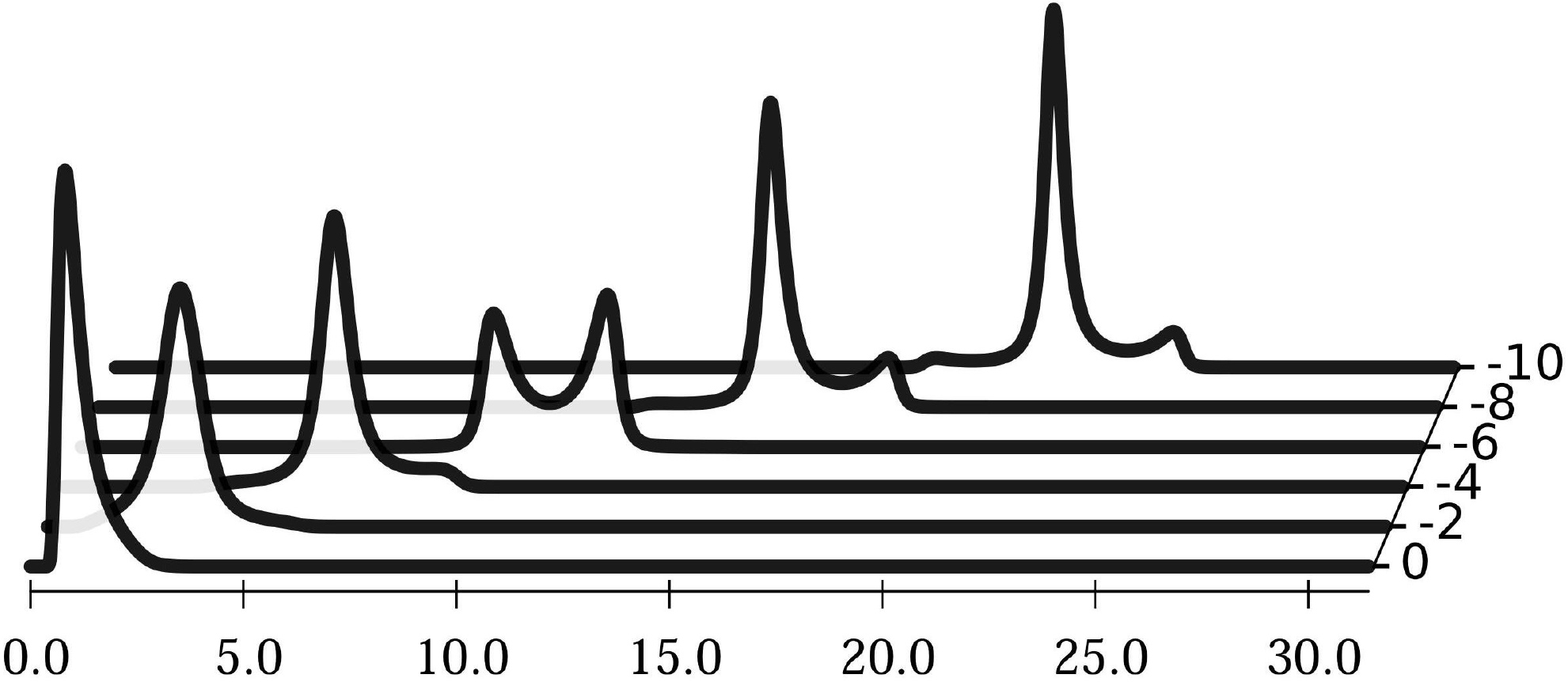} }}%
 \put(-120,-2){\color{black}\small $\theta$}
 \put(-6,28){\color{black}\small $x$}
 \qquad
 \subfloat[\centering $\beta=6$ ]{{\includegraphics[width=0.40\linewidth]{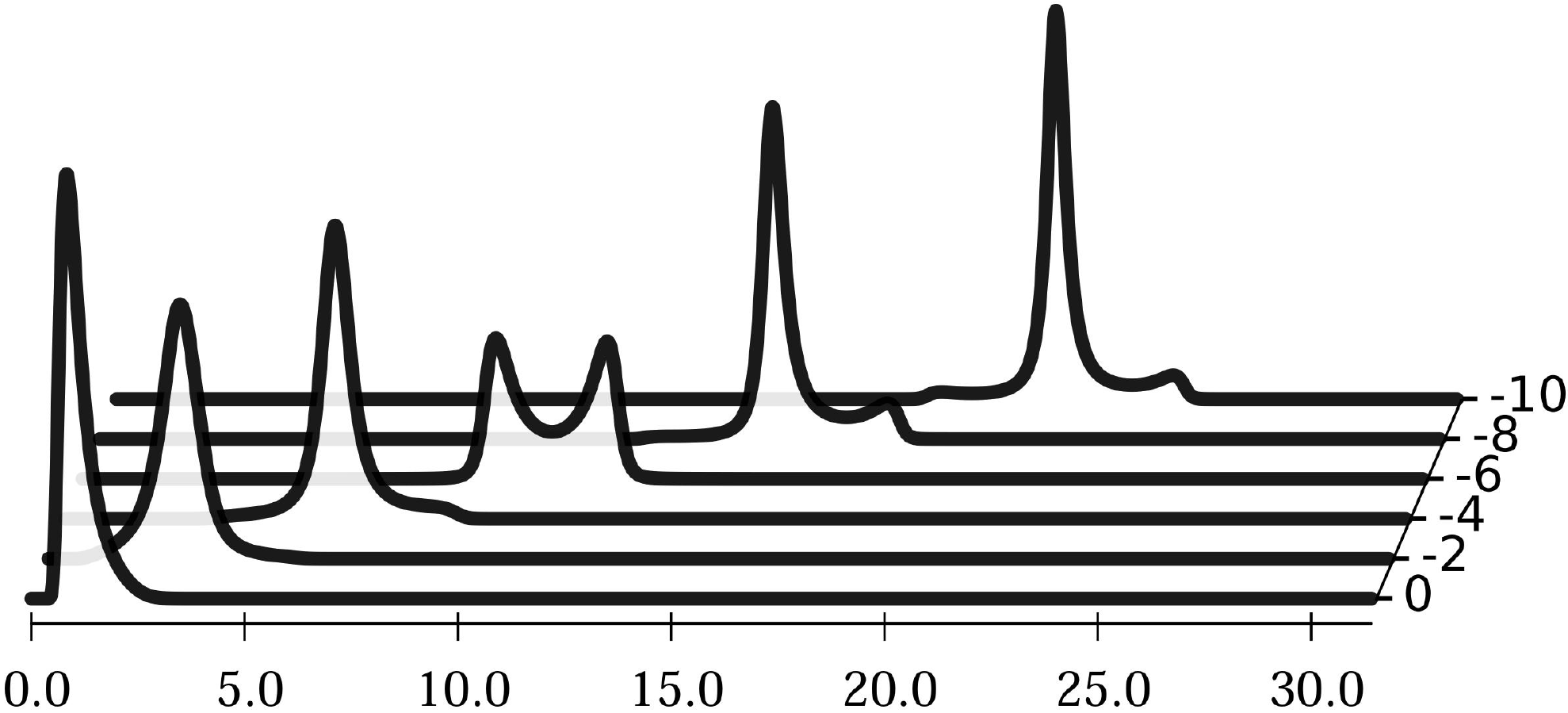} }}%
 \put(-120,-2){\color{black}\small $\theta$}
 \put(-6,28){\color{black}\small $x$}
 \qquad\qquad
 \subfloat[\centering $\beta=7$ ]{{\includegraphics[width=0.40\linewidth]{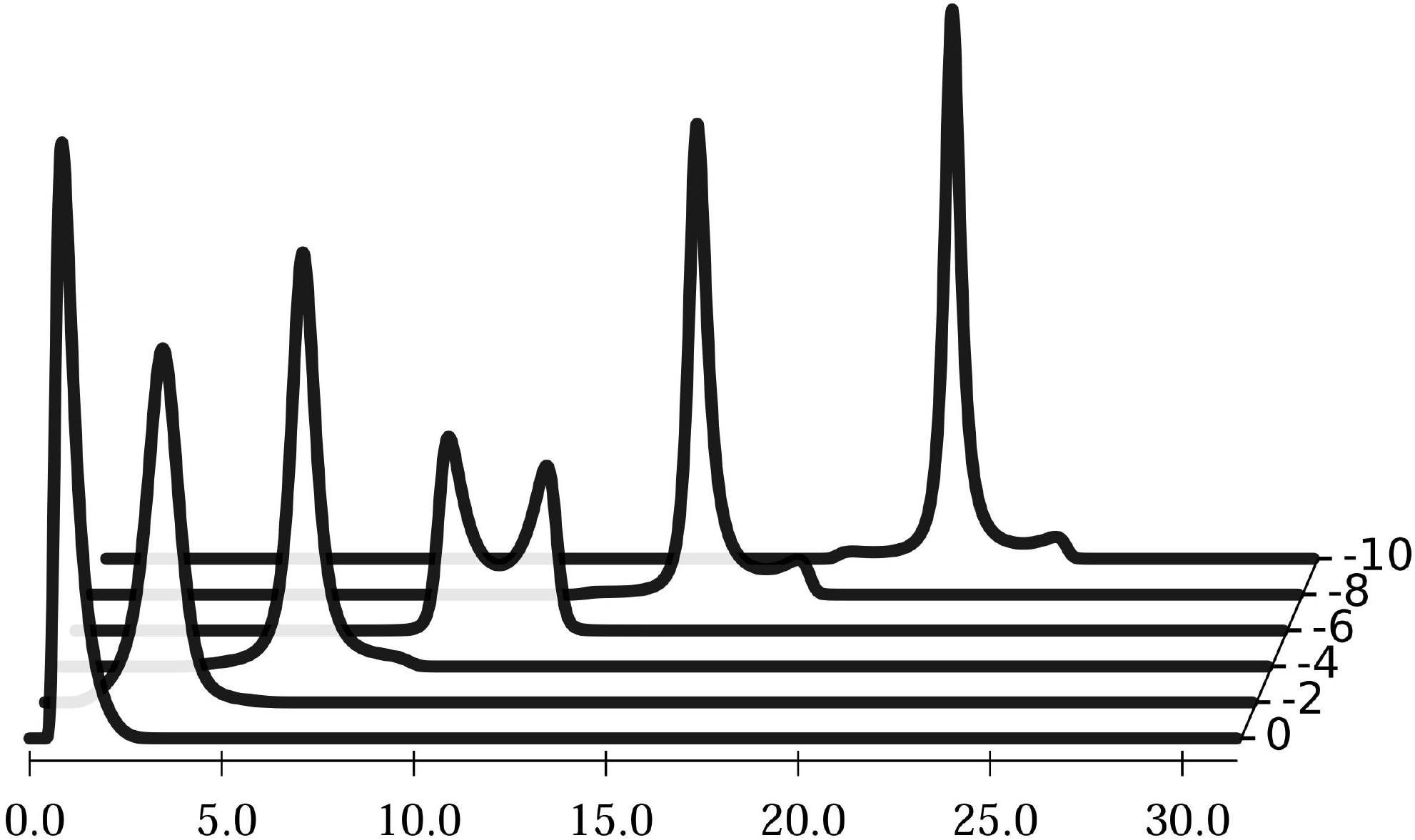} }}%
 \put(-120,-2){\color{black}\small $\theta$}
 \put(-6,28){\color{black}\small $x$}
 \caption{The evolution of $\rho(x,\theta):=\frac {\partial H}{\partial \theta}$ as $x$ decreases from $x=0$ to $-10$ with step size $=-2$ for $\beta=3,5,6,7$. Finite-difference discretization with trapezoidal method is used on~\eqref{spect} with $\Delta x=-10^{-3}$ and $M=10^4$ for $\theta\in [0,10\pi]$.}\label{fig:pp}
\end{figure}
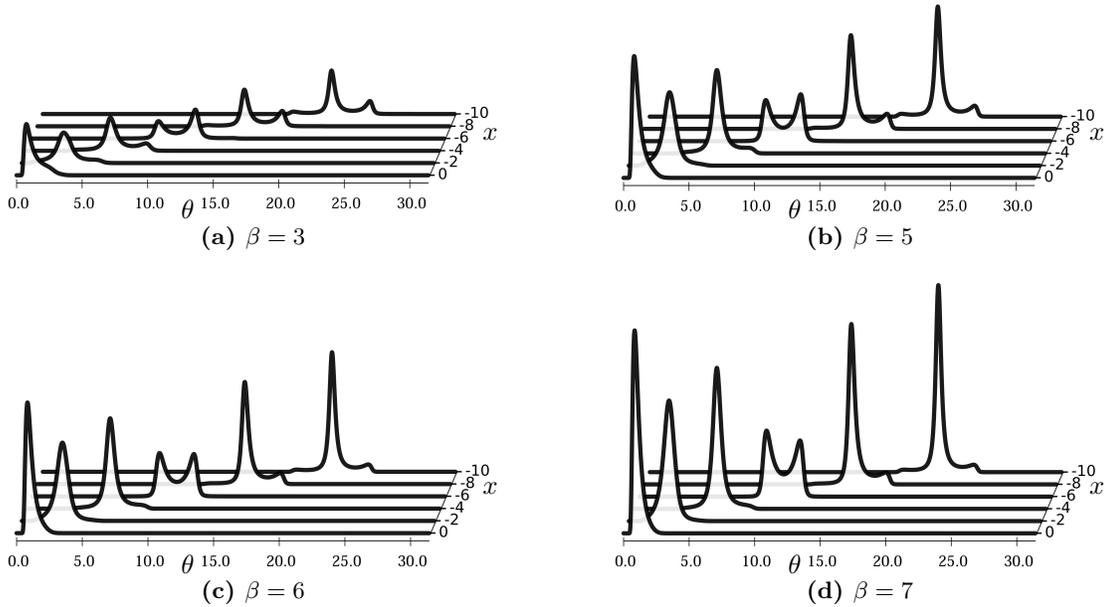

\subsection[Evolution of the density rho(x,theta):=partial H(x,theta)/partial theta]{Evolution of the density $\boldsymbol{\rho(x,\theta):=\partial H(x,\theta)/\partial \theta}$}

Figure~\ref{fig:pp} shows the waterfall plots of the approximation of $\rho(x,\theta):=\partial {H}(x,\theta)/\partial{\theta}$ for $\theta\in [0,10\pi]$ and $x=0,-2,\dots,-10$ using finite-difference discretization with trapezoidal method on~\eqref{spect} with $\theta_{M}=10\pi$ and $M=10^4$. Values of the other parameters are the same as in~\eqref{eqn:paraf}. As the value of $x$ decreases, the density of the solution propagates to the right.

Figure~\ref{fig:ppp} shows the contour plots of the approximation of $H(x,\theta)$ using finite-difference discretization with trapezoidal method with values of the parameters given as in~\eqref{eqn:paraf}. For each contour plot, the initial condition $H(x_{0},\theta)$ is found along the top of the plot and the approximate Tracy--Widom distribution, $F_{\beta}(x)\approx H(x,\theta_{M})$, is obtained on the right-hand side of the plot.

\subsection{Density of the Tracy--Widom distribution}
Figure~\ref{p:p}\,(a) shows the plot of the approximation of $F_{\beta}'$ for $\beta=1$ to $10$ and $x\in[-6,4]$ using finite-difference discretization with trapezoidal method with values of the other parameters as in~\eqref{eqn:paraf}. As we can see from the plot, as the value of $\beta$ increases, $F_{\beta}'$ becomes more concentrated, and its peak moves leftwards (See Appendix~\ref{apx:1} for a discussion of the exact limiting behavior as $\beta\to \infty$).

Figure~\ref{p:p}\,(b) is a two-dimensional version of Figure~\ref{p:p}\,(a), which is also generated using the finite-difference discretization with trapezoidal method. It provides a closer view of $F_{\beta}'$ for $\beta=1$ to $4$ with step size $=0.2$. The red curves from right to left correspond to $\beta=1,2,4$ respectively. The black curves show $F_{\beta}'$ for the other values of $\beta$.

\begin{figure}[th!]\vspace{-2mm}
 \centering
 \subfloat[\vspace*{-5mm}\centering $\beta=3$] {{\includegraphics[width=0.43\linewidth]{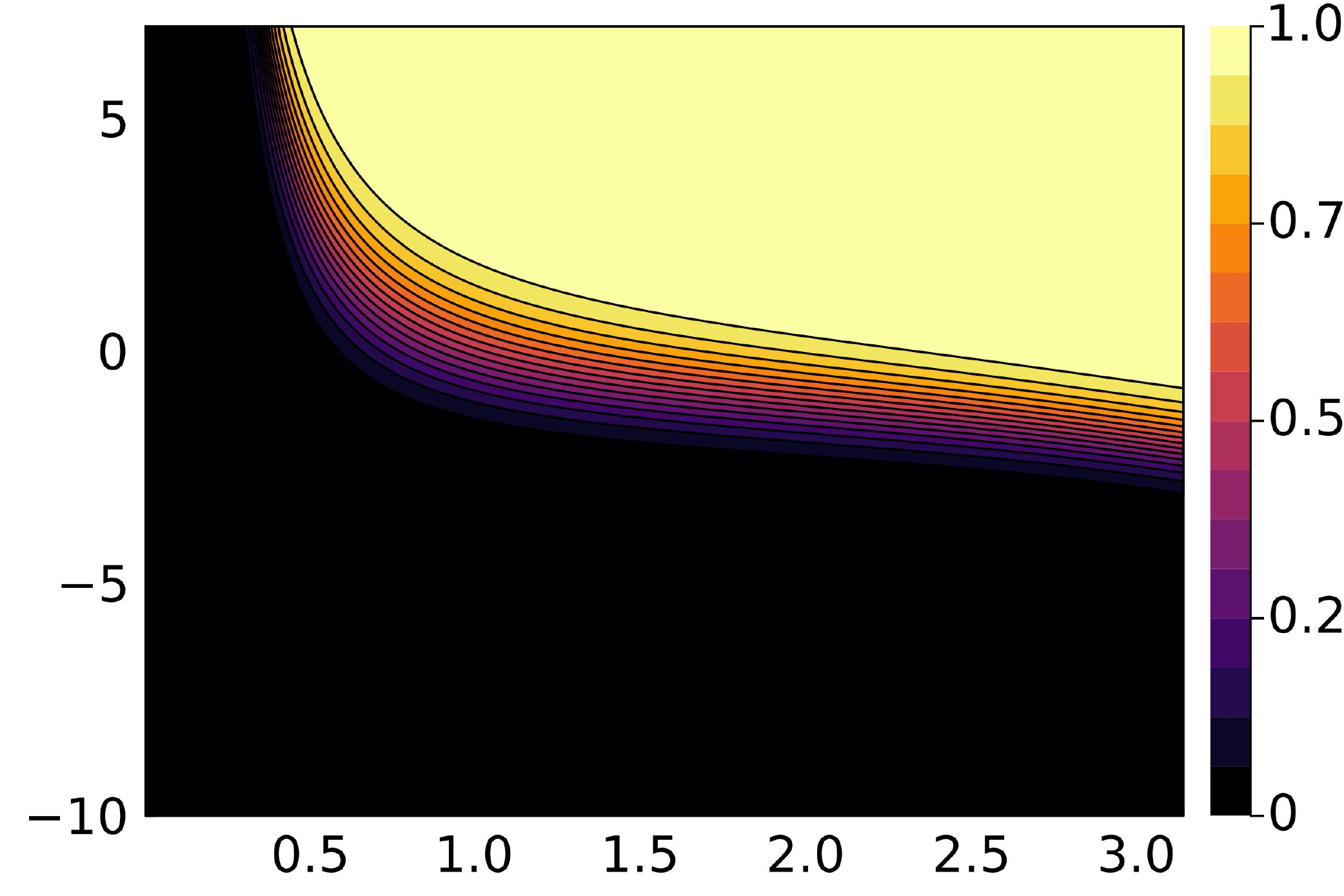} }}%
 \put(-98,-5){\color{black}\small $\theta$}
 \put(-200,63){\rotatebox{90}{\color{black}\small $x$}}
 \qquad
 \vspace*{5mm}
 \subfloat[\vspace*{-5mm}\centering $\beta=5$ ]{{\includegraphics[width=0.43\linewidth]{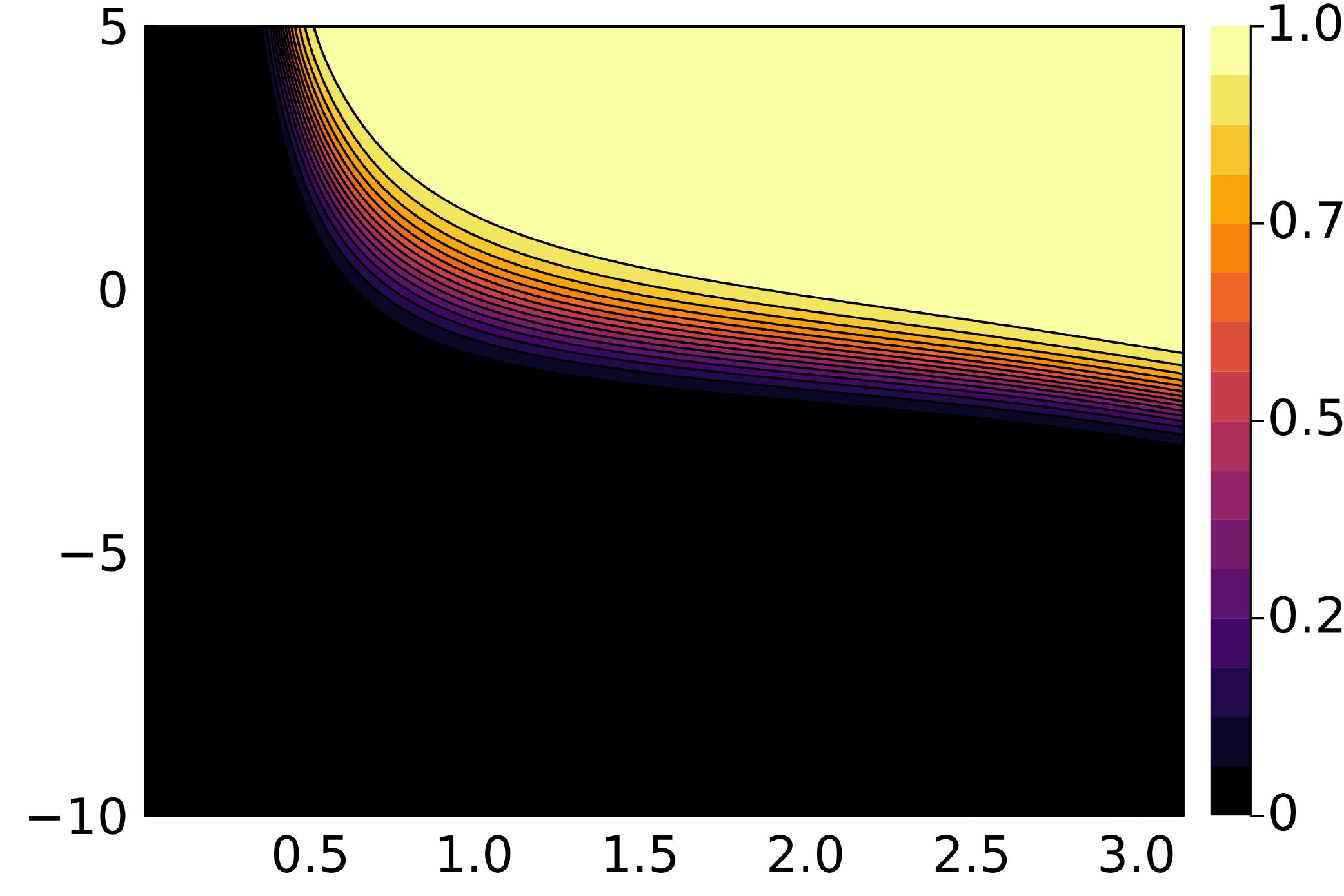} }}%
 \put(-98,-5){\color{black}\small $\theta$}
 \put(-200,63){\rotatebox{90}{\color{black}\small $x$}}
 \qquad
 \subfloat[\vspace*{-5mm}\centering $\beta=6$ ]{{\includegraphics[width=0.43\linewidth]{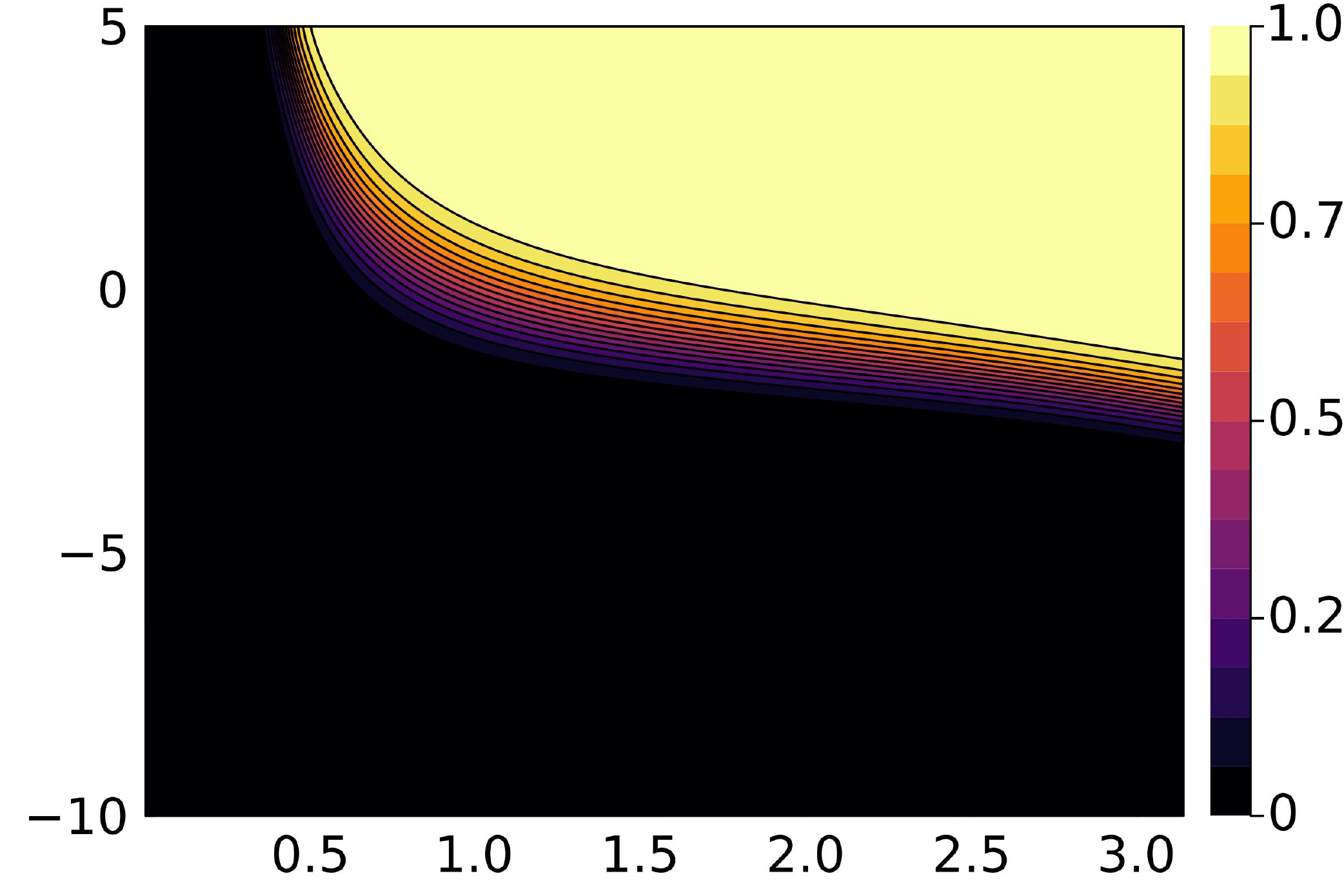} }}%
 \put(-98,-5){\color{black}\small $\theta$}
 \put(-200,63){\rotatebox{90}{\color{black}\small $x$}}
 \qquad
 \vspace*{5mm}
 \subfloat[\vspace*{-5mm}\centering $\beta=7$ ]{{\includegraphics[width=0.43\linewidth]{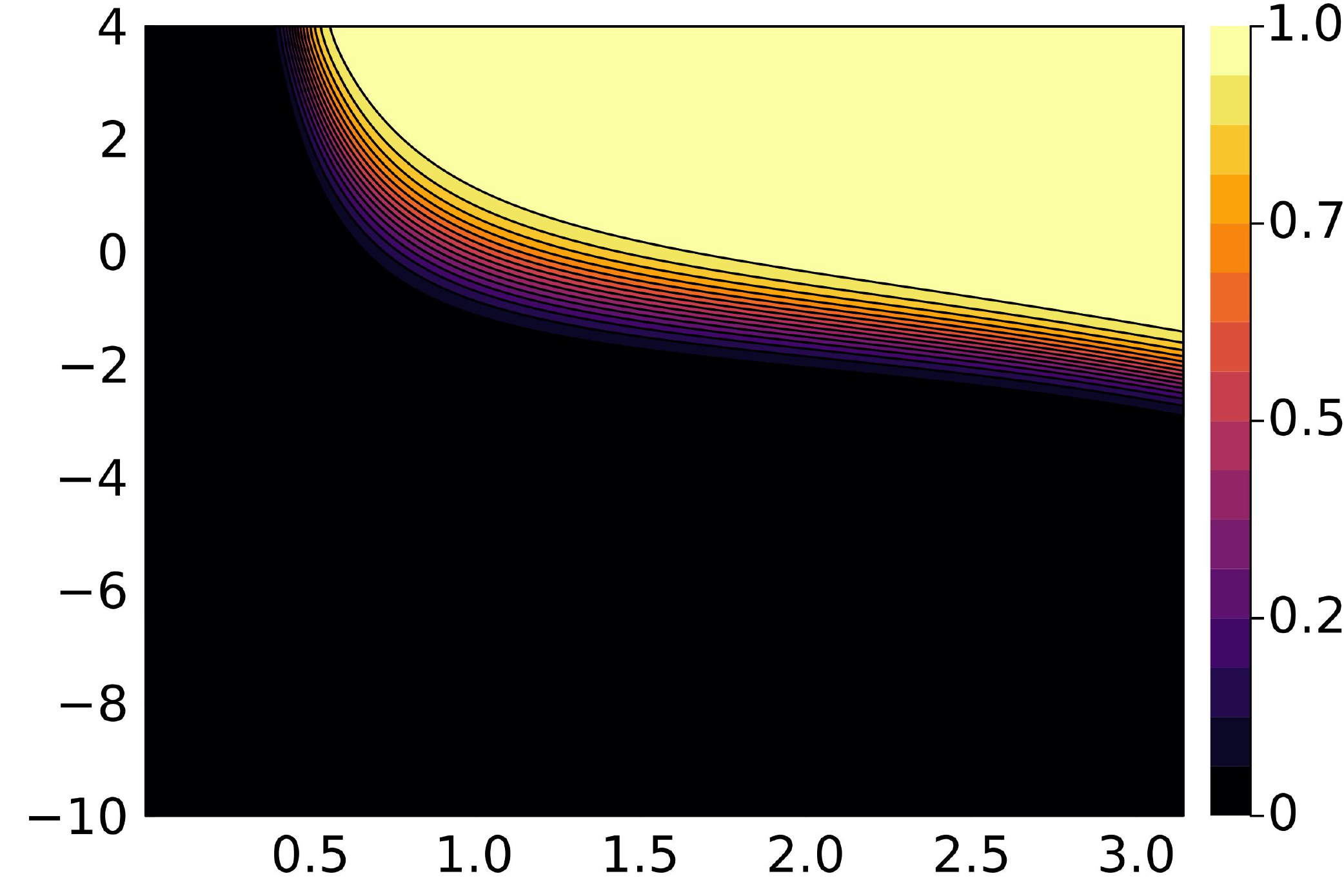} }}%
 \put(-98,-5){\color{black}\small $\theta$}
 \put(-200,63){\rotatebox{90}{\color{black}\small $x$}}
 \caption{The evolution of $H(x,\theta)$ as $x$ decreases from $\lfloor 13/\sqrt{\beta} \rfloor$ to $-10$ for $\beta=3,5,6,7$. Finite-difference discretization with trapezoidal method is used with $\Delta x=-10^{-3}$ and $M=10^3$.}%
 \label{fig:ppp}
 \end{figure}

\begin{figure}[th!]\vspace{-7mm}
 \centering
 \subfloat[\centering $\beta=1$ to $10$] {{\raisebox{2mm}{\includegraphics[width=0.44\linewidth]{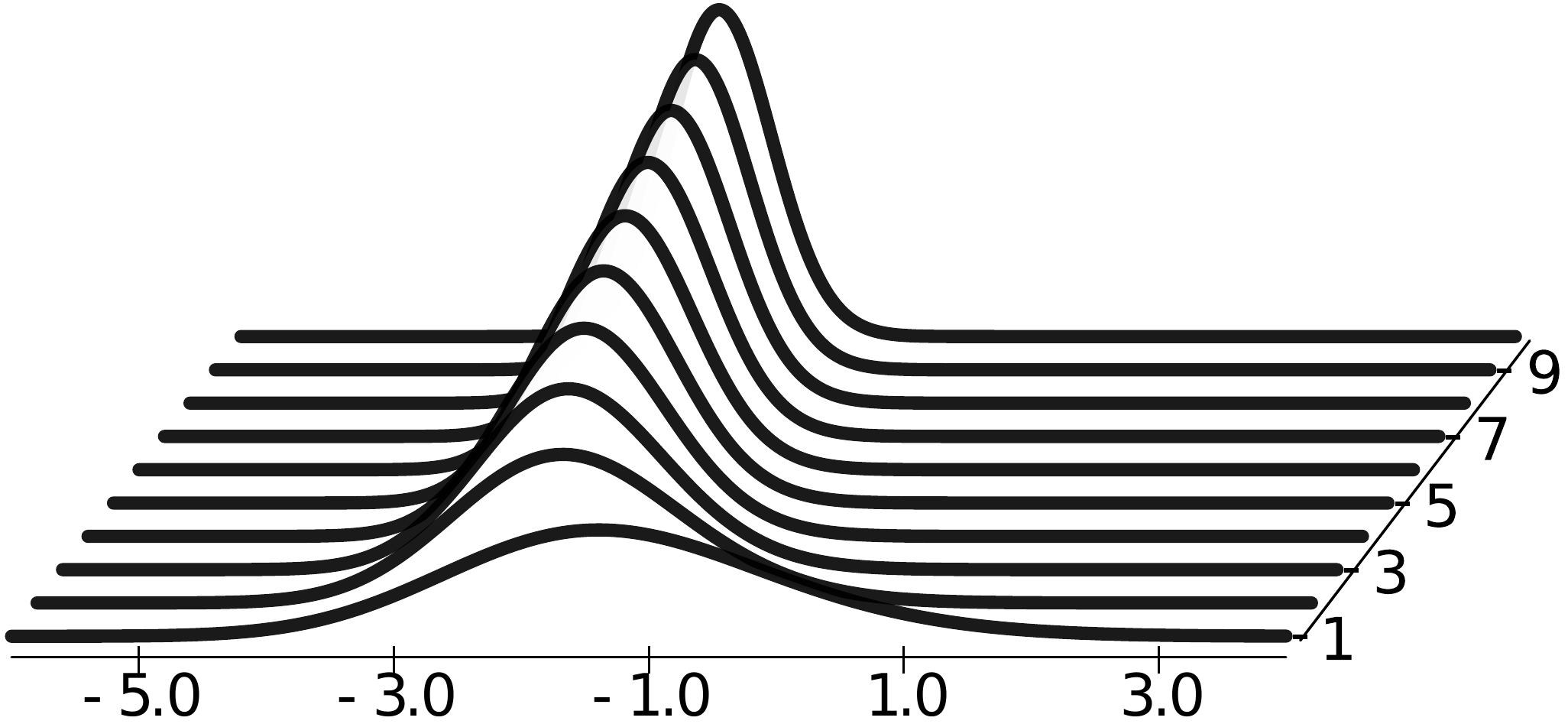}}}}%
 \put(-120,0){\color{black}\small $x$}
 \put(-10,32){\color{black}\small $\beta$}
 \qquad\
 \subfloat[\vspace*{-5mm}\centering $\beta=1$ to $4$]{{\raisebox{1mm}{\includegraphics[width=0.44\linewidth]{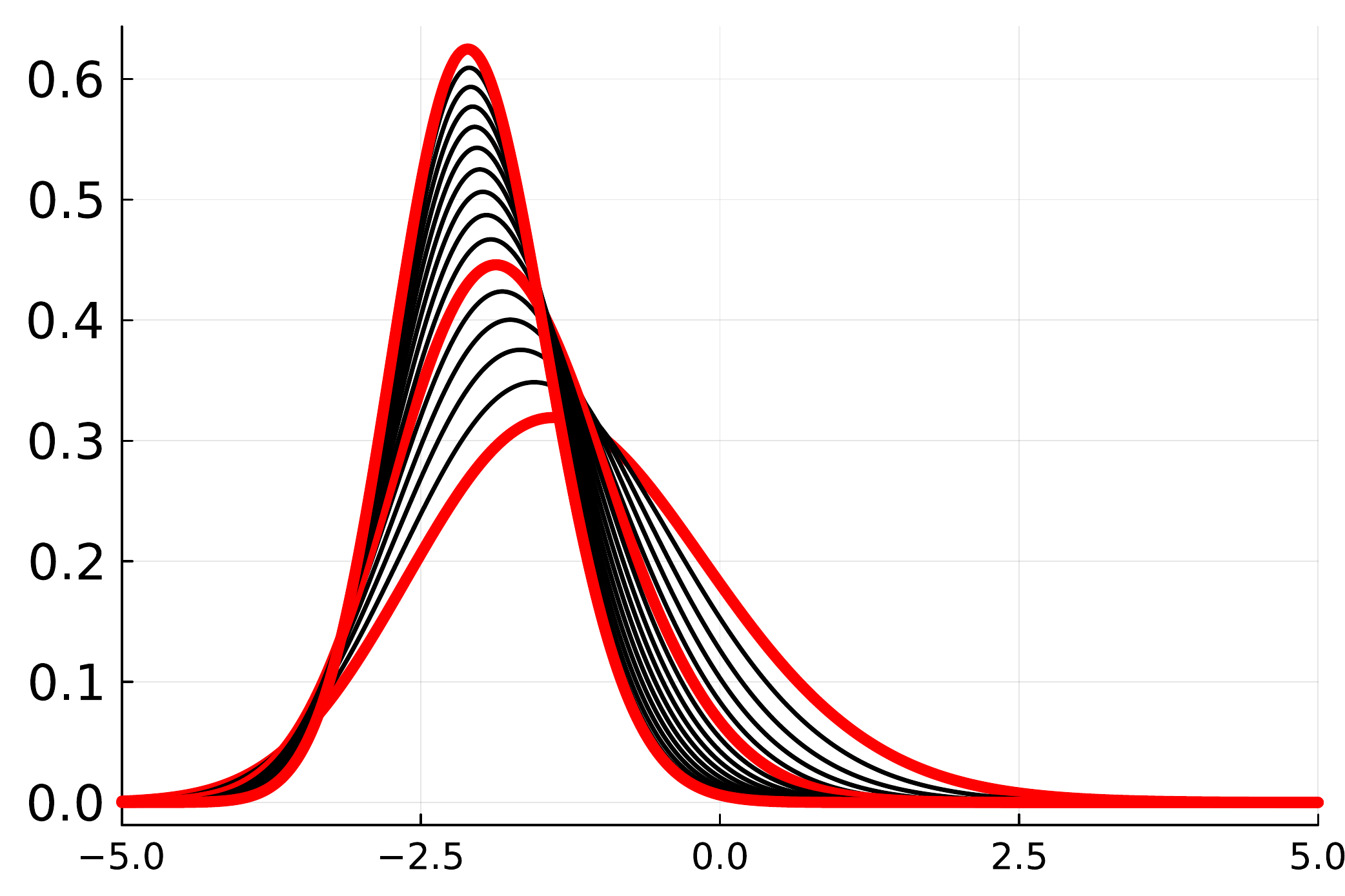}} }}%
 \put(-220,63){\rotatebox{90}{\color{black}\small $F_{\beta}'(x)$}}
 \put(-100,-4){\color{black}\small $x$}
 \caption{Approximation of the density of the Tracy--Widom distribution for different values of $\beta$ using the finite-difference discretization with trapezoidal method with values for the other parameters as in~\eqref{eqn:paraf}.}%
 \label{p:p}\vspace{-1mm}
\end{figure}

\subsection{Limiting densities of other eigenvalues}

\begin{figure}[th!]
 \centering
 \subfloat[\vspace*{-5.5mm}\centering $\beta=3$] {{\includegraphics[width=0.46\linewidth]{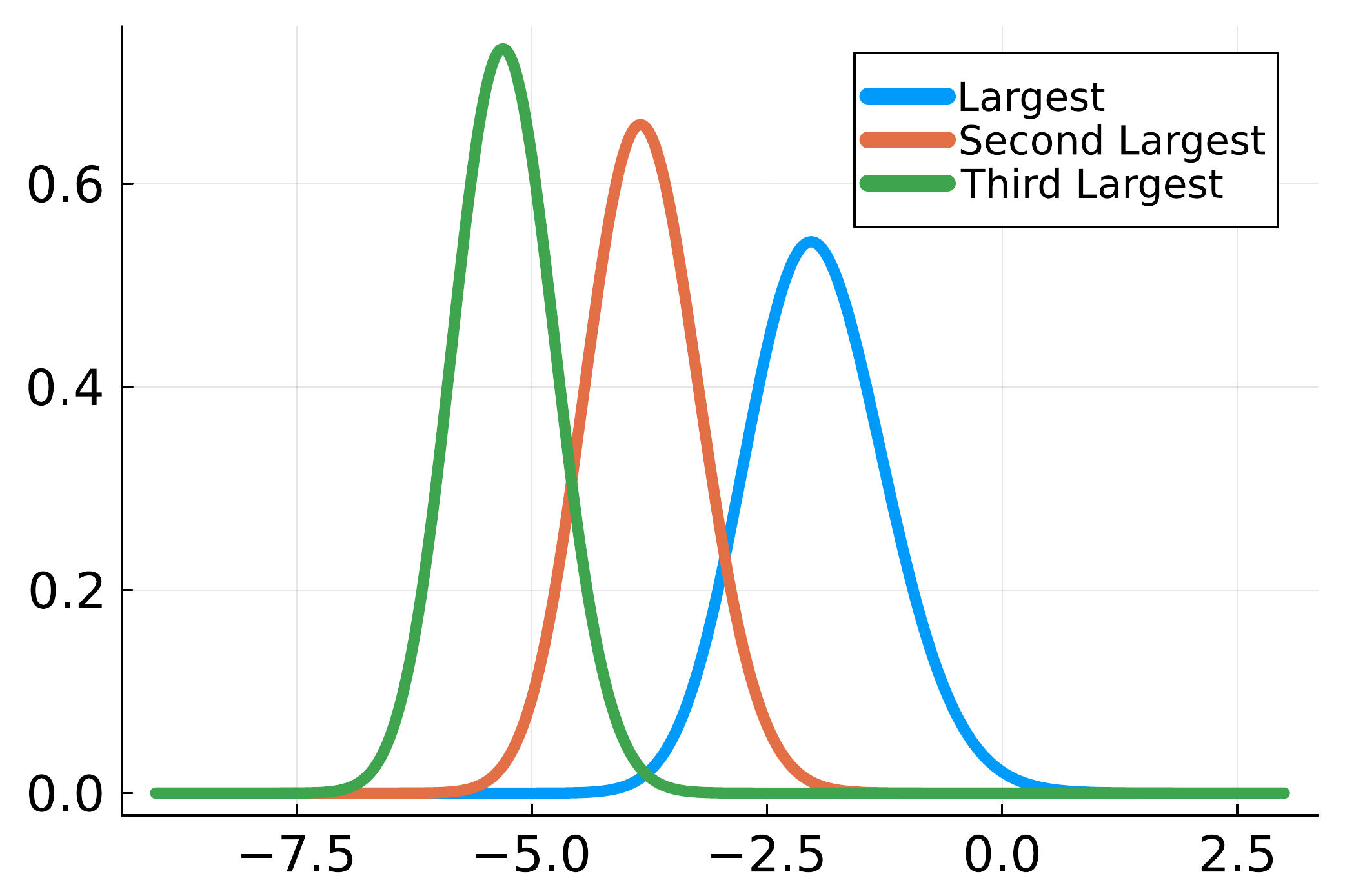} }}%
 \put(-103,-5){\color{black}\small $x$}
 \qquad
 \vspace*{5mm}
 \subfloat[\vspace*{-5.5mm}\centering $\beta=5$ ]{{\includegraphics[width=0.46\linewidth]{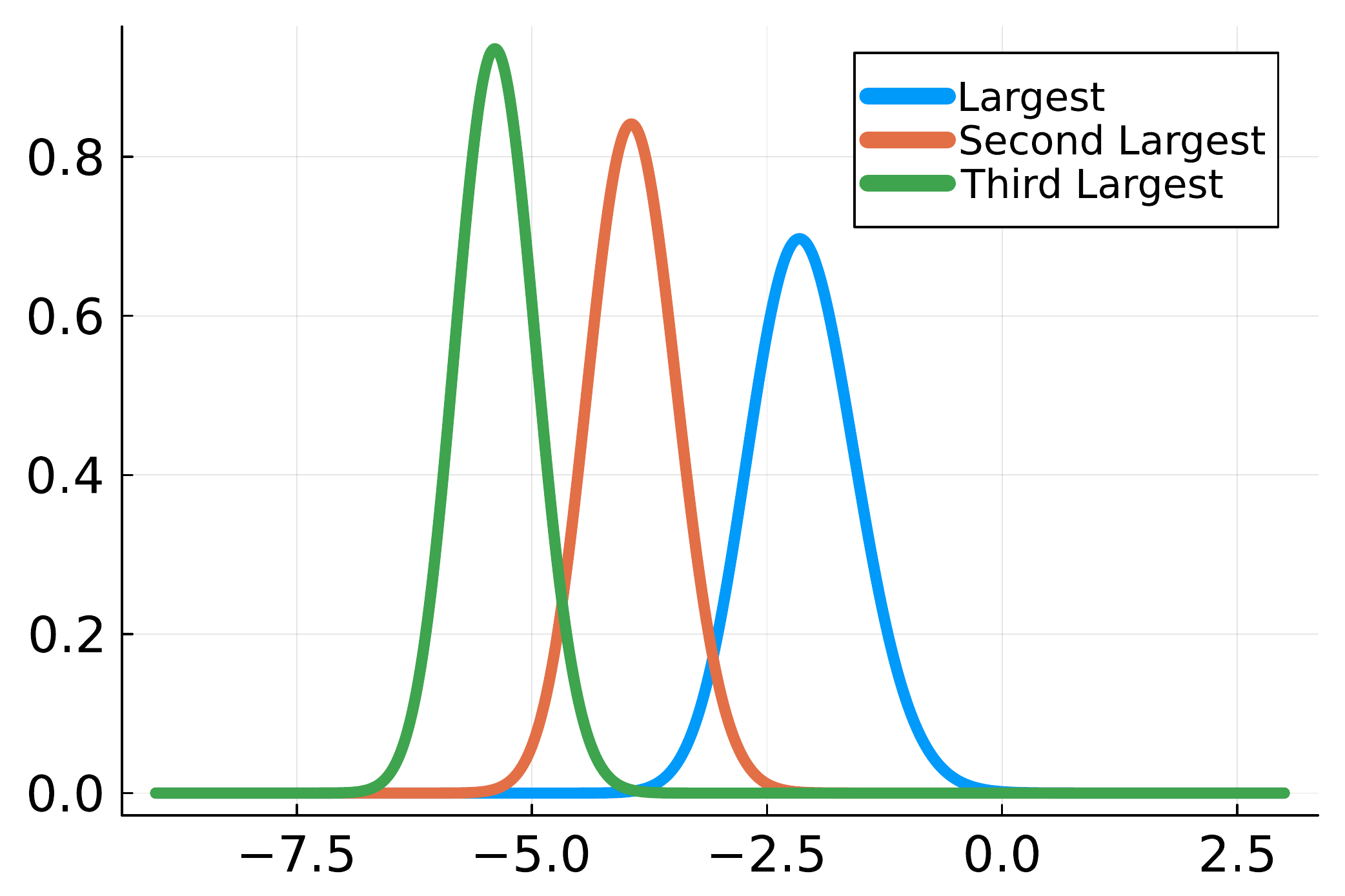} }}%
 \put(-103,-5){\color{black}\small $x$}
 \qquad
 \subfloat[\vspace*{-5.5mm}\centering $\beta=6$ ]{{\includegraphics[width=0.46\linewidth]{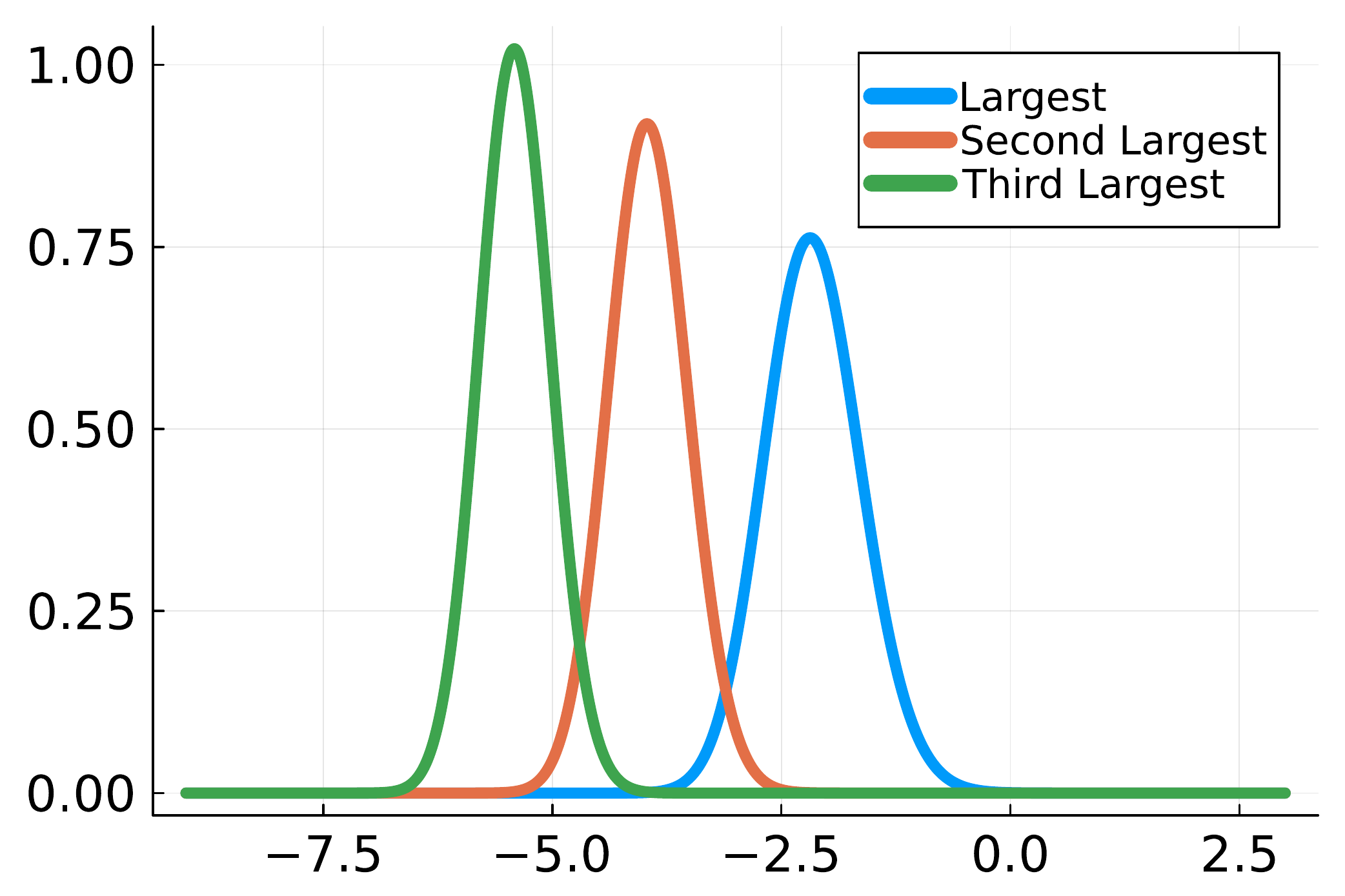} }}%
 \put(-103,-5){\color{black}\small $x$}
 \qquad
 \vspace*{5mm}
 \subfloat[\vspace*{-5.5mm}\centering $\beta=7$ ]{{\includegraphics[width=0.46\linewidth]{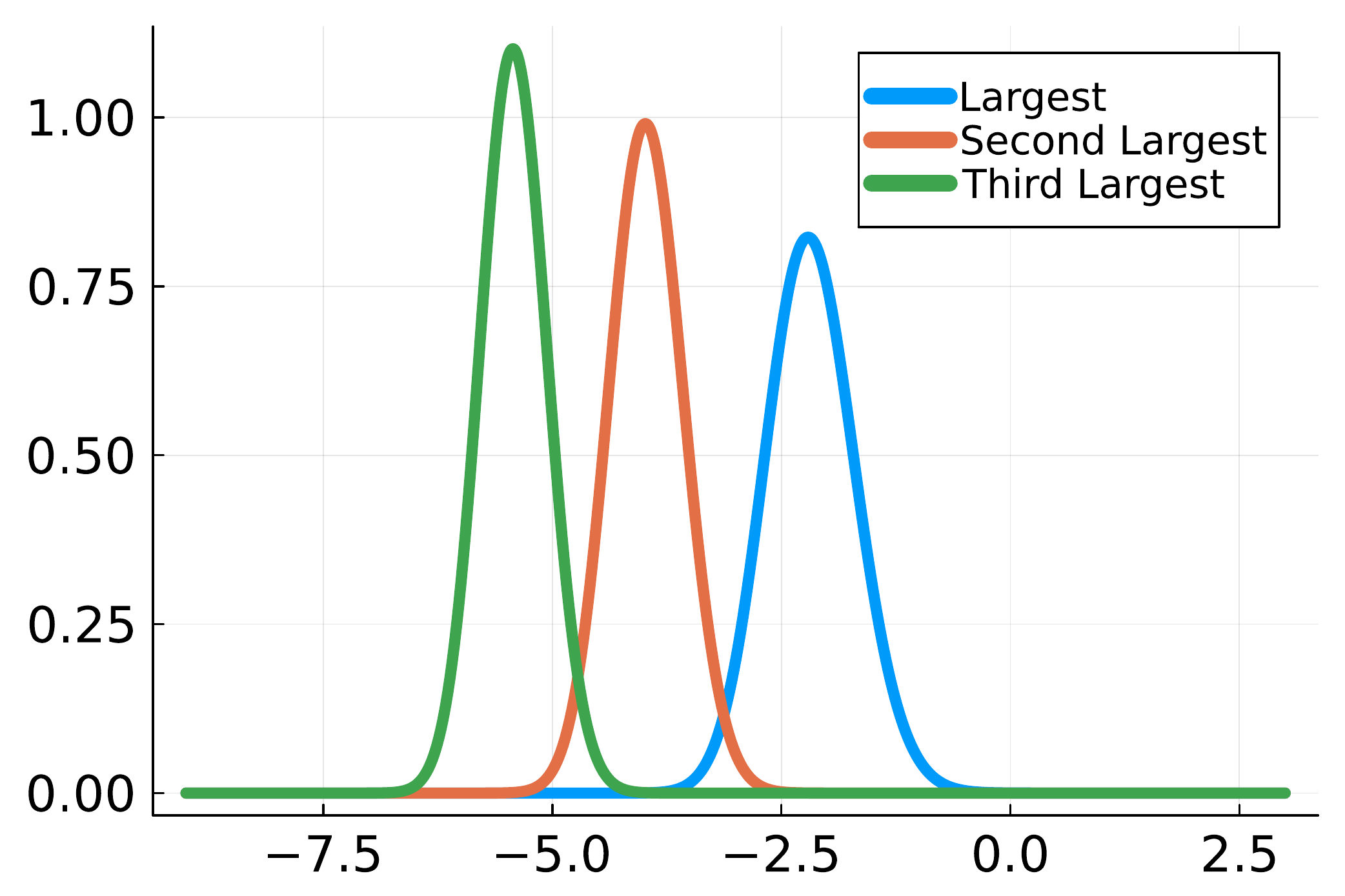} }}%
 \put(-103,-5){\color{black}\small $x$}
 \caption{Densities of $-\Lambda_{0}$, $-\Lambda_{1}$, and $-\Lambda_{2}$ of $\mathcal {H}_{\beta}$ for $\beta=3,5,6,7$. Finite-difference discretization with trapezoidal method is used with values of the parameters as in~\eqref{eqn:paraf}.}
 \label{p:p2}
\end{figure}

By~\cite[Theorem 2.4.3]{Bloemendal2011FiniteRP}, one finds the limiting density of the $k$th largest eigenvalue, after rescaling, of the $\beta$-Hermite ensemble $H^{\beta}_{n}$ at $\theta=k\pi$.

Using finite-difference discretization with trapezoidal method with values for the parameters as in~\eqref{eqn:paraf}, Figure~\ref{p:p2} shows the limiting densities of the largest three eigenvalues of $H^{\beta}_{n}$, namely,~$F'(\cdot,\pi)$,~$F'(\cdot,2\pi)$, and $F'(\cdot,3\pi)$ from right to left respectively. They can also be interpreted as the densities of $-\Lambda_{0}$, $-\Lambda_{1}$, and $-\Lambda_{2}$ of the stochastic Airy operator $\mathcal {H}_{\beta}$. As the value of $k$ increases, the limiting distribution becomes more concentrated.

\section{Outlook and open questions}\label{sec:outlook}

The error analysis in this paper was empirical, as a proof of convergence is elusive. It is likely that eigenvalue perturbation theory could be used to help show that the spectrum of the $x$-dependent families of matrices we consider remains close to, or inside of, the stability regions for the time-steppers we have chosen --- at least for sufficiently small time steps. A challenge here is to obtain quantitative bounds, making ``sufficiently small'' precise, and giving useful conclusions. Furthermore, even if the spectrum is understood, it is currently not known how to estimate the eigenvector condition numbers and pseudospectra of these families of matrices.

We also believe it to be possible to achieve better accuracy by adapting the global spectral method of Olver and Townsend~\cite{Townsend2015} and future work will be in this direction. A high-precision implementation of this idea could help corroborate conjectures about the tail behavior of the Tracy--Widom distribution.

\appendix

\section[The range of Delta x that causes instability]{The range of $\boldsymbol{\Delta x}$ that causes instability}\label{a:2}
\begin{figure}[t]
 \centering
 \subfloat[\vspace*{-6.5mm}\centering Error at $x=-2$]{{\includegraphics[width=0.46\linewidth]{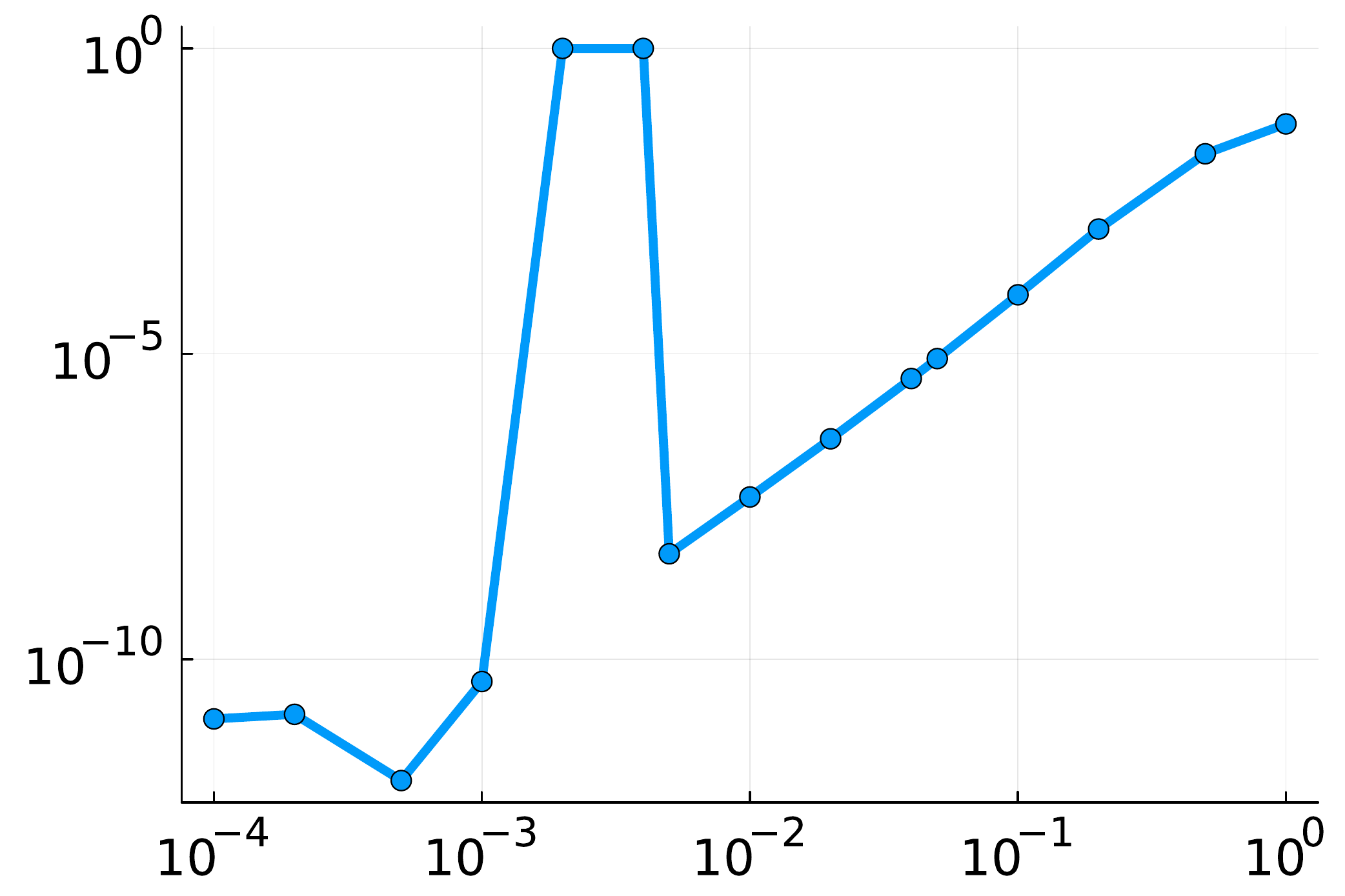} }}%
 \put(-105,-8){\color{black}\small $|\Delta x|$}
 \put(-225,55){\rotatebox{90}{\large Error}}
 \qquad\qquad
 \subfloat[\vspace*{-6.5mm}\centering Mean $\langle \delta\left(x,\Delta x\right)\rangle$]{{\includegraphics[width=0.46\linewidth]{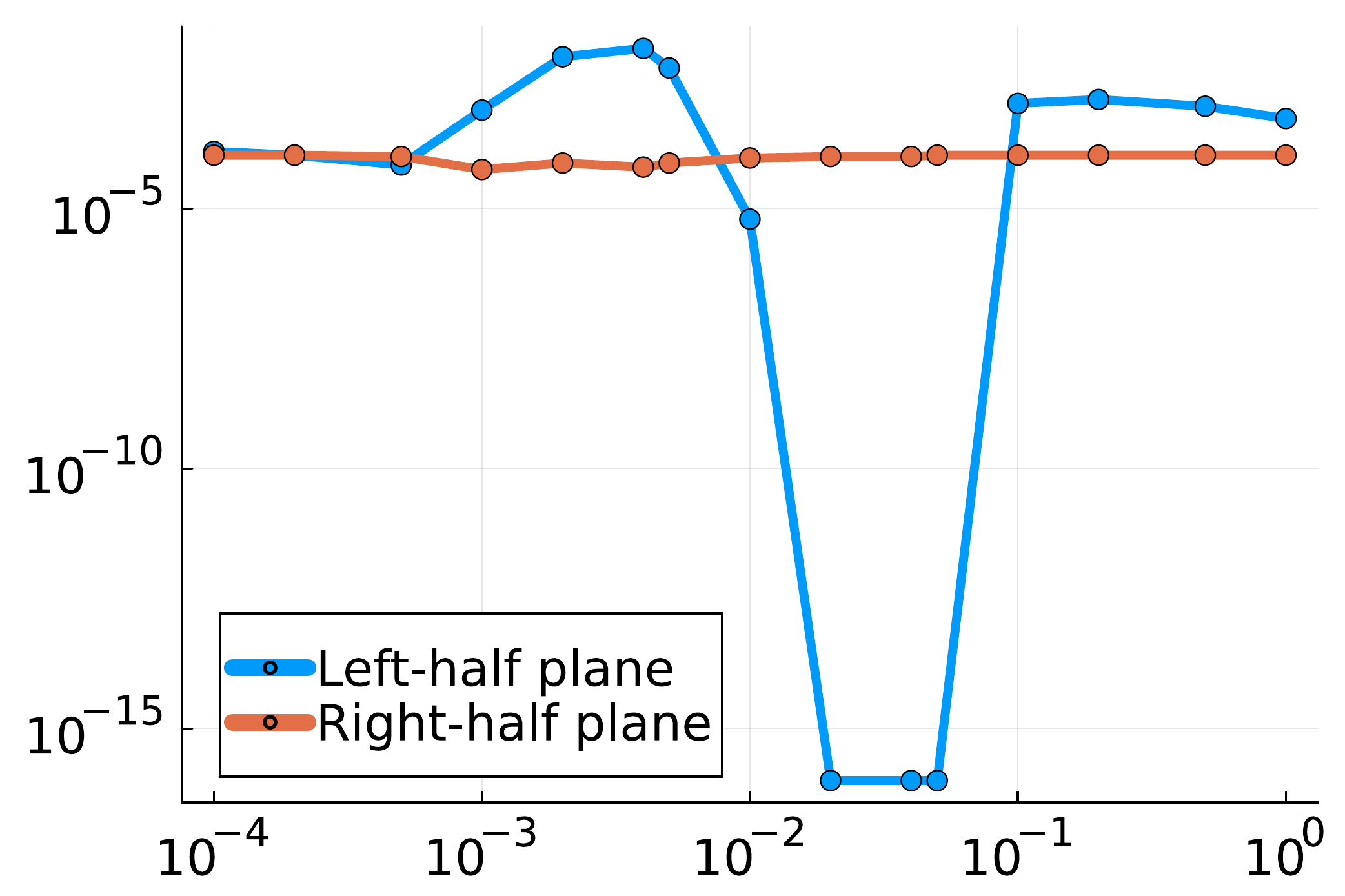} }}%
 \put(-107,-8){\color{black}\small $|\Delta x|$}
 \put(-225,40){\rotatebox{90}{\small $\langle \delta\left(x,\Delta x\right)\rangle$}}
 \qquad
 \vspace*{5mm}
 \subfloat[\vspace*{-6.5mm}\centering Mean $\langle \mu(x,\Delta x)\rangle$] {{\includegraphics[width=0.46\linewidth]{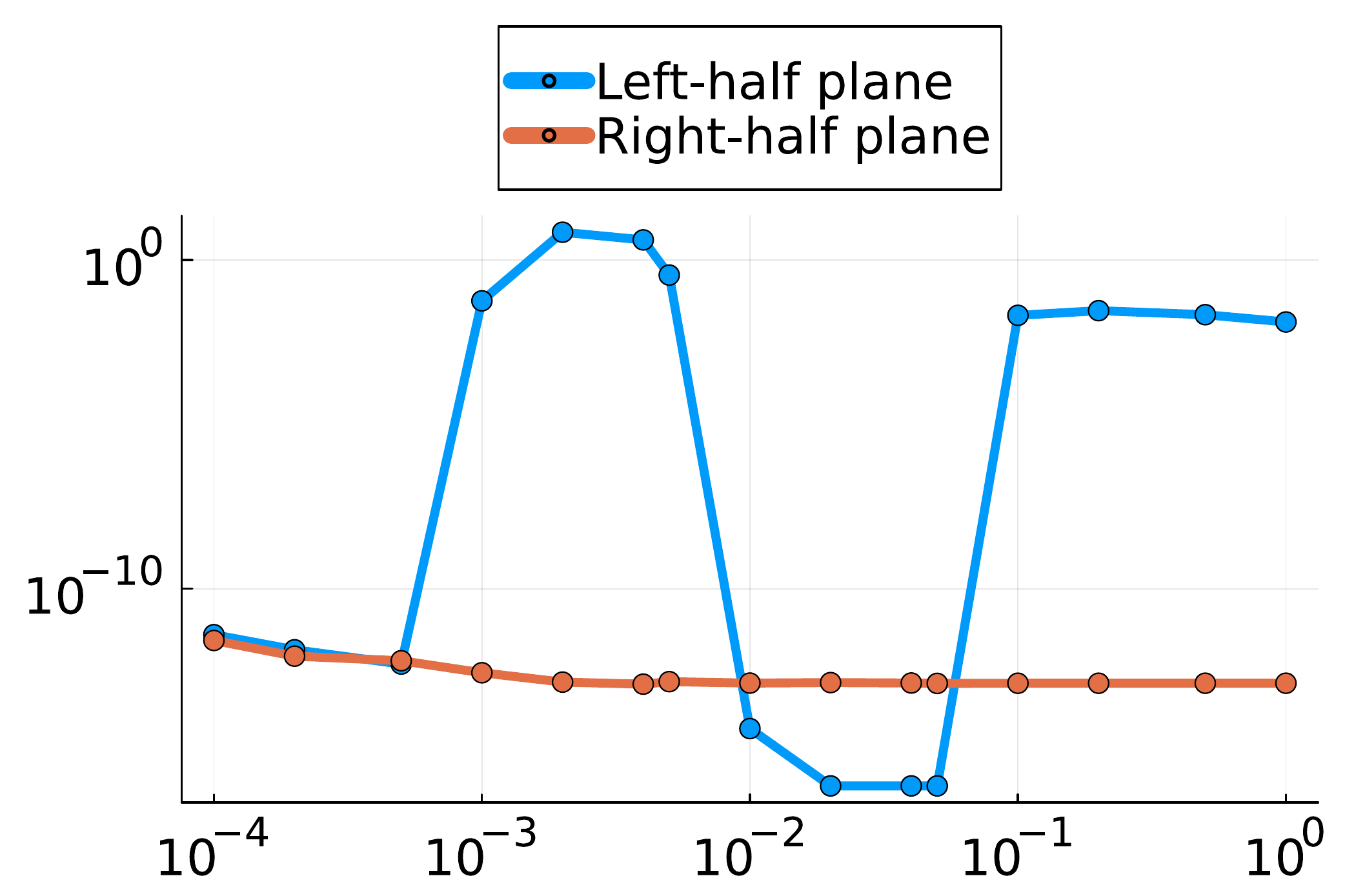} }}%
 \put(-107,-8){\color{black}\small $|\Delta x|$}
 \put(-225,40){\rotatebox{90}{\small $\langle\mu(x,\Delta x)\rangle$}}
 \caption{The non-monotonic convergence for BDF3 along with the mean fraction $\langle \delta\left(x,\Delta x\right)\rangle$ of the unstable eigenvalues and the mean value $\langle \mu(x,\Delta x)\rangle$. (a) \texttt{TW}($\beta\texttt{=}2$; $x_{0}=\lfloor 13/\sqrt{\beta}\rfloor$, \texttt{method="spectral"}, \texttt{step="bdf3"}, \texttt{interp=false}, $\Delta x$, $M\texttt{=}8000$) is implemented with $\Delta x =  -1,\allowbreak -0.5,-0.2,-0.1,-0.05,-0.04,-0.02,-0.01,-0.005,-0.004,
 -0.002,-0.001,-0.0005,-0.0002,\allowbreak -0.0001$.
 Measured errors are capped at one for readability.  (b)~The blue (red) dots correspond to the mean fraction $\langle \delta_{l}(x,\Delta x)\rangle (\langle \delta_{r}(x,\Delta x)\rangle )$ as defined in~\eqref{eqn:deltal} (\eqref{eqn:deltar}). (c)~The blue (red) dots correspond to the mean value $\langle \mu_{l}(x,\Delta x)\rangle (\langle \mu_{r}(x,\Delta x)\rangle )$ as defined in~\eqref{eqn:mul}(\eqref{eqn:mur}).}%
 \label{p:bdf3_non}
\end{figure}

\begin{figure}[!ht]\vspace{-2mm}
 \centering
 \subfloat[\vspace*{-5.5mm}\centering Zoomed-out view] {{\includegraphics[width=0.34\linewidth]{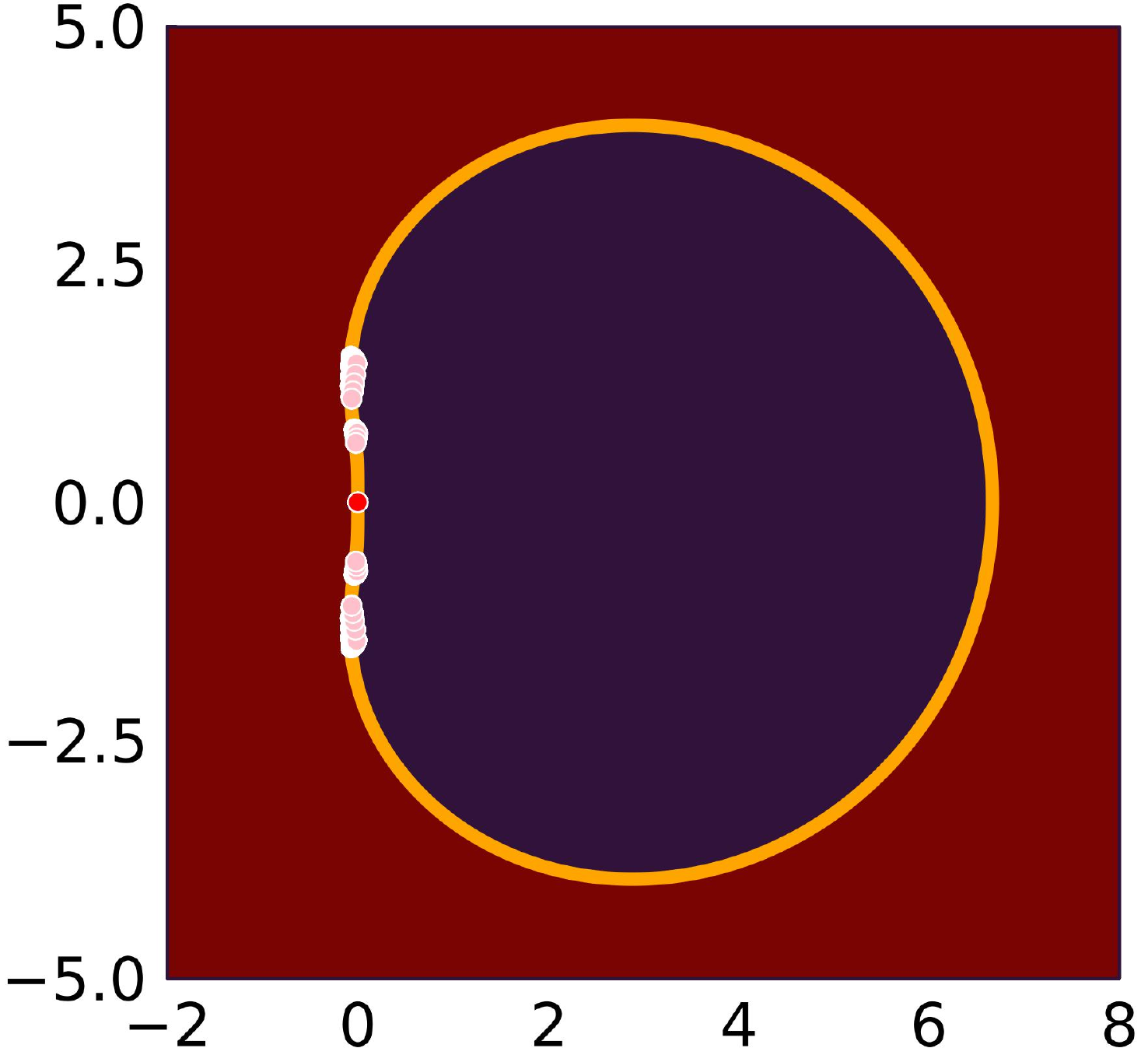} }}%
 \put(-168,67){\rotatebox{90}{\color{black}\small $\operatorname{Im}(z)$}}
 \put(-85,-7){\color{black}\small $\operatorname{Re}(z)$}
 \qquad\qquad
 \vspace*{5mm}
 \subfloat[\vspace*{-5.5mm}\centering Zoomed-in view]{{\includegraphics[width=0.33\linewidth]{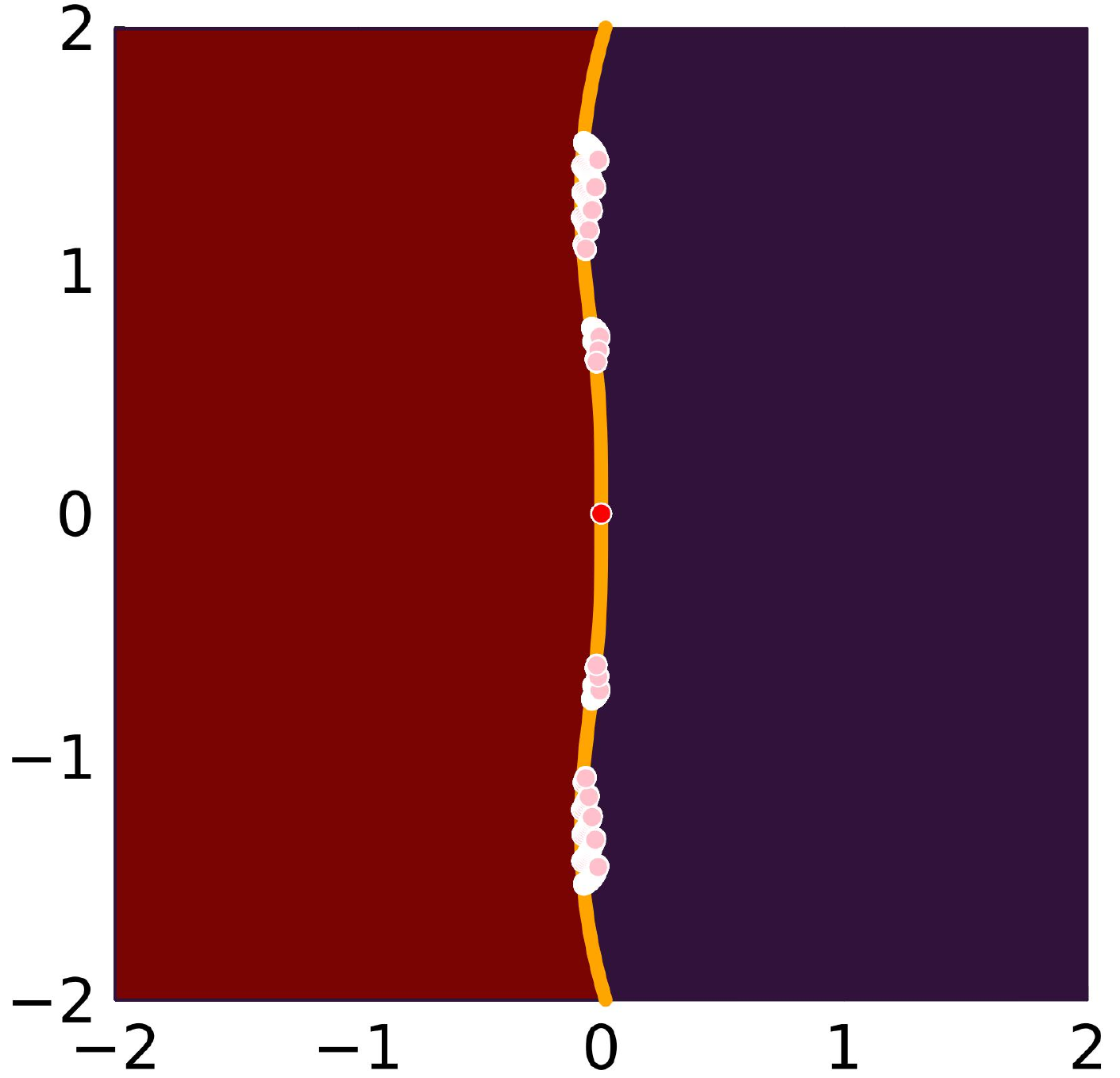} }}%
 \put(-168,67){\rotatebox{90}{\color{black}\small $\operatorname{Im}(z)$}}
 \put(-85,-7){\color{black}\small $\operatorname{Re}(z)$}

 \vspace{-0.5mm}

 \caption{The absolute stability region of BDF3 applying on spectral discretization (the maroon-shaded region along with the orange boundary curve) and the eigenvalues $z$ of $\Delta x(A+xB)$ (the red dots with white boundary correspond to $\operatorname{Re}(z)>0$, and the pink dots with white boundary correspond to $\operatorname{Re}(z)<0$) for $\beta=2$, $x=x_{N},-9,\dots,x_{0}$, $\Delta x=-0.004,-0.002$, and $M=8\times 10^3$.}\label{p:bdf3_z}
 \vspace{-1mm}
\end{figure}

\begin{figure}[!ht]
 \centering
 \subfloat[\vspace*{-6.5mm}\centering Error at $x=-2$]{{\includegraphics[width=0.46\linewidth]{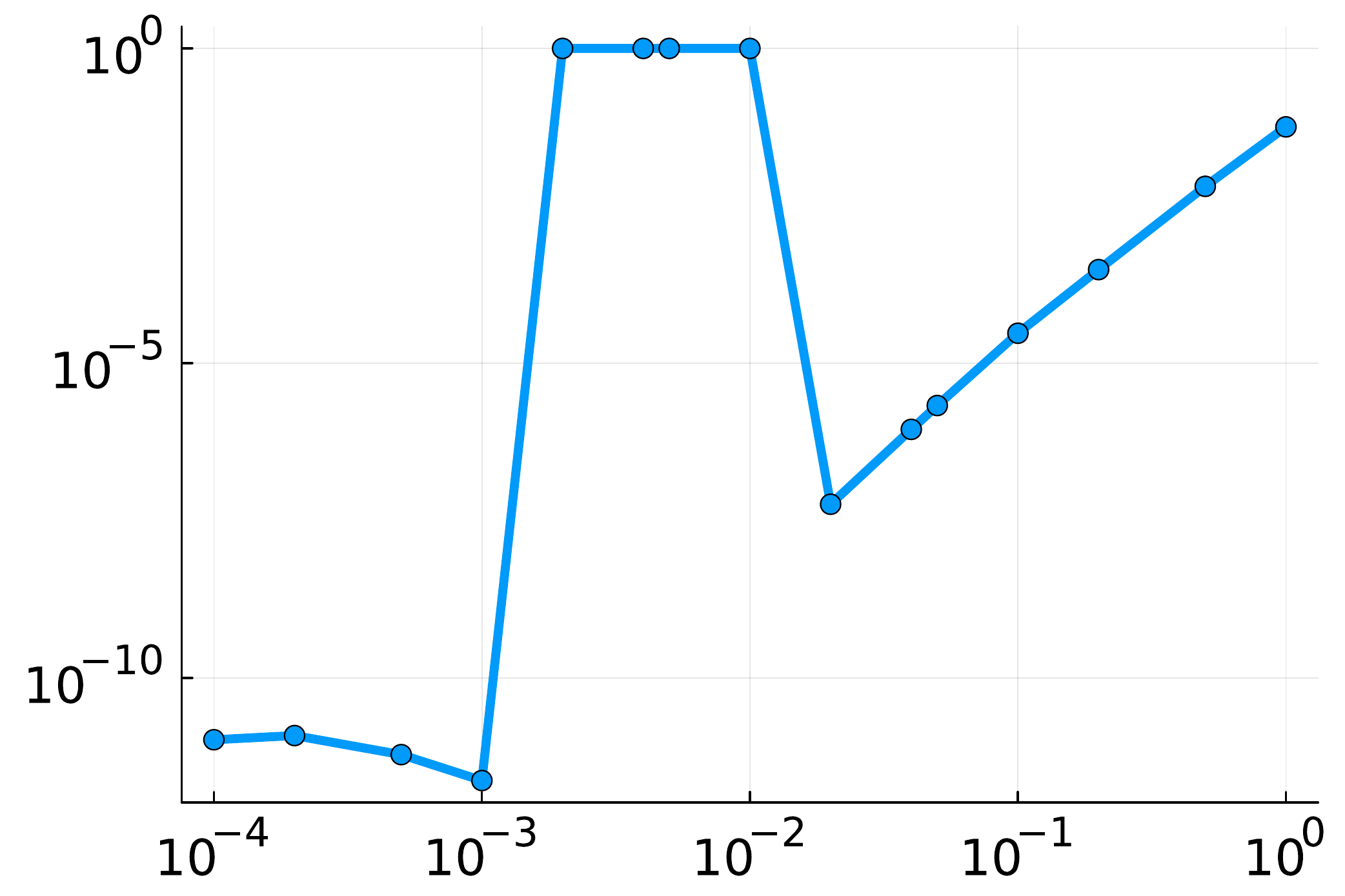} }}%
 \put(-105,-8){\color{black}\small $|\Delta x|$}
 \put(-225,55){\rotatebox{90}{\large Error}}
 \qquad\qquad
 \subfloat[\vspace*{-6.5mm}\centering Mean $\langle \delta\left(x,\Delta x\right)\rangle$]{{\includegraphics[width=0.46\linewidth]{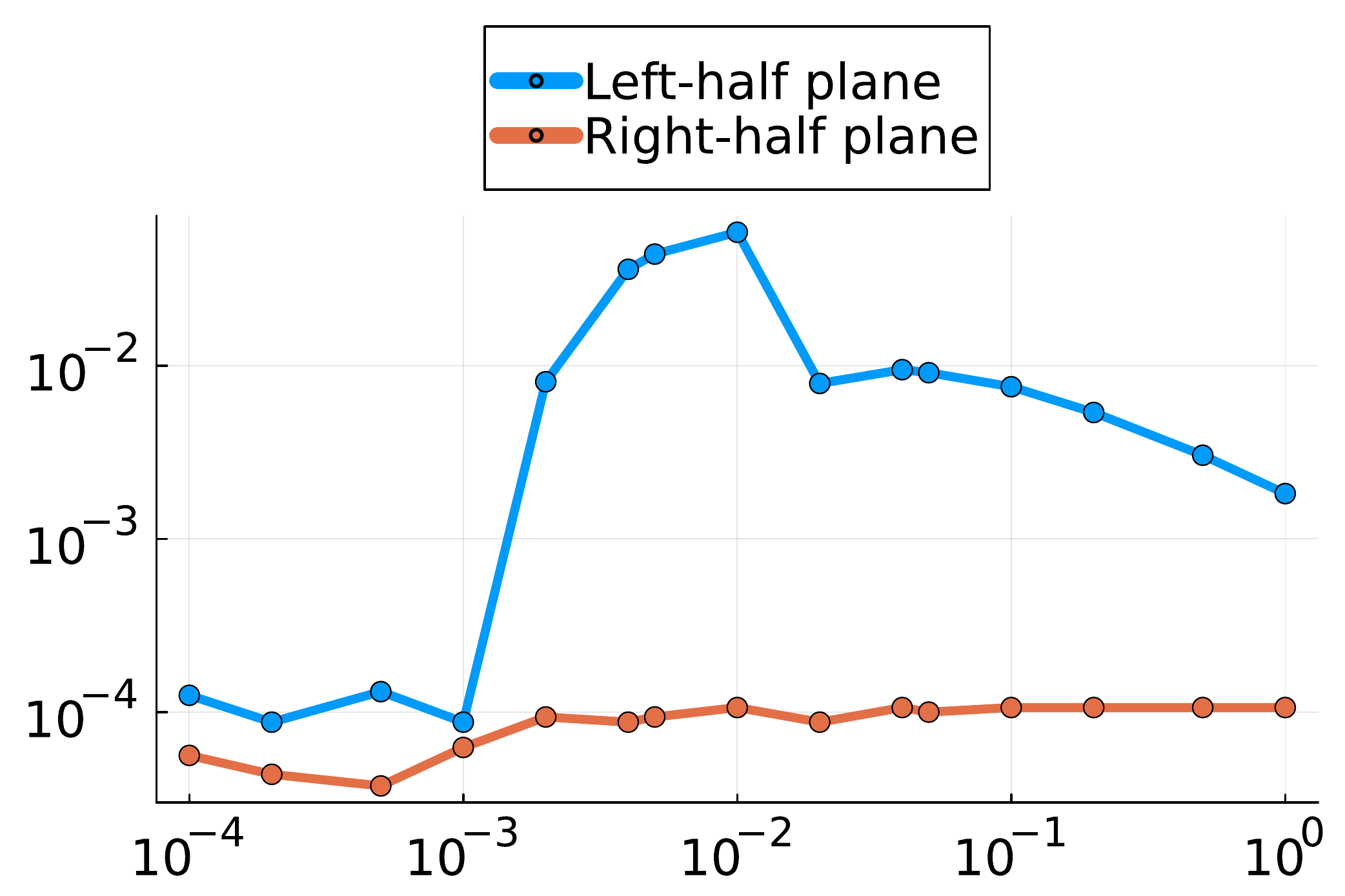} }}%
 \put(-107,-8){\color{black}\small $|\Delta x|$}
 \put(-225,40){\rotatebox{90}{\small $\langle \delta\left(x,\Delta x\right)\rangle$}}
 \qquad
 \vspace*{5mm}
 \subfloat[\vspace*{-6.5mm}\centering Mean $\langle\mu(x,\Delta x)\rangle$] {{\includegraphics[width=0.46\linewidth]{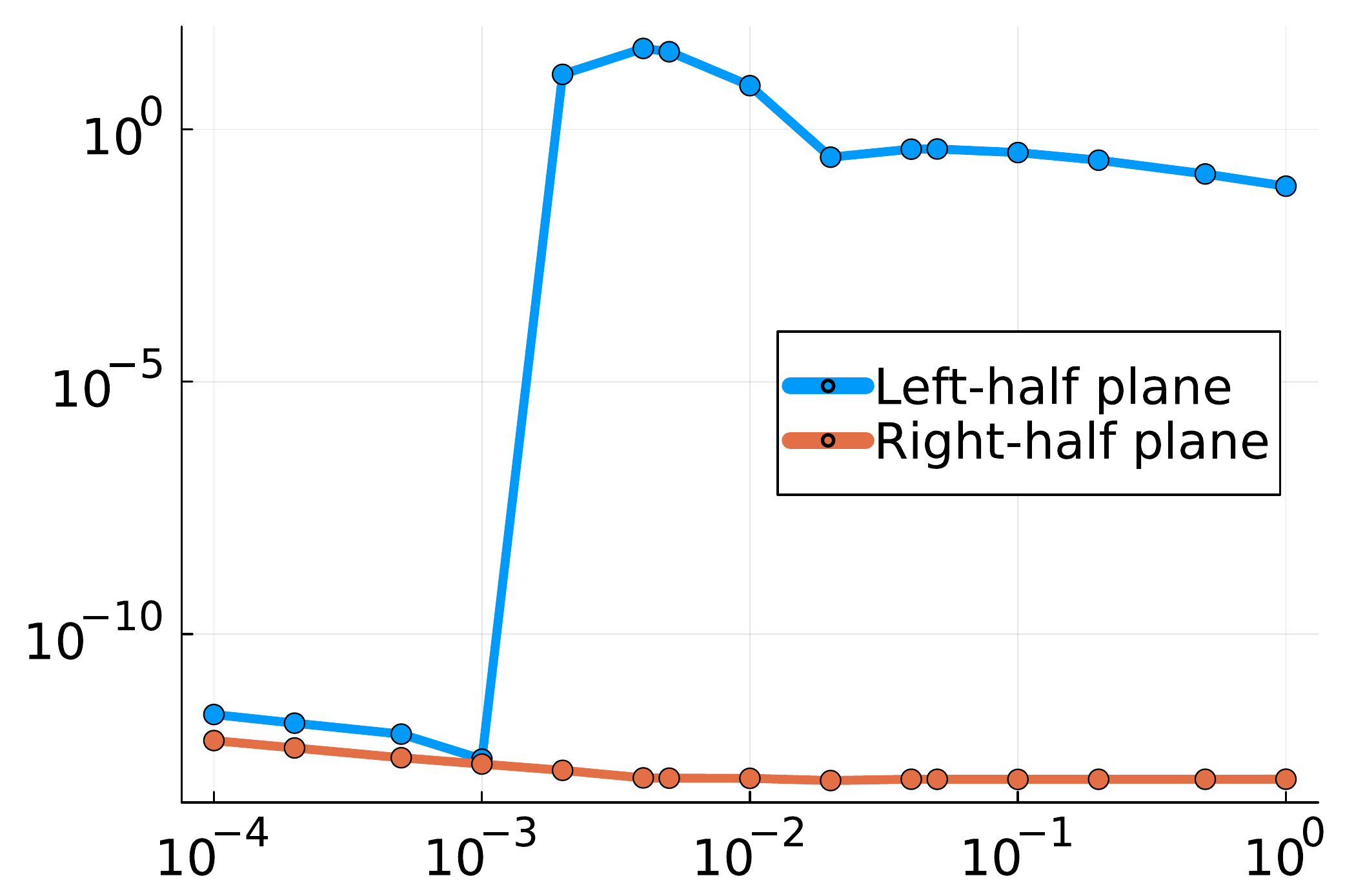} }}%
 \put(-107,-8){\color{black}\small $|\Delta x|$}
 \put(-225,44){\rotatebox{90}{\small $\langle\mu(x,\Delta x)\rangle$}}
 \caption{The non-monotonic convergence for BDF4 along with the mean fraction $\langle \delta\left(x,\Delta x\right)\rangle$ of the unstable eigenvalues and the mean value $\langle \mu(x,\Delta x)\rangle$. The detailed computing process for each plot is the same as for BDF3, see the caption of Figure~\ref{p:bdf3_non}.}%
\end{figure}

\begin{figure}[!ht]\vspace{-3mm}
 \centering
 \subfloat[\vspace*{-5.5mm}\centering Zoomed-out view] {{\includegraphics[width=0.37\linewidth]{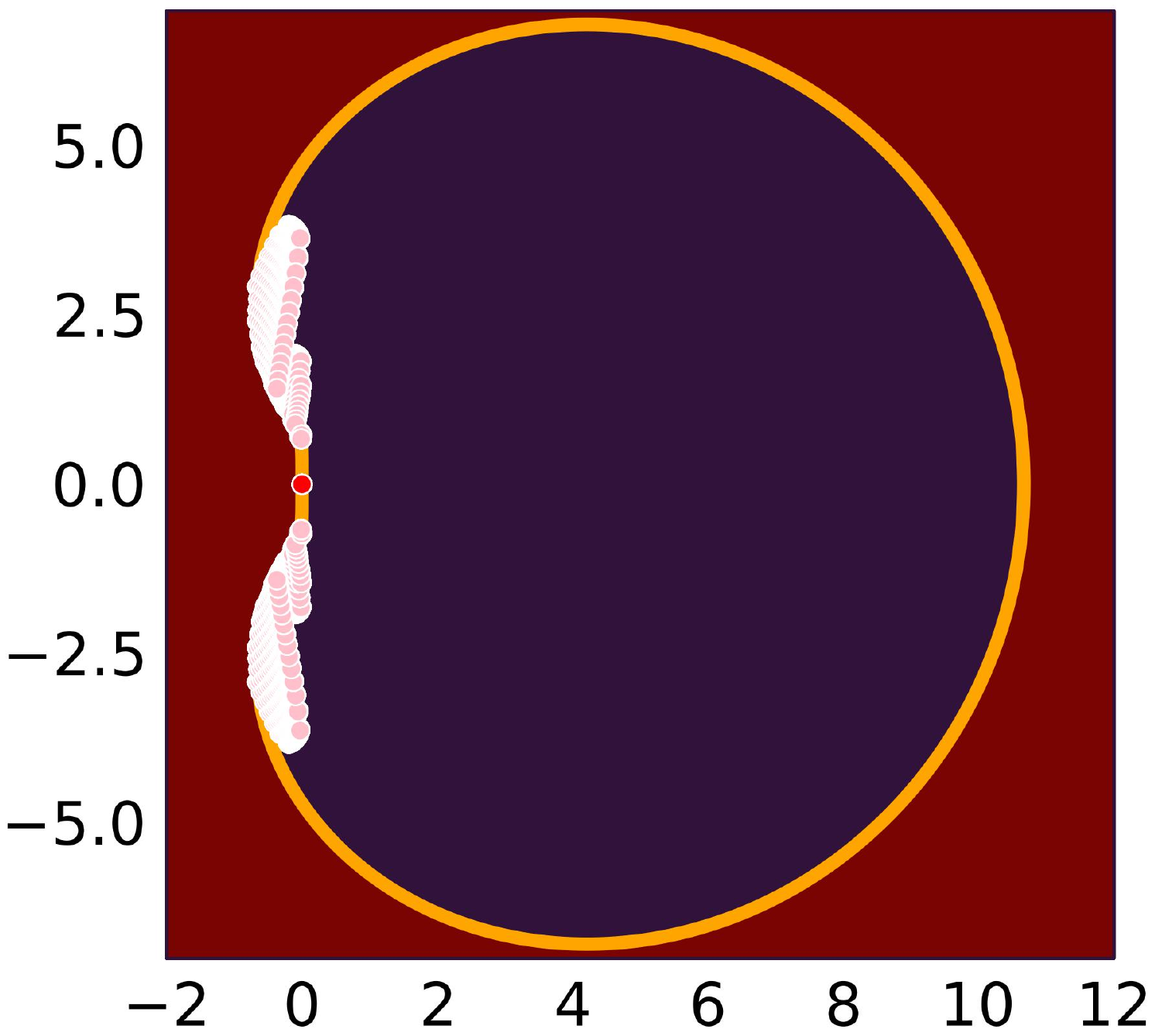} }}%
 \put(-183,67){\rotatebox{90}{\color{black}\small $\operatorname{Im}(z)$}}
 \put(-90,-7){\color{black}\small $\operatorname{Re}(z)$}
 \qquad\qquad
 \vspace*{5mm}
 \subfloat[\vspace*{-5.5mm}\centering Zoomed-in view]{{\includegraphics[width=0.35\linewidth]{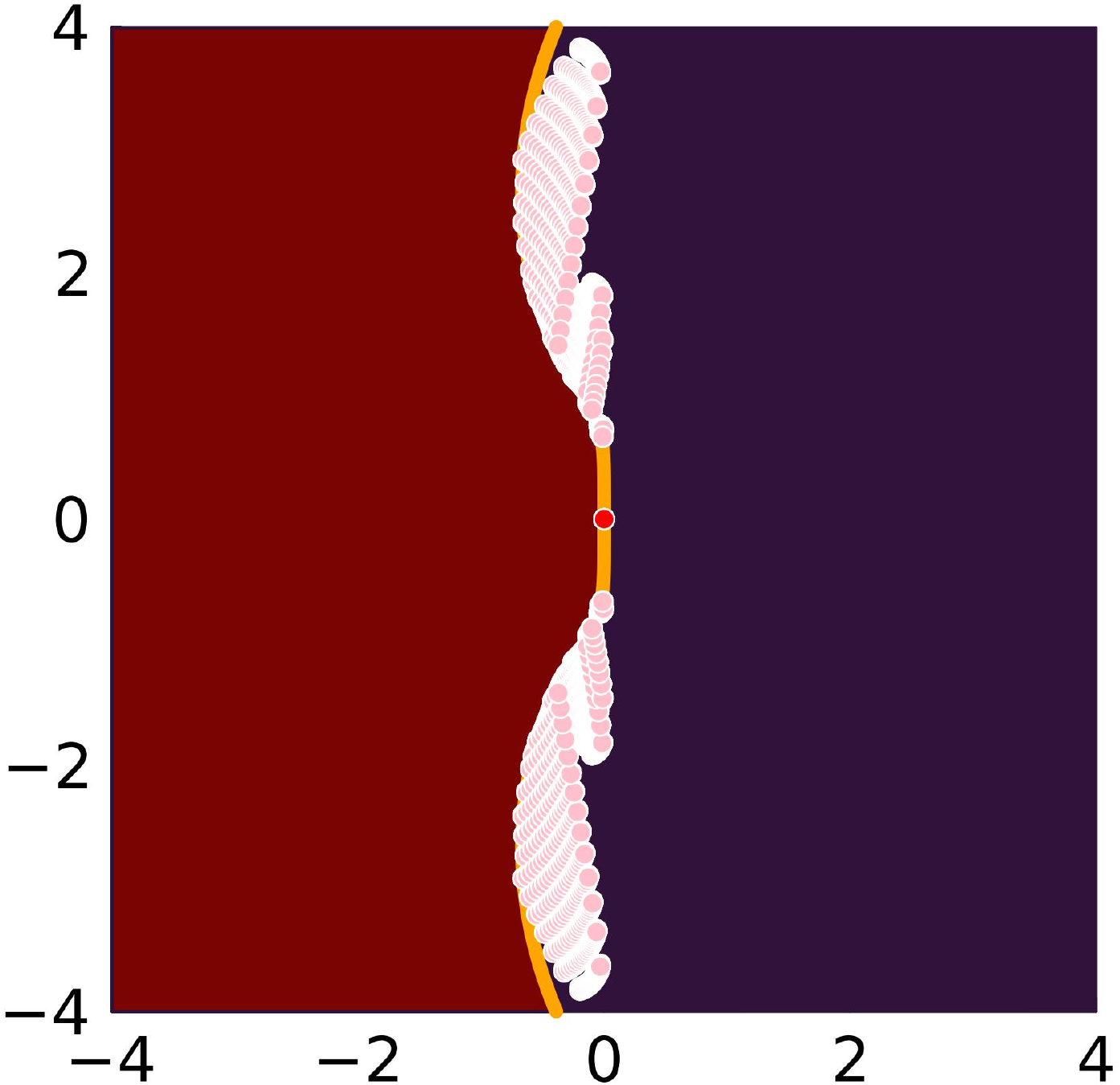} }}%
 \put(-178,67){\rotatebox{90}{\color{black}\small $\operatorname{Im}(z)$}}
 \put(-90,-7){\color{black}\small $\operatorname{Re}(z)$}
 \caption{The absolute stability region of BDF4 (the maroon-shaded region along with the orange boundary curve) and the eigenvalues $z$ of $\Delta x(A+xB)$ (the red dots with white boundary correspond to $\operatorname{Re}(z)>0$, and the pink dots with white boundary correspond to $\operatorname{Re}(z)<0$) for $\beta=2$, $x=x_{N},-9,\dots,x_{0}$, $\Delta x=-0.01,-0.005,-0.004,-0.002$, and $M=8\times 10^3$.}%
\end{figure}
\begin{figure}[!ht]
 \centering
 \subfloat[\vspace*{-6.5mm}\centering Error at $x=-2$]{{\includegraphics[width=0.46\linewidth]{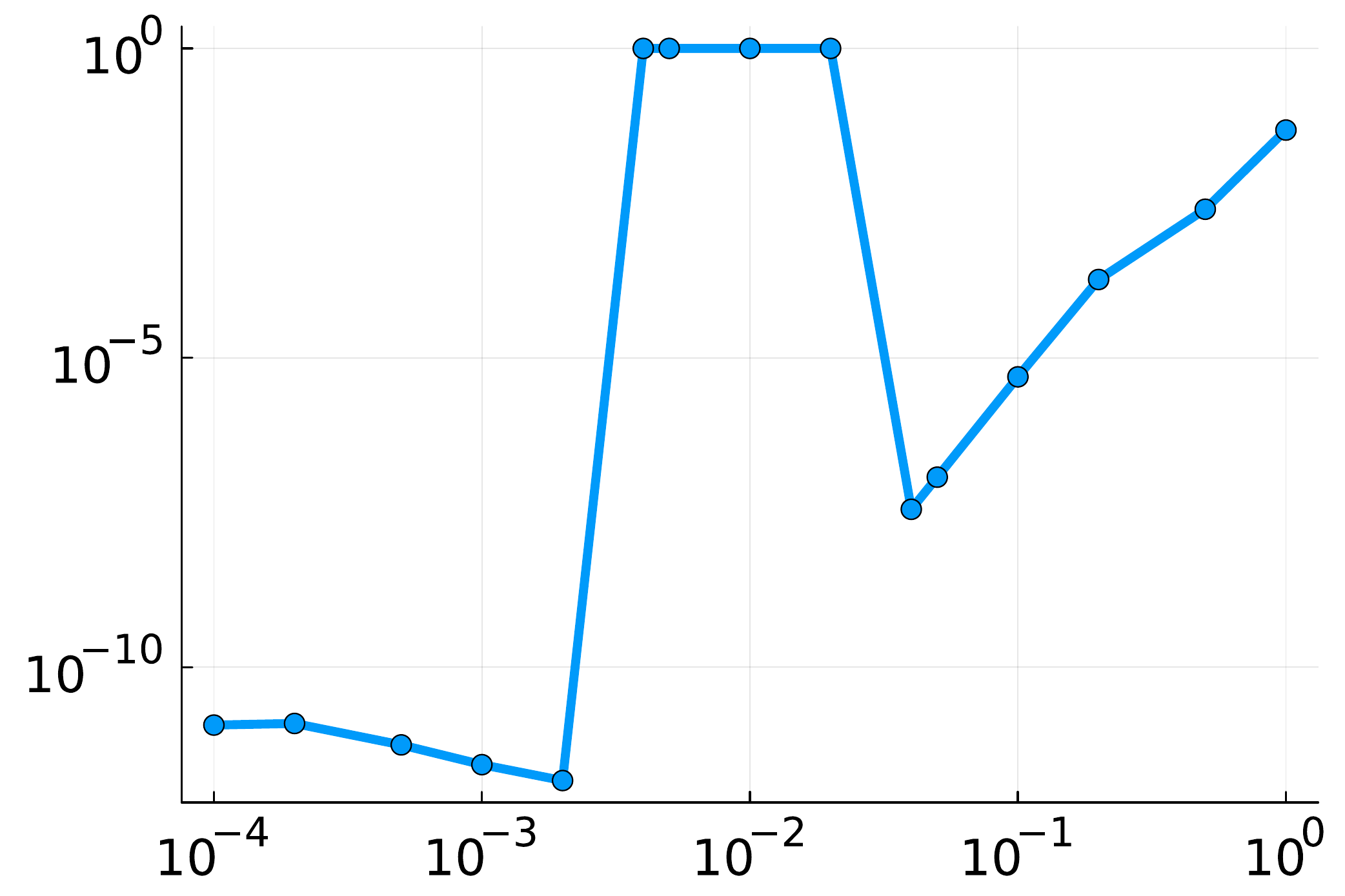} }}%
 \put(-105,-8){\color{black}\small $|\Delta x|$}
 \put(-225,55){\rotatebox{90}{\large Error}}
 \qquad\qquad
 \subfloat[\vspace*{-6.5mm}\centering Mean $\langle \delta\left(x,\Delta x\right)\rangle$]{{\includegraphics[width=0.46\linewidth]{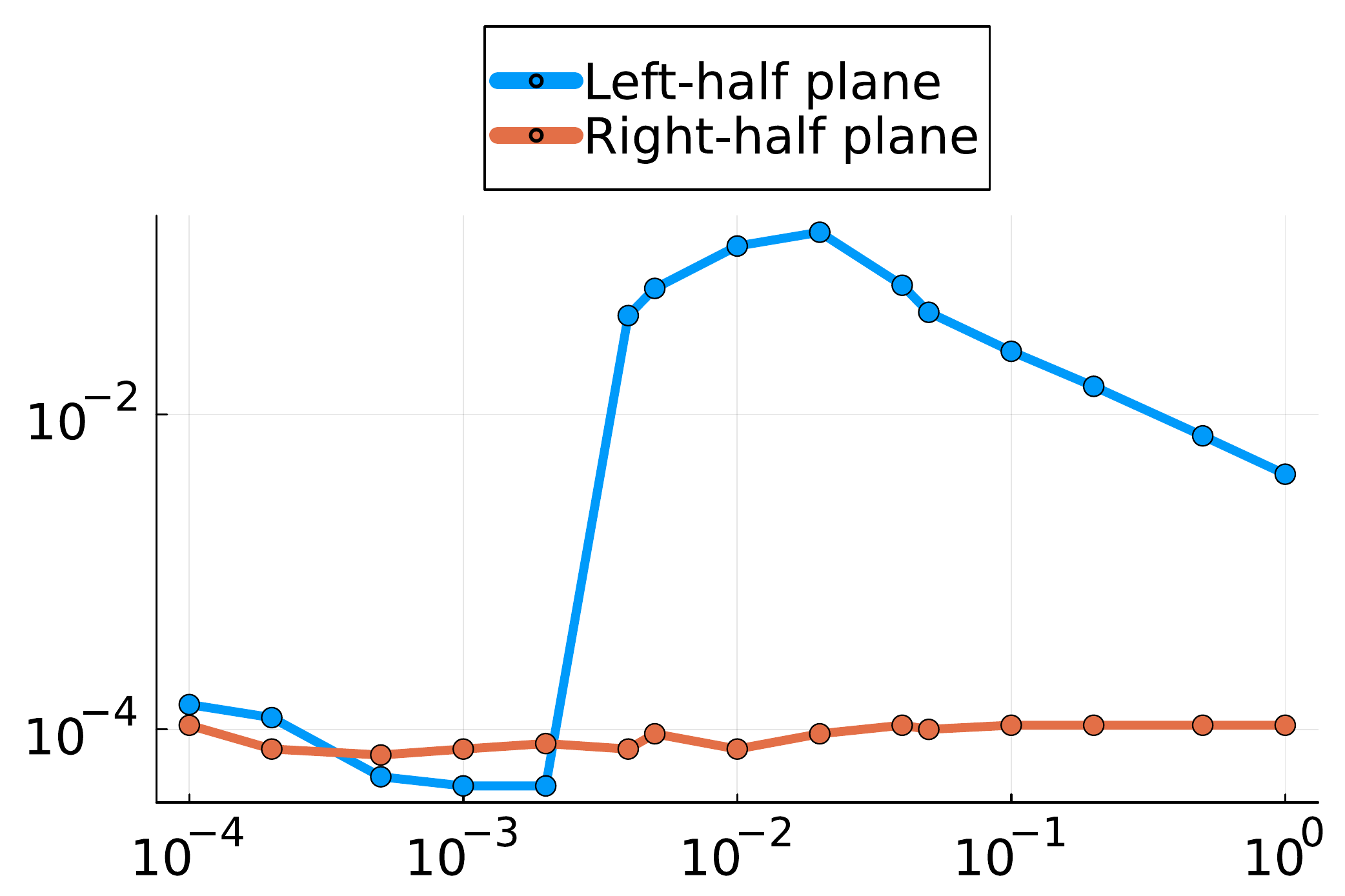} }}%
 \put(-107,-8){\color{black}\small $|\Delta x|$}
 \put(-225,40){\rotatebox{90}{\small $\langle \delta\left(x,\Delta x\right)\rangle$}}
 \qquad
 \vspace*{5mm}
 \subfloat[\vspace*{-6.5mm}\centering Mean $\langle\mu(x,\Delta x)\rangle$] {{\includegraphics[width=0.46\linewidth]{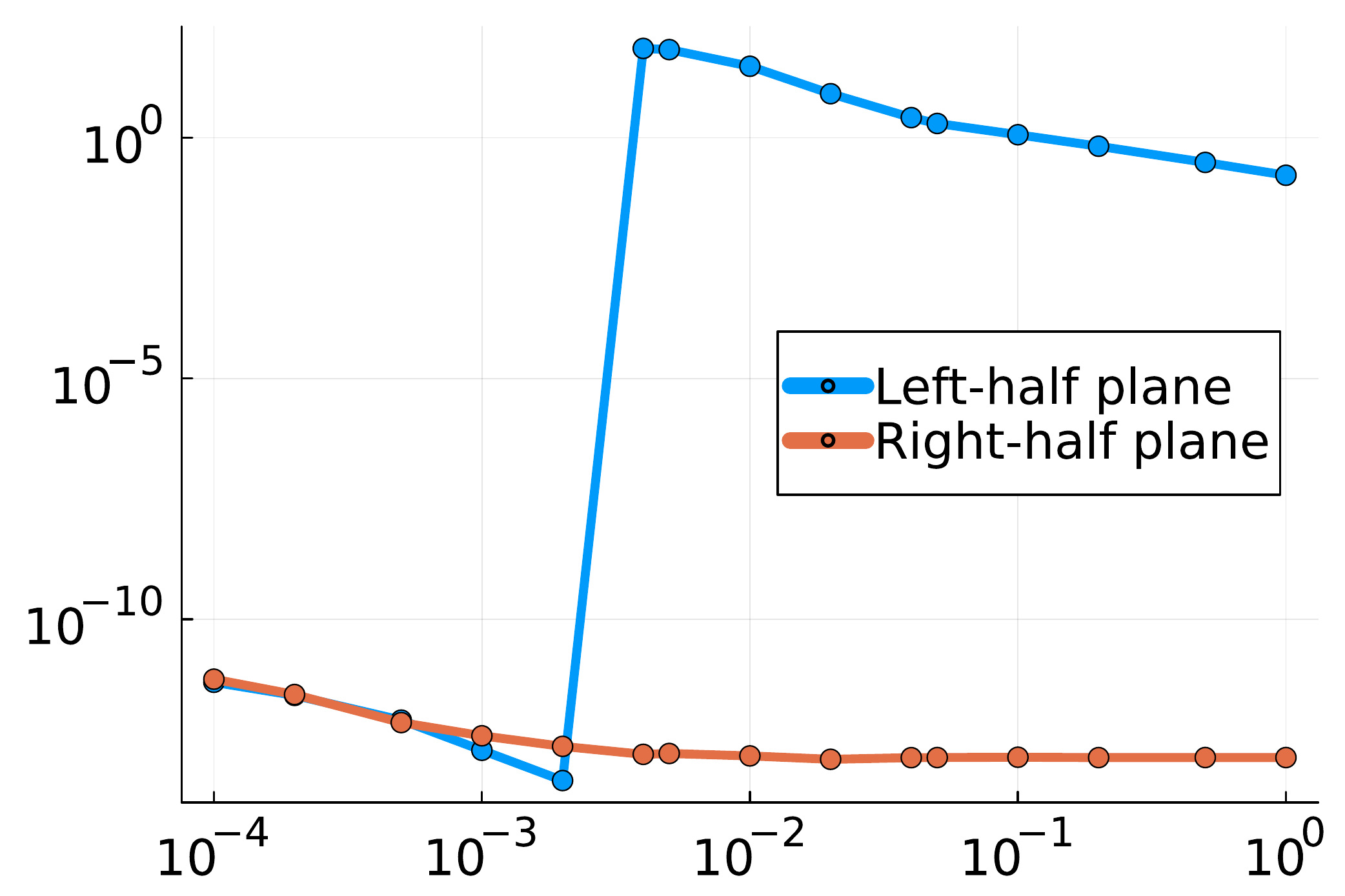} }}%
 \put(-107,-8){\color{black}\small $|\Delta x|$}
 \put(-225,43){\rotatebox{90}{\small $\langle\mu(x,\Delta x)\rangle$}}
 \caption{The non-monotonic convergence for BDF5 along with the mean fraction $\langle \delta\left(x,\Delta x\right)\rangle$ of the unstable eigenvalues and the mean value $\langle \mu(x,\Delta x)\rangle$. The detailed computing process for each plot is the same as for BDF3, see the caption of Figure~\ref{p:bdf3_non}.}%
\end{figure}

\begin{figure}[!ht]
 \centering
 \subfloat[\vspace*{-5.5mm}\centering Zoomed-out view] {{\includegraphics[width=0.335\linewidth]{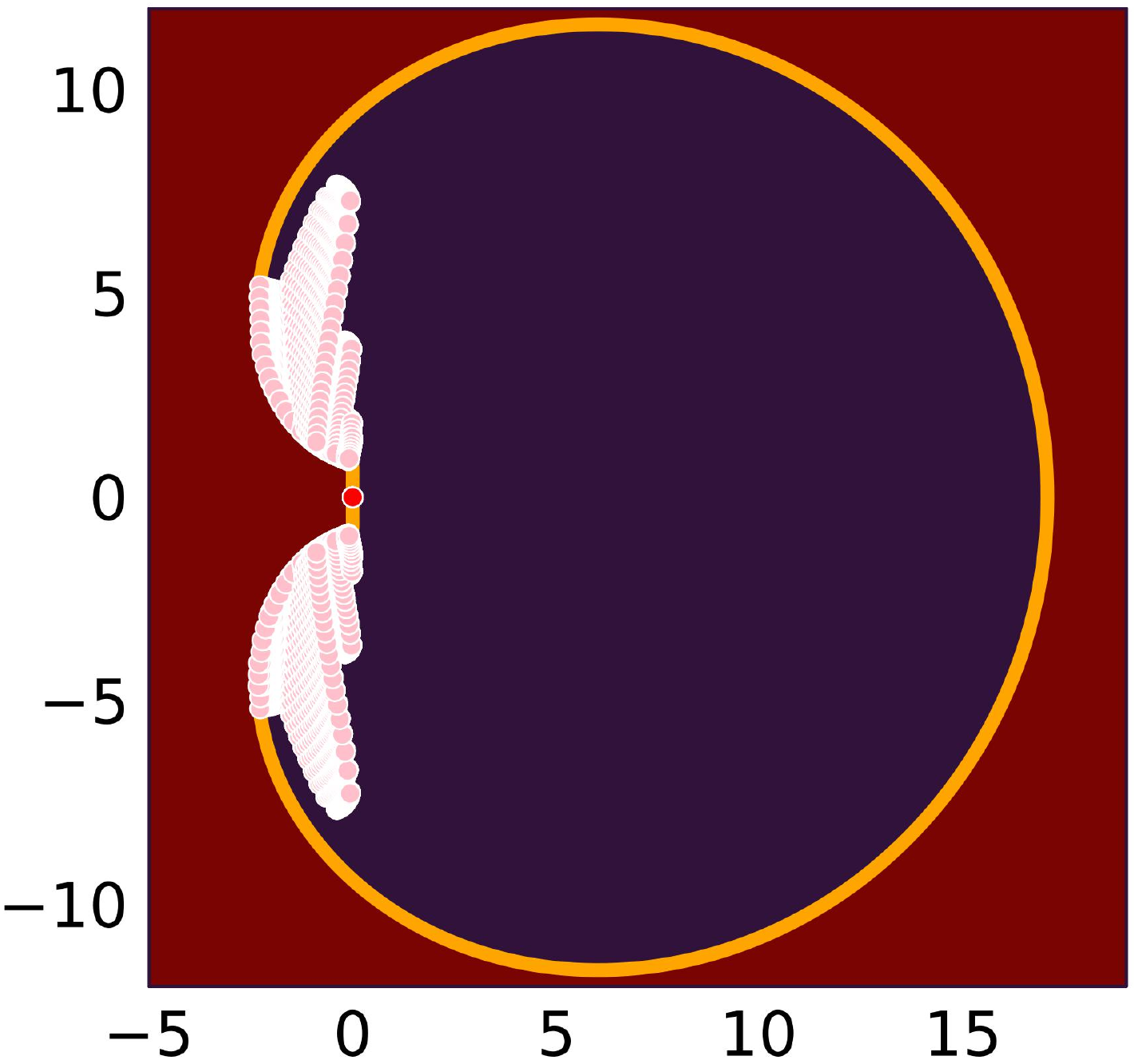} }}%
 \put(-170,67){\rotatebox{90}{\color{black}\small $\operatorname{Im}(z)$}}
 \put(-80,-7){\color{black}\small $\operatorname{Re}(z)$}
 \qquad\qquad
 \vspace*{5mm}
 \subfloat[\vspace*{-5.5mm}\centering Zoomed-in view]{{\includegraphics[width=0.33\linewidth]{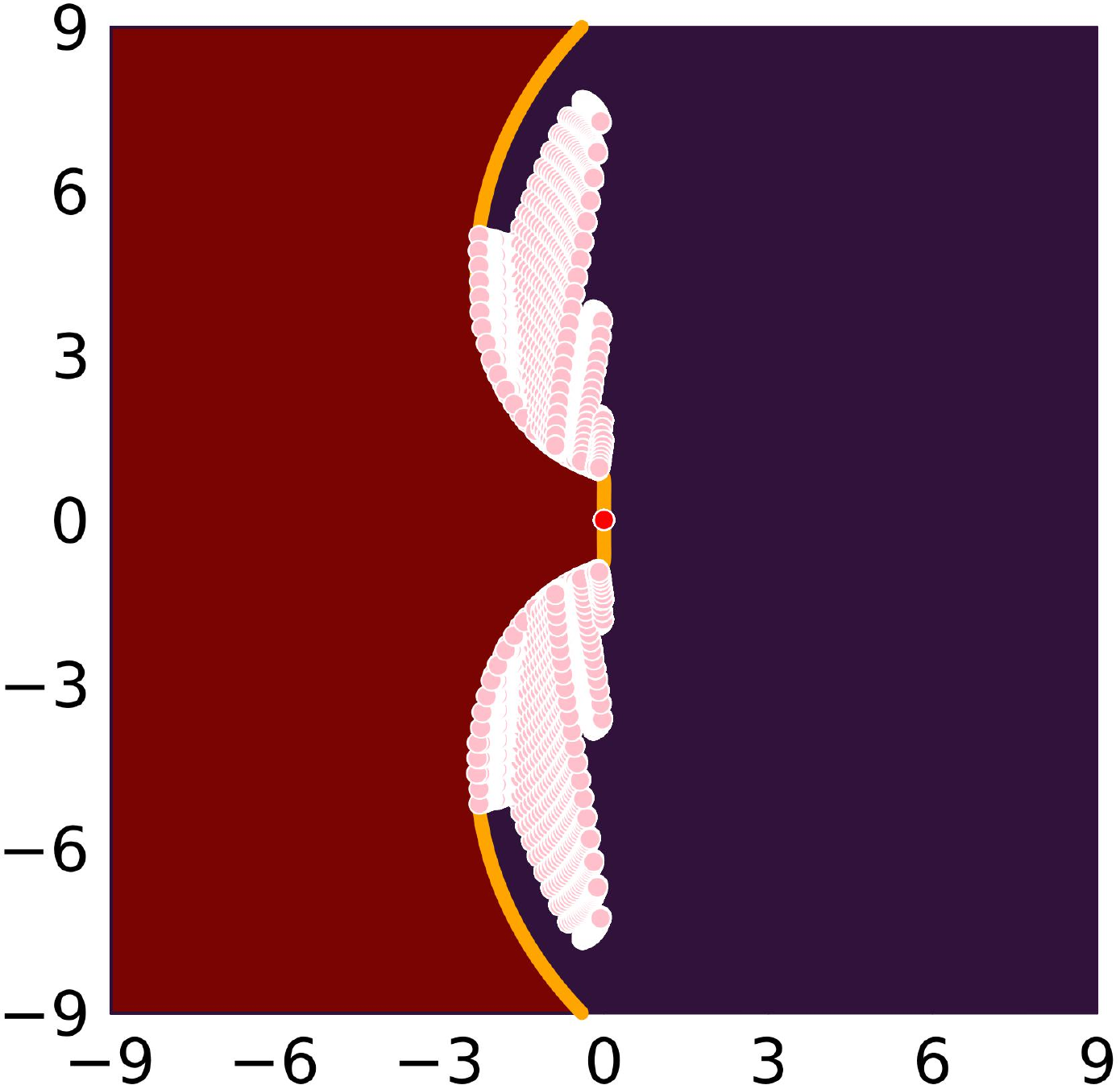} }}%
 \put(-170,67){\rotatebox{90}{\color{black}\small $\operatorname{Im}(z)$}}
 \put(-80,-7){\color{black}\small $\operatorname{Re}(z)$}
 \caption{The absolute stability region of BDF5 (the maroon-shaded region along with the orange boundary curve) and the eigenvalues $z$ of $\Delta x(A+xB)$ (the red dots with white boundary correspond to $\operatorname{Re}(z)>0$, and the pink dots with white boundary correspond to $\operatorname{Re}(z)<0$) for $\beta=2$, $x=x_{N},-9,\dots,x_{0}$, $\Delta x=-0.02,-0.01,-0.005,-0.004$, and $M=8\times 10^3$.}%
\end{figure}
\begin{figure}[!ht]
 \centering
 \subfloat[\vspace*{-6.5mm}\centering Error at $x=-2$]{{\includegraphics[width=0.46\linewidth]{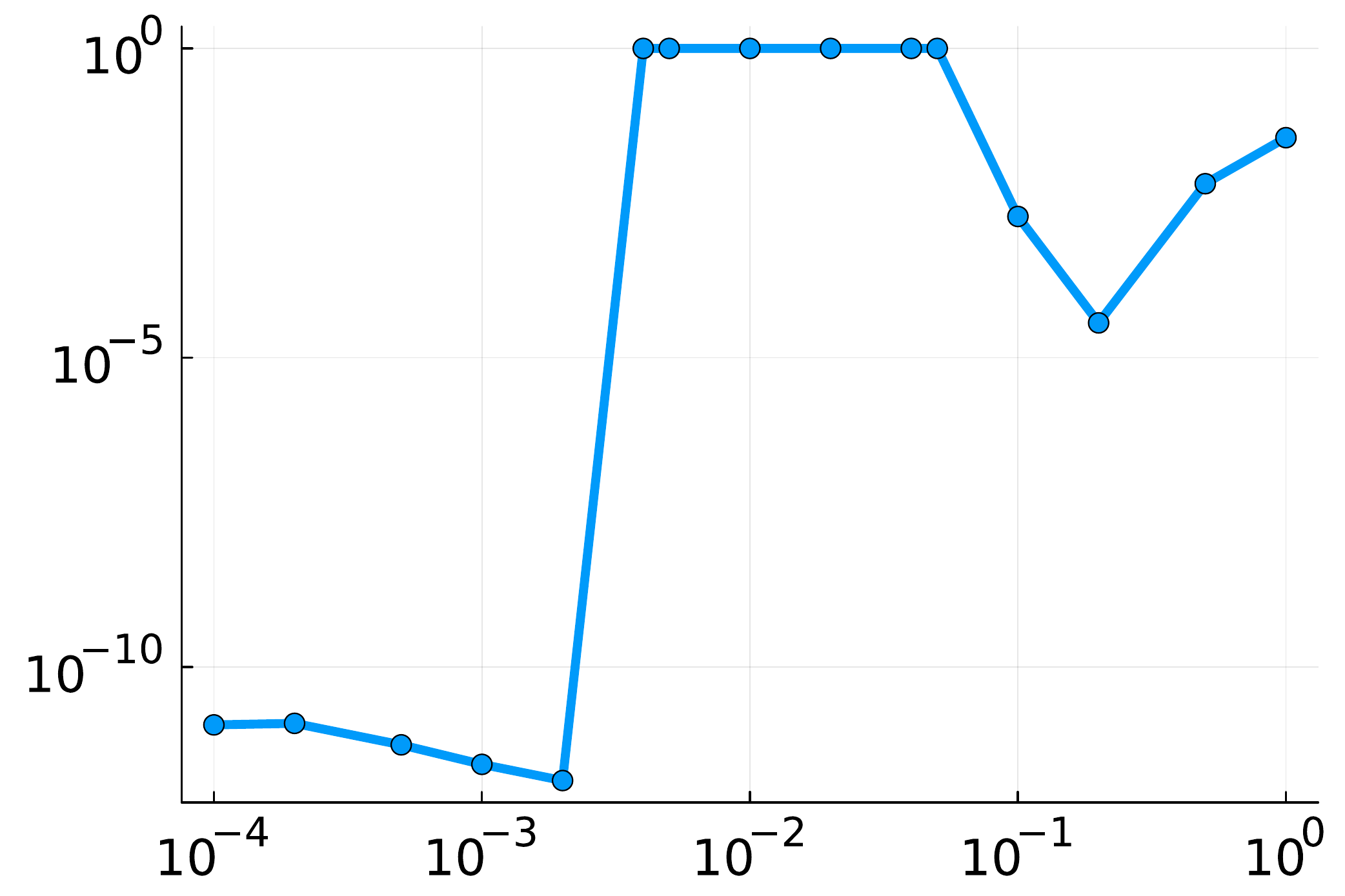} }}%
 \put(-105,-8){\color{black}\small $|\Delta x|$}
 \put(-225,55){\rotatebox{90}{\large Error}}
 \qquad\qquad
 \subfloat[\vspace*{-6.5mm}\centering Mean $\langle \delta\left(x,\Delta x\right)\rangle$]{{\includegraphics[width=0.46\linewidth]{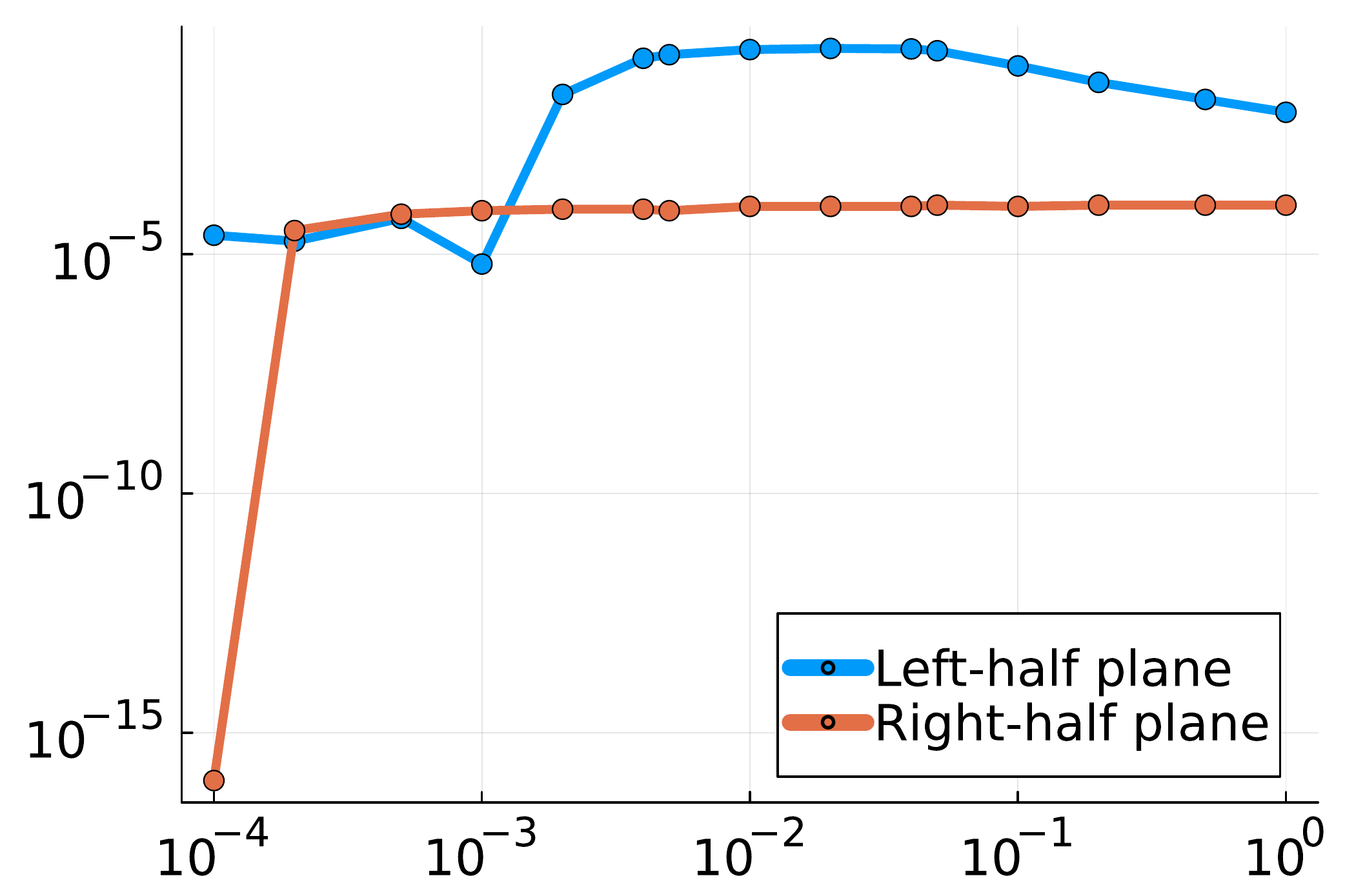} }}%
 \put(-107,-8){\color{black}\small $|\Delta x|$}
 \put(-225,40){\rotatebox{90}{\small $\langle \delta\left(x,\Delta x\right)\rangle$}}
 \qquad
 \vspace*{5mm}
 \subfloat[\vspace*{-6.5mm}\centering Mean $\langle\mu(x,\Delta x)\rangle$] {{\includegraphics[width=0.46\linewidth]{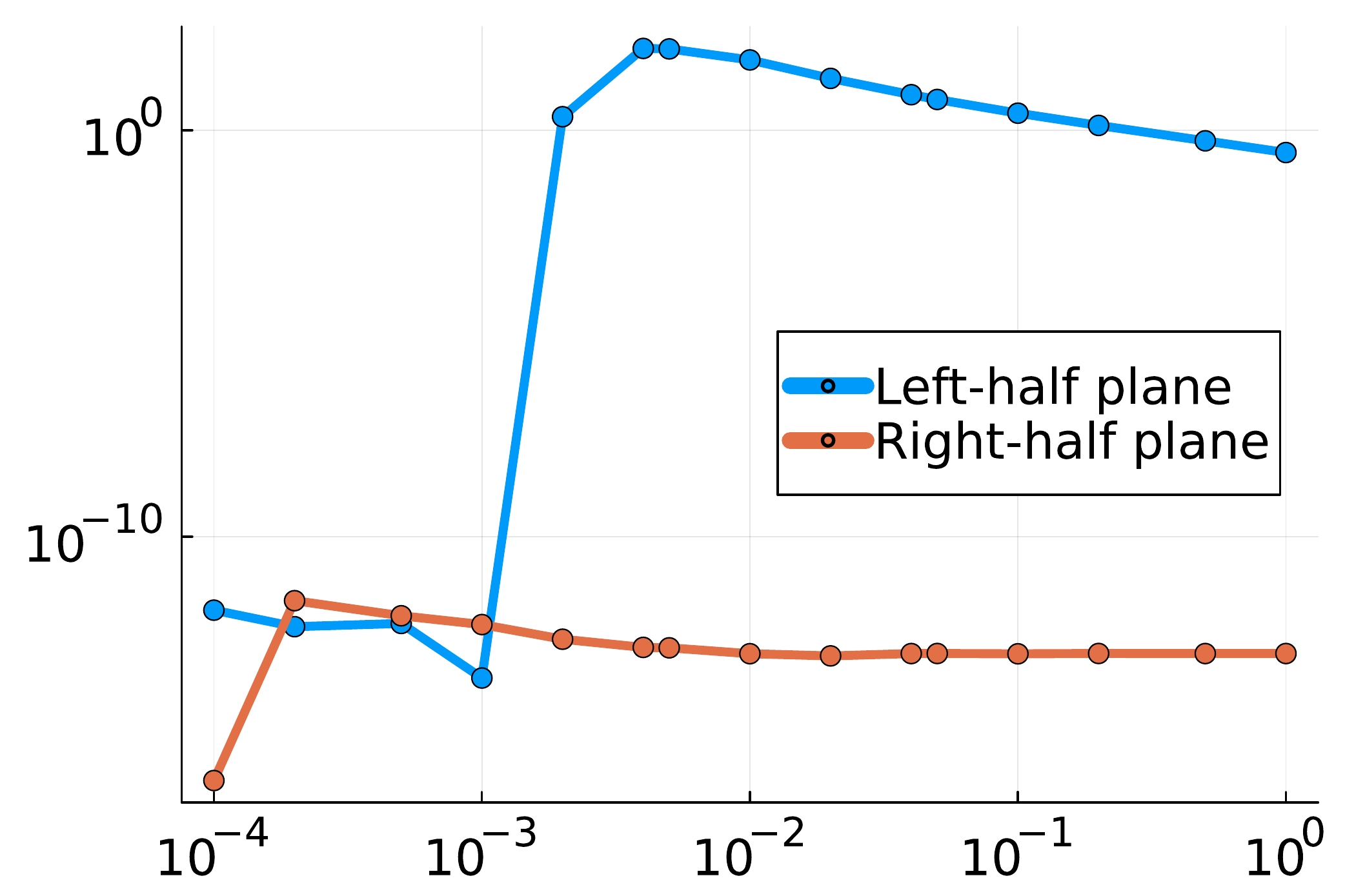} }}%
 \put(-107,-8){\color{black}\small $|\Delta x|$}
 \put(-225,43){\rotatebox{90}{\small $\langle\mu(x,\Delta x)\rangle$}}
 \caption{The non-monotonic convergence for BDF6 along with the mean fraction $\langle \delta\left(x,\Delta x\right)\rangle$ of the unstable eigenvalues and the mean value $\langle \mu(x,\Delta x)\rangle$. The detailed computing process for each plot is the same as for BDF3, see the caption of Figure~\ref{p:bdf3_non}.}%
\end{figure}
\begin{figure}[!ht]\vspace{-2mm}
 \centering
 \subfloat[\vspace*{-5.5mm}\centering Zoomed-out view] {{\includegraphics[width=0.36\linewidth]{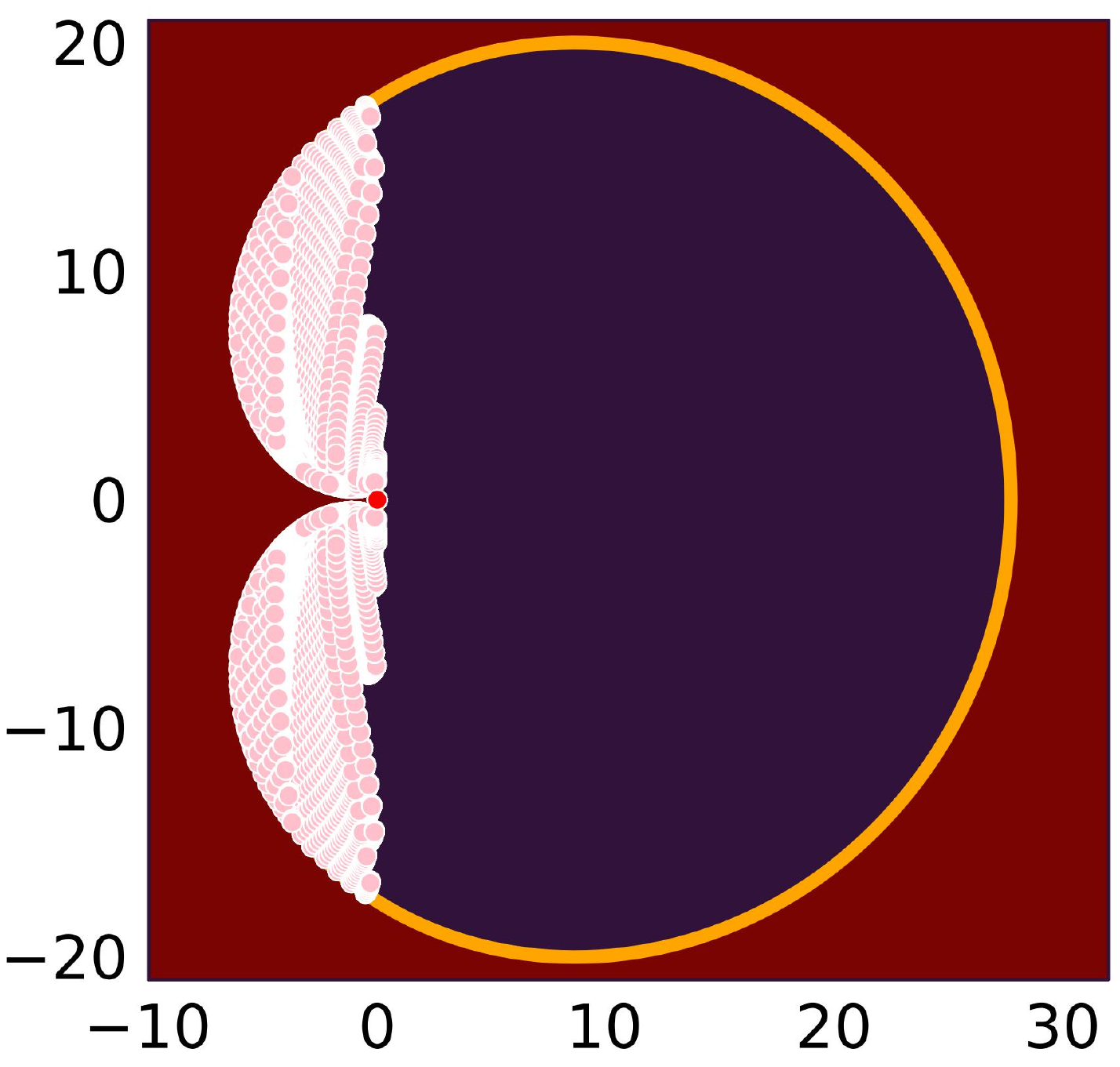} }}%
 \put(-183,67){\rotatebox{90}{\color{black}\small $\operatorname{Im}(z)$}}
 \put(-88,-7){\color{black}\small $\operatorname{Re}(z)$}
 \qquad\qquad
 \vspace*{5mm}
 \subfloat[\vspace*{-5.5mm}\centering Zoomed-in view]{{\includegraphics[width=0.364\linewidth]{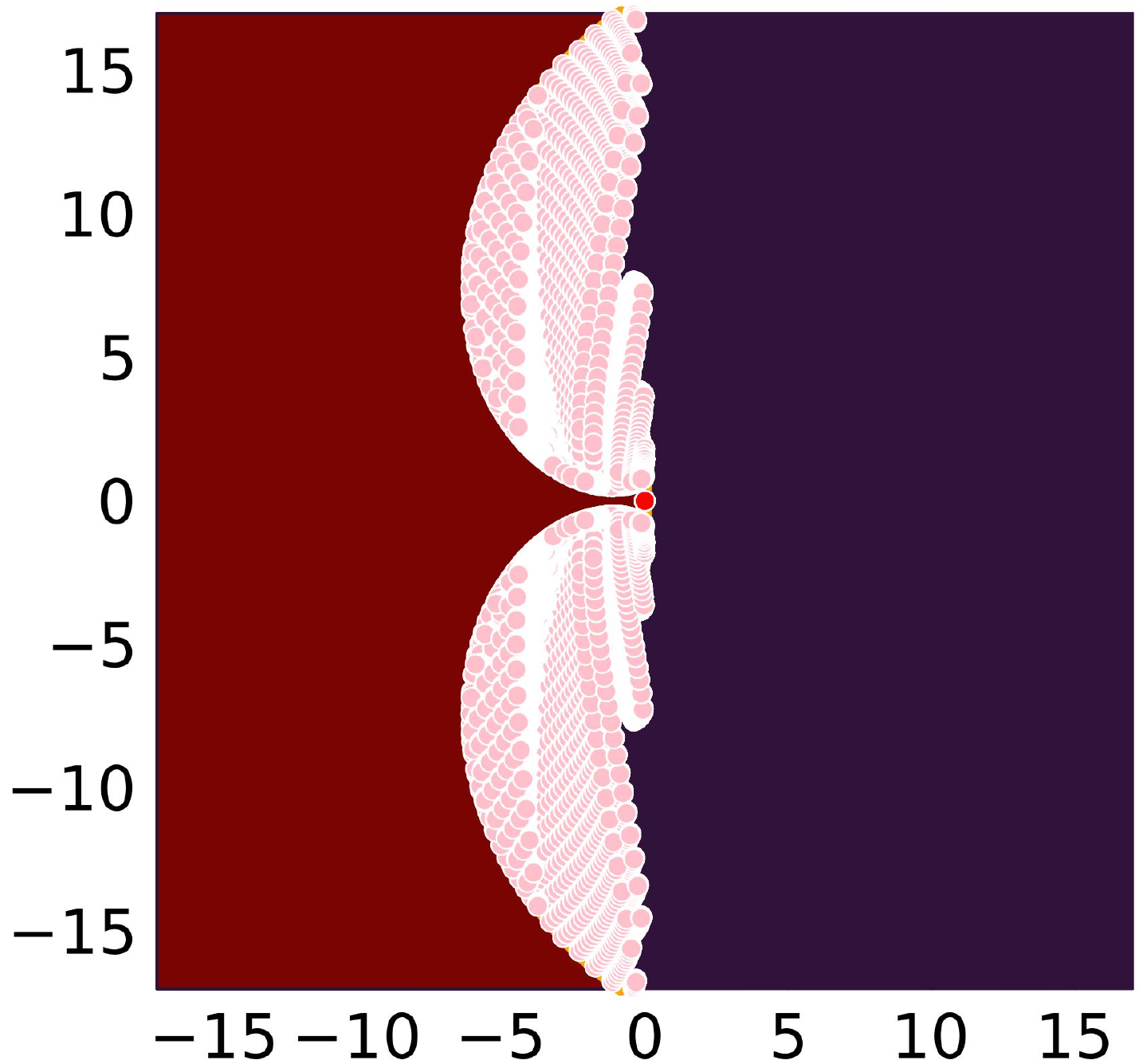} }}%
 \put(-183,67){\rotatebox{90}{\color{black}\small $\operatorname{Im}(z)$}}
 \put(-88,-7){\color{black}\small $\operatorname{Re}(z)$}

 \vspace{-1mm}

 \caption{The absolute stability region of BDF6 (the maroon-shaded region along with the orange boundary curve) and the eigenvalues $z$ of $\Delta x(A+xB)$ (the red dots with white boundary correspond to $\operatorname{Re}(z)>0$, and the pink dots with white boundary correspond to $\operatorname{Re}(z)<0$) for $\beta=2$, ${x=x_{N},-9,\dots,x_{0}}$, $\Delta x=-0.1,-0.05,-0.04,-0.02,-0.01,-0.005,-0.004$, and $M=8\times 10^3$.}%
 \label{p:bdf6_z}%
\end{figure}

When applying BDF methods to the spectral discretization, we find that each BDF method has a corresponding range for $\Delta x$ that causes instability. Moreover, once the value of $\Delta x$ exits this range, convergence appears to occur at the expected rate. As a result, when an error plot is generated with respect to the value of $\left|\Delta x\right|$, we will observe a non-monotonic pattern of convergence. Applying a $r$-step linear method with the form~\cite[Section~7.3]{10.5555/1355322}
\[
 \sum_{j=0}^{r}\alpha_{j}U^{n+j}=k\sum_{j=0}^{r}\beta_{j}f\big(U^{n+j},t_{n+j}\big)
\]
to $u'=\lambda u$ gives
\[
 \sum_{j=0}^{r}\alpha_{j}U^{n+j}=k\sum_{j=0}^{r}\beta_{j}\lambda U^{n+j}\ \Longrightarrow \  \sum_{j=0}^{r}(\alpha_{j}-z\beta_{j})U^{n+j}=0,
\]
where $z\equiv k\lambda$. We call $\sum_{j=0}^{r}(\alpha_{j}-z\beta_{j})\xi^{j}$ the stability polynomial of this method and denote it by $\pi\left(\xi;z\right)$~\cite[Section 7.3]{10.5555/1355322}. The absolute stability region of a general $r$-step method consists of the values of $z$ such that the roots $\xi_{i}(z)$, $i=1,2,\dots,r$ of $\pi(\xi;z)$ satisfy the following two conditions:
\begin{enumerate}\itemsep=0pt
 \item[(1)] $ |\xi_{i} |\leq 1$ for $i=1,2,\dots,r$,

 \item[(2)] If $\xi_{i}$ is a repeated root, then $|\xi_{i}|<1$.
\end{enumerate}
See, for example,~\cite[Definition~6.2]{10.5555/1355322} for the definition. To analyze the roots of $\pi(\xi;z)$ for each BDF method as a function of $x$ and $\Delta x$, suppose the eigenvalues of $(A+xB)$ are given by $\lambda_{1},\lambda_{2},\dots,\lambda_{2M+1}$ and define
\begin{gather*}
 \mu_{l}(x,\Delta x):=\max_j\left|\frac {\max_{i}|\xi_{i}(\Delta x\lambda_{j})|-1}{\Delta x}\right| \qquad \text{for} \ \operatorname{Re}(\Delta x \lambda_{j})<0,
\\
 \mu_{r}(x,\Delta x):=\max_j\left|\frac {\max_{i}|\xi_{i}(\Delta x\lambda_{j})|-1}{\Delta x}\right|\qquad \text{for} \ \operatorname{Re}(\Delta x \lambda_{j})>0.
\end{gather*}
Similarly, define
\begin{gather*}
 \delta_{l}(x,\Delta x):=\frac {1}{2M+1}\#\big\{j\colon\max_{i}|\xi_{i}(\Delta x \lambda_{j})|>1,\,\operatorname{Re}(\Delta x \lambda_{j})<0\big\},
\\
 \delta_{r}(x,\Delta x):=\frac {1}{2M+1}\#\big\{j\colon \max_{i}|\xi_{i}(\Delta x \lambda_{j})|>1,\,\operatorname{Re}(\Delta x \lambda_{j})>0\big\},
\end{gather*}
where $A+xB$ is the coefficient matrix of the spectral discretization and $\#$ gives the cardinality of the set. The function $\mu_{l}(x,\Delta x)\left(\mu_{r}(x,\Delta x)\right)$ helps capture the maximum growth rate instabilities in the numerical method caused by the eigenvalues of $\Delta x(A+xB)$ in the left(right)-half plane and $\delta_{l}(x,\Delta x)\left(\delta_{r}(x,\Delta x)\right)$ gives the fraction of eigenvalues of $\Delta x(A+xB)$ in the left(right)-half plane that give rise to a positive growth rate.

We then compute the mean value $\langle \mu_{l}(x,\Delta x)\rangle$, $\langle \mu_{r}(x,\Delta x)\rangle$, $\langle \delta_{l}(x,\Delta x)\rangle$, and $\langle \delta_{r}(x,\Delta x)\rangle$ respectively over $x=-10,-9,\dots,\lfloor 13/\sqrt{\beta}\rfloor$:
\begin{gather}
 \langle \mu_{l} (x,\Delta x )\rangle:=\frac {1}{10+\lfloor13/\sqrt{\beta}\rfloor}\sum_{x=-10}^{\lfloor13/\sqrt{\beta}\rfloor}\mu_{l} (x,\Delta x ),
 \label{eqn:mul}
\\
 \langle \mu_{r} (x,\Delta x )\rangle:=\frac {1}{10+\lfloor13/\sqrt{\beta}\rfloor}\sum_{x=-10}^{\lfloor13/\sqrt{\beta}\rfloor}\mu_{r} (x,\Delta x ),
 \label{eqn:mur}
\\
 \langle \delta_{l} (x,\Delta x )\rangle:=\frac {1}{10+\lfloor13/\sqrt{\beta}\rfloor}\sum_{x=-10}^{\lfloor13/\sqrt{\beta}\rfloor}\delta_{l} (x,\Delta x ),
 \label{eqn:deltal}
\\
\langle \delta_{r}\left(x,\Delta x\right)\rangle:=\frac {1}{10+\lfloor13/\sqrt{\beta}\rfloor}\sum_{x=-10}^{\lfloor13/\sqrt{\beta}\rfloor}\delta_{r} (x,\Delta x).
 \label{eqn:deltar}
\end{gather}
Figures~\ref{p:bdf3_non}--\ref{p:bdf6_z} show the non-monotonic convergence for each BDF method along with the mean fraction $\langle r\left(x,\Delta x\right)\rangle$ of the unstable eigenvalues of $\Delta x(A+xB)$, the value of $\langle \mu(x,\Delta x)\rangle$, and the eigenvalues of $\Delta x(A+xB)$ within the unstable range for $\Delta x$. On one hand, it is not the unstable eigenvalues in the right-half plane causing the non-monotonic convergence. On the other hand, the range for $\Delta x$ that causes instability coincides with the range for $\Delta x$ whose values of $\langle\mu(x,\Delta x)\rangle$ in the left-half plane are larger than $\exp(1)$ and have the largest magnitudes. Moreover, it also coincides with the range for $\Delta x$ that contains the largest mean fraction $\langle r(x,\Delta x)\rangle$ of the unstable eigenvalues from the left-half plane. Note that for BDF6, though the error at $x=-0.1$ does not blow up, it still suggests that $x=-0.1$ causes instability. For BDF3, the range for $\Delta x$ that causes instability is approximately $[-0.004,-0.002]$. For BDF4, the range is approximately $[-0.01,-0.002]$. For BDF5, the range is approximately $[-0.02,-0.004]$. For BDF6, the range is approximately $[-0.1,-0.004]$.

\section[The case when beta=infty]{The case when $\boldsymbol{\beta=\infty}$}\label{a:1}

We consider the limiting behavior of $F_{\beta}$ as $\beta\to \infty$. If we let $\beta=\infty$, the original boundary value problem becomes
\begin{equation}
\frac {\partial F}{\partial x}+\big(x-\omega^2\big)\frac {\partial F}{\partial \omega}=0\qquad \text{for}\ (x,\omega)\in \mathbb {R}^{2},
\label{eqn:inf}
\end{equation}
with the following boundary conditions:
\begin{gather}
F(x,\omega)\to 1\qquad \text{as} \ x,\omega\to \infty \ \text{together,}\label{bcs1}\\
F(x,\omega)\to 0\qquad \text{as} \ \omega\to -\infty \ \text{with}\ x \ \text{bounded above.}\label{bcs2}
\end{gather}
Using the method of characteristics~\cite{Kevorkian1990}, rewrite~\eqref{eqn:inf} as
\[
 \frac {{\rm d} F}{{\rm d} x}=0,
\]
where
\[
 \frac {{\rm d} F}{{\rm d} x}=\frac {\partial F}{\partial x}+\frac {{\rm d} \omega}{{\rm d} x}\frac {\partial F}{\partial \omega}
\]
along curves for which
\begin{equation}
 \frac {{\rm d} \omega}{{\rm d} x}=x-\omega^2.
 \label{eqn:air}
\end{equation}
Substituting $\omega=u'/u$ in~\eqref{eqn:air} yields the Airy equation
\[
 u''-xu=0,
\]
whose solution is given by
\[
u=c_{1}\operatorname{Ai}(x)+c_{2}\operatorname{Bi}(x),
\]
a linear combination of $\operatorname{Ai}(x)$, the Airy function of the first kind, and $\operatorname{Bi}(x)$, the Airy function of the second kind~\cite[Chapter~9]{NIST:DLMF}. Thus, the solution of~\eqref{eqn:air} is given by
\begin{equation}
 \omega=\frac {u'}{u}=\frac {c_{1}\operatorname{Ai}'(x)+c_{2}\operatorname{Bi}'(x)}{c_{1}\operatorname{Ai}(x)+c_{2}\operatorname{Bi}(x)},
 \label{eqn:characteris}
\end{equation}
which are the characteristic curves of $F(x,\omega)$.
\begin{Claim}
 There exists a unique characteristic curve for $F(x,\omega)$ along which~$\omega\to -\infty$ as ${x\to \infty}$.
\end{Claim}
\begin{proof}
 For existence, consider the characteristic curves
 \begin{equation}
 \tilde{\omega}(x)=\frac {\operatorname{Ai}'(x)}{\operatorname{Ai}(x)},
 \label{eqn:cch}
 \end{equation}
 which is obtained from~\eqref{eqn:characteris} by setting $c_{2}=0$ and $c_{1}\neq 0$. By~\cite[Section 9.7\,(ii)]{NIST:DLMF},
 \begin{equation}
 \lim_{x\to \infty}\tilde{\omega}(x)=\lim_{x\to \infty}\frac {\operatorname{Ai}'(x)}{\operatorname{Ai}(x)}=-\infty.
 \label{eqn:limbehavior2}
 \end{equation}
 Therefore, among the characteristic curves represented by~\eqref{eqn:cch}, there exists one along which $\omega\to -\infty$ as $x\to \infty$.

 For uniqueness, note that for fixed values of $c_{1}$ and $c_{2}\neq 0$, as a well-defined function of $x$, $\omega\to \infty$ as $x\to \infty$. Thus, combined with~\eqref{eqn:limbehavior2}, we conclude that only one such characteristic curve exists.
\end{proof}
After the change of variable $\omega=-\cot \theta$, we obtain
\begin{equation}
 \cot \theta=-\frac {c_{1}\operatorname{Ai}'(x)+c_{2}\operatorname{Bi}'(x)}{c_{1}\operatorname{Ai}(x)+c_{2}\operatorname{Bi}(x)},
 \label{eqn:xi}
\end{equation}
and the boundary conditions become
\begin{align*}
&F(x,\theta)\to 1\qquad \text{as} \ x\to \infty \ \text{and} \ \theta\to \pi,\\
&F(x,\theta)\to 0\qquad \text{as} \ \theta\to 0 \ \text{with} \ x \ \text{bounded above.}
\end{align*}

\begin{figure}[t]
 \centering
 \subfloat {{\includegraphics[width=0.45\linewidth]{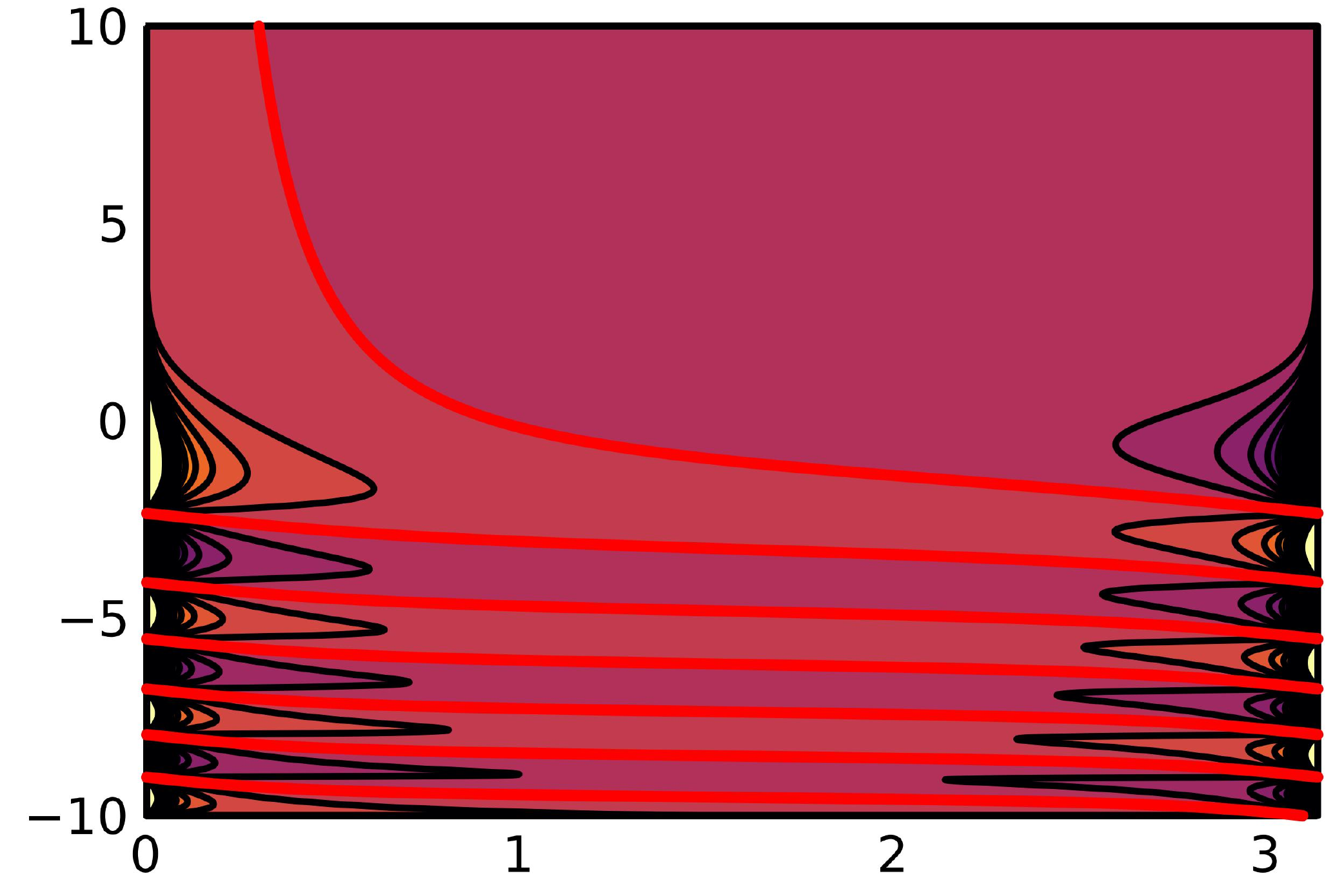} }}%
 \put(-100,-8){\color{black} $\theta$}
 \put(-210,70){\rotatebox{90}{\color{black} $x$}}
 \caption{Contour plot of~\eqref{eqn:xi} with respect to $x$ and $\theta$ for $x\in [-10,10]$ and $\theta\in [0,\pi]$. The red curves correspond to the level zero. The values of $x$ where the red curves cross $\theta = \pi$ coincide with the zeros of $\mathrm{Ai}(x)$.}
 \label{p:pin}
\end{figure}

Figure~\ref{p:pin} shows the contour plot of~\eqref{eqn:xi} with respect to $x$ and $\theta$ for $x\in[-10,10]$ and ${\theta\in[0,\pi]}$. The red curves correspond to the level zero, which occurs when $c_{1}\neq 0$ and $c_{2}=0.$ We can see that there is one and only one red curve along which $\theta\to 0$ as $x\to \infty$. Denote the first real zero (closest to $x=0$) of Ai$(x)$ by $x_{1}\approx -2.33811$. It can be verified from~\eqref{eqn:xi} by setting $c_{2}=0$ and $c_{1}\neq 0$ that $x\to x^{+}_{1}$ as $\theta\to \pi^{-}$. Set $F(x,\omega)$ in the following way:
\begin{equation}
 F(x,\omega)=\begin{dcases}
0, & \omega<\dfrac {\operatorname{Ai}'(x)}{\operatorname{Ai}(x)},\\
1, & \omega>\dfrac {\operatorname{Ai}'(x)}{\operatorname{Ai}(x)}.
 \end{dcases}
 \label{eqn:sol}
\end{equation}
Then~\eqref{eqn:sol} is the solution of~\eqref{eqn:inf} satisfying both boundary conditions~\eqref{bcs1} and~\eqref{bcs2}.
\label{apx:1}

\subsection*{Acknowledgements}

This work is partially supported by NSFDMS-1945652. The authors would like to thank the anonymous referees for their helpful comments and suggestions, which have significantly contributed to the clarity of this paper.

\pdfbookmark[1]{References}{ref}
\LastPageEnding


\begin{thebibliography}{99}
\footnotesize\itemsep=0pt

\bibitem{Bloemendal2011FiniteRP}
Bloemendal A., Finite rank perturbations of random matrices and their continuum
 limits, Ph.D.~Thesis, {U}niversity of {T}oronto, Canada, 2011.

\bibitem{article7}
Bloemendal A., Vir\'ag B., Limits of spiked random matrices~{I},
 \href{https://doi.org/10.1007/s00440-012-0443-2}{\textit{Probab. Theory Related Fields}} \textbf{156} (2013), 795--825,
 \href{https://arxiv.org/abs/1011.1877}{arXiv:1011.1877}.

\bibitem{Bornemann2009OnTN}
Bornemann F., On the numerical evaluation of distributions in random matrix
 theory: a review, \textit{Markov Process. Related Fields} \textbf{16} (2010),
 803--866, \href{https://arxiv.org/abs/0904.1581}{arXiv:0904.1581}.

\bibitem{BornemannFredholm}
Bornemann F., On the numerical evaluation of {F}redholm determinants,
 \href{https://doi.org/10.1090/S0025-5718-09-02280-7}{\textit{Math. Comp.}} \textbf{79} (2010), 871--915, \href{https://arxiv.org/abs/0804.2543}{arXiv:0804.2543}.

\bibitem{Borot2012}
Borot G., Nadal C., Right tail asymptotic expansion of {T}racy--{W}idom beta
 laws, \href{https://doi.org/10.1142/S2010326312500062}{\textit{Random Matrices Theory Appl.}} \textbf{1} (2012), 1250006,
 23~pages.

\bibitem{Dumaz2013}
Dumaz L., Vir\'ag B., The right tail exponent of the {T}racy--{W}idom $\beta$
 distribution, \href{https://doi.org/10.1214/11-AIHP475}{\textit{Ann. Inst. H.~Poincar\'e Probab. Statist.}} \textbf{49}
 (2013), 915--933, \href{https://arxiv.org/abs/1102.4818}{arXiv:1102.4818}.

\bibitem{article}
Dumitriu I., Edelman A., Matrix models for beta ensembles, \href{https://doi.org/10.1063/1.1507823}{\textit{J.~Math.
 Phys.}} \textbf{43} (2002), 5830--5847, \href{https://arxiv.org/abs/math-ph/0206043}{arXiv:math-ph/0206043}.

\bibitem{article2}
Edelman A., Sutton B.D., From random matrices to stochastic operators,
 \href{https://doi.org/10.1007/s10955-006-9226-4}{\textit{J.~Stat. Phys.}} \textbf{127} (2007), 1121--1165,
 \href{https://arxiv.org/abs/math-ph/0607038}{arXiv:math-ph/0607038}.

\bibitem{Ferrari2005}
Ferrari P.L., Spohn H., A determinantal formula for the {GOE} {T}racy--{W}idom
 distribution, \href{https://doi.org/10.1088/0305-4470/38/33/L02}{\textit{J.~Phys.~A}} \textbf{38} (2005), L557--L561,
 \href{https://arxiv.org/abs/math-ph/0505012}{arXiv:math-ph/0505012}.

\bibitem{Grava2016}
Grava T., Its A., Kapaev A., Mezzadri F., On the {T}racy--{W}idom{$_\beta$}
 distribution for {$\beta=6$}, \href{https://doi.org/10.3842/SIGMA.2016.105}{\textit{SIGMA}} \textbf{12} (2016), 105,
 26~pages, \href{https://arxiv.org/abs/1607.01351}{arXiv:1607.01351}.

\bibitem{Kevorkian1990}
Kevorkian J., Partial differential equations. {A}nalytical solution techniques,
 Wadsworth \& Brooks/Cole Math. Ser., Wadsworth \& Brooks/Cole Advanced Books
 \& Software, Pacific Grove, CA, 1990.

\bibitem{10.5555/1355322}
LeVeque R.J., Finite difference methods for ordinary and partial differential
 equations. Steady-state and time-dependent problems, \href{https://doi.org/10.1137/1.9780898717839}{Society for Industrial
 and Applied Mathematics (SIAM)}, Philadelphia, PA, 2007.

\bibitem{Li2018OnTO}
Li Y., On the open question of the {T}racy--{W}idom distribution of
 $\beta$-ensemble with $\beta=6$, \href{https://arxiv.org/abs/1812.00522}{arXiv:1812.00522}.

\bibitem{Mays2021}
Mays A., Ponsaing A., Schehr G., Tracy--{W}idom distributions for the
 {G}aussian orthogonal and symplectic ensembles revisited: a skew-orthogonal
 polynomials approach, \href{https://doi.org/10.1007/s10955-020-02695-w}{\textit{J.~Stat. Phys.}} \textbf{182} (2021), 28,
 55~pages, \href{https://arxiv.org/abs/2007.14597}{arXiv:2007.14597}.

\bibitem{Meh2004}
Mehta M.L., Random matrices, \textit{Pure Appl. Math. (Amsterdam)}, Vol. 142,
 3rd ed., Elsevier/Academic Press, Amsterdam, 2004.

\bibitem{nadalc}
Nadal C., Matrices al\'eatoires et leurs applications \`a la physique
 statistique et quantique, Ph.D.~Thesis, {U}niversit\'e Paris Sud - Paris~XI,
 2011, available at \url{https://theses.hal.science/tel-00633266}.

\bibitem{Nadal2011}
Nadal C., Majumdar S.N., A simple derivation of the {T}racy--{W}idom
 distribution of the maximal eigenvalue of a {G}aussian unitary random matrix,
 \href{https://doi.org/10.1088/1742-5468/2011/04/p04001}{\textit{J.~Stat. Mech. Theory Exp.}} \textbf{2011} (2011), P04001, 29~pages,
 \href{https://arxiv.org/abs/1102.0738}{arXiv:1102.0738}.

\bibitem{NIST:DLMF}
Olver F.W.J., Olde~Daalhuis A.B., Lozier D.W., Schneider B.I., Boisvert R.F.,
 Clark C.W., Miller B.R., Saunders B.V., Cohl H.S., McClain M.A., NIST digital
 library of mathematical functions, {R}elease~1.1.8 of 2022-12-15, aviable at
 \url{http://dlmf.nist.gov/}.

\bibitem{10.2307/23072161}
Ram\'{\i}rez J.A., Rider B., Vir\'ag B., Beta ensembles, stochastic {A}iry
 spectrum, and a diffusion, \href{https://doi.org/10.1090/S0894-0347-2011-00703-0}{\textit{J.~Amer. Math. Soc.}} \textbf{24} (2011),
 919--944, \href{https://arxiv.org/abs/math.PR/0607331}{arXiv:math.PR/0607331}.

\bibitem{Rumanov2016}
Rumanov I., Painlev\'e representation of {T}racy--{W}idom{$_\beta$}
 distribution for {$\beta=6$}, \href{https://doi.org/10.1007/s00220-015-2487-5}{\textit{Comm. Math. Phys.}} \textbf{342} (2016),
 843--868, \href{https://arxiv.org/abs/1408.3779}{arXiv:1408.3779}.

\bibitem{article3}
Sutton B.D., The stochastic operator approach to random matrix theory, Ph.D.~Thesis, {M}assachusetts Institute of Technology, 2005.

\bibitem{Townsend2015}
Townsend A., Olver S., The automatic solution of partial differential equations
 using a global spectral method, \href{https://doi.org/10.1016/j.jcp.2015.06.031}{\textit{J.~Comput. Phys.}} \textbf{299}
 (2015), 106--123, \href{https://arxiv.org/abs/1409.2789}{arXiv:1409.2789}.

\bibitem{TRACY1993115}
Tracy C.A., Widom H., Level-spacing distributions and the {A}iry kernel,
 \href{https://doi.org/10.1016/0370-2693(93)91114-3}{\textit{Phys. Lett.~B}} \textbf{305} (1993), 115--118,
 \href{https://arxiv.org/abs/hep-th/9210074}{arXiv:hep-th/9210074}.

\bibitem{cmp/1104254495}
Tracy C.A., Widom H., Level-spacing distributions and the {A}iry kernel,
 \href{https://doi.org/10.1007/BF02100489}{\textit{Comm. Math. Phys.}} \textbf{159} (1994), 151--174,
 \href{https://arxiv.org/abs/hep-th/9211141}{arXiv:hep-th/9211141}.

\bibitem{cmp/1104286442}
Tracy C.A., Widom H., On orthogonal and symplectic matrix ensembles,
 \href{https://doi.org/10.1007/BF02099545}{\textit{Comm. Math. Phys.}} \textbf{177} (1996), 727--754,
 \href{https://arxiv.org/abs/solv-int/9509007}{arXiv:solv-int/9509007}.

\bibitem{article6}
Valk\'o B., Vir\'ag B., Continuum limits of random matrices and the {B}rownian
 carousel, \href{https://doi.org/10.1007/s00222-009-0180-z}{\textit{Invent. Math.}} \textbf{177} (2009), 463--508,
 \href{https://arxiv.org/abs/0712.2000}{arXiv:0712.2000}.

\bibitem{tracywidombeta}
Zhang Y., Trogdon T., TracyWidomBeta, 2023, available at
 \url{https://github.com/Yiting687691/TracyWidomBeta.jl}.

\end{thebibliography}
\end{document}